\documentclass[reqno, 11pt]{amsart}
\usepackage{}
\usepackage{bbm}
\usepackage{url}
\usepackage{stmaryrd}
\usepackage{mathrsfs}
\usepackage{cases}
\usepackage{amsfonts}
\usepackage{amssymb}
\usepackage{amsmath}
\usepackage{enumerate}
\usepackage{enumitem}
\usepackage{tikz}
\usepackage{extarrows}
\usepackage{mathtools}
\usepackage{citeref}
\allowdisplaybreaks[4]
\newtheorem{theorem}{Theorem}[section]
\newtheorem{lemma}[theorem]{Lemma}
\newtheorem{proposition}[theorem]{Proposition}

\theoremstyle{definition}
\newtheorem{definition}[theorem]{Definition}

\newtheorem{remark}[theorem]{Remark}
\numberwithin{equation}{section}

 \let\al=\alpha
\let\b=\beta
\let\d=\delta

\let\g=\gamma
\let\G=\Gamma
\let\fy=\infty  

\let\r=\rho
\let\s=\sigma

\let\z=\zeta
\let\ep=\epsilon
\let\la=\lambda
\let\La=\Lambda

\let\va=\varphi

\def\bbR{\mathbb{R}}

\def\scrP{\mathscr{P}}

\def\scrF{\mathscr{F}}

\def\scrL{\mathscr{L}}
\def\scrV{\mathscr{V}}

\def\calC{\mathcal {C}}

\def\calF{\mathcal {F}}

\def\calL{\mathcal {L}}
\def\calM{\mathcal {M}}

\def\calS{\mathcal {S}}

\def\rr{{\mathbb R}}
\def\rd{{{\bbR}^d}} 
 
\def\rdd{{{\rr}^{2d}}}  
\def\rddd{{{\rr}^{4d}}}
\def\zd{{{\mathbb{Z}}^d}}

\def\zdd{{{\mathbb{Z}}^{2d}}}

\def\mp{M^p}
\def\mpd{M^p(\rd)}
\def\mpq{M^{p,q}}
\def\mpqd{M^{p,q}(\rd)}
\def\mfi{M^{1}}
\def\mfid{M^{1}(\rd)}
\def\mf{M^{\infty}}
\def\mfd{M^{\infty}(\rd)}
\def\msj{M^{\infty,1}}
\def\msjd{M^{\infty,1}(\rd)}
\def\msjdd{M^{\infty,1}(\rdd)}
\def\mif{M^{1,\fy}}

\def\mpd{M^p(\rd)}

\def\wpq{W^{p,q}}

\newcommand{\be}{\begin{equation*}}
	\newcommand{\ee}{\end{equation*}}
\newcommand{\ben}{\begin{equation}}
	\newcommand{\een}{\end{equation}}
\newcommand{\bn}{\begin{enumerate}}
	\newcommand{\en}{\end{enumerate}}
\newcommand{\bs}{\backslash}
\newcommand{\lan}{\langle}
\newcommand{\ran}{\rangle}

\begin{document}
\title[Sharp estimates of the Schr\"{o}dinger type propagators on modulation spaces]
{Sharp estimates of the Schr\"{o}dinger type propagators on modulation spaces}

\address{School of Science, Jimei University, Xiamen, 361021, P.R.China}
\email{weichaoguomath@gmail.com}
\author{WEICHAO GUO}
\address{School of Mathematics and Statistics, Xiamen University of Technology, Xiamen, 361024, P.R.China} 
\email{guopingzhaomath@gmail.com}
\author{GUOPING ZHAO}


\begin{abstract}
This paper is devoted to conducting a comprehensive and self-contained study of the boundedness on modulation spaces of Fourier integral operators
arising when solving Schr\"{o}dinger type operators.
The symbols of these operators belong to the  Sj\"{o}strand class $\msj$,
and their phase functions satisfy certain regularity conditions associated with mixed modulation spaces.
Our conclusions cover two novel situations corresponding to the so-called mild and high growth of phase functions. 
These conclusions represent essential improvements and generalizations of existing results.
Our method is based on a reasonable decomposition and scaling of the symbol and phase functions, ensuring their membership in appropriate mixed modulation spaces.
In a certain sense, all conclusions of this paper are optimal.
\end{abstract}

\maketitle


\section{Introduction}
In this paper, we consider the Fourier integral operators (FIOs) defined by
\be
Tf(x)=\int_{\rdd}\sigma(x,\xi)\hat{f}(\xi)e^{2\pi i\Phi(x,\xi)}d\xi,
\ee
where the functions $\s$ and $\Phi$ are referred to as the symbol and phase function, respectively. 
When $\s e^{2\pi i\Phi}\in \calS'(\rdd)$, the corresponding FIO with symbol $\s$ and phase function $\Phi$ admits a weak formulation:
\be
\lan Tf, g\ran
=
\lan \s e^{2\pi i\Phi}, g\otimes \bar{\hat{f}}\ran
=\lan \scrF_2(\s e^{2\pi i\Phi}), g\otimes \bar{f}\ran,\ \ \ \ f,g\in \calS(\rd).
\ee

As is well established, FIOs constitute a pivotal mathematical tool for investigating the behavior of solutions to partial differential equations. 
We refer the reader to \cite{CorderoNicolaRodino2009TAMS, Hoermander1971AM, RuzhanskySugimoto2006CPDE, SeegerSoggeStein1991AM2} 
for the study of 
FIOs that arise in the solutions of hyperbolic equations, where the corresponding phase functions are positively homogeneous of degree 1 in the variable $\xi$.
Distinct from such FIOs, our focus lies on FIOs encountered in analyzing the Cauchy problem for
 Schr\"{o}dinger-type equations,
foundational works in this regard include those by
 Asada--Fujiwara \cite{AsadaFujiwara1978JJMNS} and Cordoba-Fefferman \cite{CordobaFefferman1978CPDE}.

In this paper, we consider a class of more general FIOs associated with Schr\"{o}dinger-type propagators.
Primary examples stem from so-called metaplectic operators (see \cite{CorderoNicola2008JFA, CorderoGroechenigNicolaRodino2013JMPA9}), 
pseudodifferential operators in the Kohn-Nirenberg form 
(see \cite{GroechenigHeil1999IEOT, Groechenig2006RMI, SugimotoTomita2008JFAA, CorderoNicola2010IMRNI, CorderoNicola2018IMRNI, GuoChenFanZhao2022IMRN}) and unimodular Fourier multipliers (see \cite{Benyi2007JFA, Miyachi2009PAMS, Cunanan2014JMAA, NicolaTabacoo2018JPDOA, GuoZhao2020JFA, ZhaoGuo2024NMJ}).
To characterize our phase functions precisely, we introduce the following conditions for real-valued functions $\Phi$ on $\rdd$.
\begin{itemize}
    \item $(\al,t_1,t_2)$-growth condition for second-order derivatives: Let $\ep>0$.
    For some $\al\in (-\fy,1]$, $t_1,t_2\geq 0$,
    the function $\Phi$ satisfies:    
    \begin{itemize}[label=$\circ$]
        \item $|\nabla_{x}\Phi(x,\xi)-\nabla_{x}\Phi(0,\xi)|\lesssim \lan x\ran^{1-\al}$ for all $x\in \rd$,\ 
        uniformly for $\xi\in \rd$,
        \item $\partial_{x,\xi}^{\g}\Phi\in W^{\fy,\fy}_{1\otimes v_{d+\ep,d+\ep}}(\rdd)$ for every multi-index $\g$ with
        $|\g|=2$,
        \item $\lan x\ran^{-t_1}\partial_{x,x}^{\g}\Phi\in W^{\fy,\fy}_{1\otimes v_{d+\ep,d+\ep}}(\rdd)$ and
	$\lan \xi\ran^{-t_2}\partial_{\xi,\xi}^{\g}\Phi\in W^{\fy,\fy}_{1\otimes v_{d+\ep,d+\ep}}(\rdd)$,
        for every multi-index $\g$ with
        $|\g|=2$.
    \end{itemize}
    \item Uniform separation condition ($x$-type and $y$-type):
    \begin{itemize}[label=$\circ$]
        \item $x$-type: For all $\xi\in \rd,\ |x_1-x_2|\gtrsim 1$, 
        \begin{equation*}
            |\nabla_{\xi} \Phi(x_1,\xi)-\nabla_{\xi} \Phi(x_2,\xi)|\gtrsim 1\ \ \text{uniformly}.
        \end{equation*}
        \item $\xi$-type: For all $x\in \rd,\ |\xi_1-\xi_2|\gtrsim 1$, 
        \begin{equation*}
            |\nabla_x \Phi(x,\xi_1)-\nabla_x \Phi(x,\xi_2)|\gtrsim 1\ \ \text{uniformly}.
        \end{equation*}
    \end{itemize}
\end{itemize}
Here, the notation $W^{\fy,\fy}_{1\otimes v_{d+\ep,d+\ep}}(\rdd)$ denotes a special Wiener amalgam space on $\rdd$, For further details, refer to Definition \ref{def-Wiener}.
For simplicity, we use the following notations for the second-order partial derivatives:
\begin{equation*}
    \partial_{x,\xi}^{\g}\Phi:= \partial_{x_j}\partial_{\xi_l}\Phi,
    \partial_{x,x}^{\g}\Phi:= \partial_{x_j}\partial_{x_l}\Phi,
    \partial_{\xi,\xi}^{\g}\Phi:= \partial_{\xi_j}\partial_{\xi_l}\Phi,
    \ \ \ 
    \text{for suitable multi-index}\ \g\ \text{with}\ |\g|=2.
\end{equation*}

Furthermore, $(\al,t_1,t_2)$-growth conditions are classified into the following cases.
\begin{itemize}
    \item Low growth: $\al=1$ and $t_1=t_2=0$.
    \item Mild growth: $\al\in (0,1)$ and $t_1=t_2=0$.
    \item Critical growth: $\al=t_1=t_2=0$.
    \item High growth: $\al=-\fy$, with $t_1,t_2\geq 0$ and $t_1+t_2>0$.
\end{itemize}
Here, we introduce the notation $(-\fy,t_1,t_2)$ to denote the case where the condition $|\nabla_{x}\Phi(x,\xi)-\nabla_{x}\Phi(0,\xi)|\lesssim \lan x\ran^{1-\al}$
is omitted. 
By using the embedding $W^{\fy,\fy}_{1\otimes v_{d+\ep,d+\ep}}(\rdd)\subset L^{\fy}(\rdd)$.
one can additionally verify that the $(0,0,0)$-growth condition is equivalent to the $(-\fy,0,0)$-growth condition.

In this paper, we consider the FIOs 
with Sj\"{o}strand symbols $\s$, where the phase functions $\Phi$
satisfy a specific $(\al,t_1,t_2)$-growth condition and a uniform separation condition.
We note that only the low-growth and critical-growth cases have been previously studied, 
while the mild-growth and high-growth cases represent entirely new contributions to the investigation of FIOs.

We focus on the boundedness of these operators acting on the so-called modulation spaces $M^{p,q}$ with $1\leq p,q\leq \fy$,  which are widely utilized in time-frequency analysis.
Modulation spaces were introduced by H. Feichtinger \cite{Feichtinger1983TRUoV} in 1983.
Toady, they are widely recognized as the ``natural'' function spaces
for time-frequency analysis \cite{GrochenigBook2013, Feichtinger2006STisaIP}.
More precisely, for $1\leq p,q\leq \fy$ and a weight function $m$ on $\rdd$,
modulation spaces $M^{p,q}_m(\rd)$
are defined by measuring the decay and integrability of the STFT (defined in Subsection 2.1) as follows:
\be
M^{p,q}_m(\rd)=\{f\in \calS'(\rd): V_gf\in L^{p,q}_m(\rdd) \},
\ee
endowed with the canonical norm. 
Here, $L^{p,q}_m(\rdd)$ denotes weighted mixed-norm Lebesgue spaces with weight $m$,
see Section 2 for more details. We use $\calM^{p,q}_m(\rd)$ to denote the closure of $\calS(\rd)$ in $M^{p,q}_m(\rd)$.

In prior work by Cordero-Nicola-Rodino \cite[Theorem 5.2]{CorderoNicolaRodino2009CoP&AA} and Cordero-Nicola \cite[Theorem 1.1]{CorderoNicola2010JFAA}, 
the $\mpq$-boundedness of FIOs associated with \textbf{low-growth phase functions} has been established. 
For FIOs with \textbf{critical-growth phase functions}, boundedness results are derived in
\cite[Theorem 1.2 and Remark 6.5]{CorderoNicola2010JFAA}; notably,
the sharpness of the threshold conditions for the regularity exponents $s_1$ and $s_2$ remains an open problem to date.
See also \cite{ElenaAnitaPatrik2013JotLMS} for some more general boundedness in the framework of modulation spaces, where the exponents
are closely related to the boundedness of pseudodifferential operators.
See \cite{ConcettiToft2009AfM, ToftConcettiGarello2010OJM} for a study of a more general type of FIO.

Note that in previous articles (see, e.g., \cite{CorderoNicolaRodino2009CoP&AA, CorderoNicola2010JFAA,ElenaAnitaPatrik2013JotLMS}), the phase function is assumed 
to be a so-called ``tame function'', which is a real-valued smooth funciton on $\rdd$, satisfying
$\partial^{\g}\Phi\in L^{\fy}$ for all $|\g|\geq 2$, and the non-degeneracy condition:
\begin{equation*}
    \left\lvert \det \left( \left. \frac{\partial^2 \Phi}{\partial x_j \partial \xi_l} \right|_{(x, \xi)} \right) \right\rvert \geq \delta > 0, \quad \forall (x, \xi) \in \mathbb{R}^{2d}.
\end{equation*}
Compared with our assumptions, the tame function actually belongs to a stronger class of functions,
satisfying both the critical growth condition and the uniform separation condition. 
More precisely, it can be proven that a tame function satisfies the $(0,0,0)$-growth condition and the uniform separation conditions,
whereas the converse does not hold.

Based on the above analysis of existing research, \textbf{our first motivation} is to solve the open problems left unresolved in \cite{CorderoNicola2010JFAA} and, at the same time, to investigate the boundedness of FIOs in the intermediate state between low-growth and critical-growth.

For $s,t\in \rr$, we define the power weights
\be
v_{s,t}(z)=\langle z_1\rangle^s\langle z_2\rangle^t,
\ \ z=(z_1,z_2)\in \rdd.
\ee
Our first result characterizes all exponents $p,q\in [1,\fy],$ $s_1,s_2\in \rr$ and $\alpha\in [0,1]$
of function spaces such that: for all symbols in $M^{\fy,1}_{v_{s_1,s_2}\otimes 1}(\rdd)$ and phase functions 
satisfying the $(\alpha, 0,0)$-growth condition and uniform separation condition, 
the corresponding operator $T_{\s,\Phi}$ is bounded on $\mpqd$. 
To avoid the fact that $\calS(\rd)$ in not dense in the endpoint spaces $M^{p,q}$ with $p=\fy$ or $q=\fy$, we only 
consider the action of FIOs on $\calS(\rd)$.
Once the corresponding boundedness is established, FIOs can be uniquely extended to bounded linear operators from 
$\calM^{p,q}$ into $M^{p,q}$.

\begin{theorem}[FIO on $\mpq$]\label{thm-FIO-Mpq}
	Let $1\leq p,q \leq \fy$, $s_1,s_2\in \rr$ and $\alpha\in [0,1]$.
	The following statements are equivalent:
	\bn
	\item
    For all symbols $\s\in M^{\fy,1}_{v_{s_1,s_2}\otimes 1}(\rdd)$ and all real-valued $C^2(\rdd)$ phase functions $\Phi$ satisfying 
    the $(\alpha,0,0)$-growth condition and uniform separation condition, the boundedness property holds:
    $$T_{\s,\Phi}\in \calL(\mpqd).$$
	\item 
        The exponents $s_1, s_2$ satisfy $s_1, s_2\geq 0$.
        In addition, for $\al \in [0,1)$,
	the following embedding relations are satisfied:
	\be
	l^{p}_{\frac{s_1}{1-\al}}(\zd)\subset l^q(\zd),\ \ \   l^q_{\frac{s_1}{1-\al}+s_2}(\zd)\subset l^p(\zd).
	\ee
    \item 
    The exponents $p,q,s_1,s_2,\al$ satisfy $s_1,s_2\geq 0$, and for $\al \in [0,1)$,
    the following conditions hold:
    \begin{itemize}[label=$\circ$]
        \item If $1/q>1/p$: $s_1>d(1-\al)(1/q-1/p)$;
        \item If $1/p>1/q$: $s_1+(1-\al)s_2>d(1-\al)(1/p-1/q)$.
    \end{itemize}
	\en
	Moreover, if any of the above statements holds,  the boundedness  
    in Statement (1) satisfies the following norm estimate:
	\be
	\|T_{\s,\Phi}\|_{\calL(\mpqd)}
	\lesssim 
	\|\s\|_{M^{\fy,1}_{v_{s_1,s_2}\otimes 1}}e^{\sum_{|\g|=2}\|\partial^{\g}\Phi\|_{W^{\fy,\fy}_{1\otimes v_{d+\ep,d+\ep}}}}.
	\ee
\end{theorem}

Remarks on Theorem \ref{thm-FIO-Mpq}:
\bn
\item 
By assuming a weaker condition, namely that the phase function satisfies $(\al,0,0)$-growth condition and uniform separation condition, Theorem \ref{thm-FIO-Mpq} obtains a stronger boundedness conclusion compared to assuming a tame phase function. 
This type of assumption is inspired by the literatures \cite{CorderoNicolaRodino2009CoP&AA, NicolaTabacoo2018JPDOA,GuoZhao2020JFA}.
\item
If we set $\al=1$ in Theorem \ref{thm-FIO-Mpq}, we recover the conclusion for the low-growth case.
This refines \cite[Theorem 1.1]{CorderoNicola2010JFAA}, 
where phase functions were required to satisfy the “tame function” property.
\item 
If we set $\al=0$ in Theorem \ref{thm-FIO-Mpq}, we obtain the sharp range of exponents for the boundedness for the critical-growth case. 
This resolves the open problem of characterizing exponents left unresolved in \cite[Theorem 1.2]{CorderoNicola2010JFAA}.
\item 
If we take $\al\in (0,1)$ in Theorem \ref{thm-FIO-Mpq}, the result is entirely new and corresponds to an intermediate case between the low-growth and critical-growth cases.
We also note that this conclusion provides a refinement and extension of the main theorem in \cite{GuoZhao2020JFA}.
See subsection 3.5 for further details. 
\item 
In comparison with the proof methods in existing literature, our innovations are as follows:
\begin{itemize}[label=$\circ$]
    \item
    \textbf{Estimate at the level of function spaces.}
    We transform the boundedness of FIOs into mixed modulation space estimates for their symbols and phase functions (inspired by the Kernel theorems), rather than mixed Lebesgue space estimates for their short-time Fourier transforms (inspired by Schur-type tests, see, e.g., \cite{CorderoNicolaRodino2009CoP&AA, CorderoNicola2010JFAA}). This new perspective enables us to handle the estimates at the level of function spaces and also allows us to weaken the conditions on the phase function to the so-called $(\al,0,0)$-growth condition and uniform separation condition.
    \item 
    \textbf{Independence of the symbol $\s$.}
    Benefiting from considering the estimate at the level of mixed modulation spaces, we find that, in a certain sense, the Sj\"{o}strand symbol $\s$ has no influence on the boundedness of FIOs (see Propositions \ref{pp-KFIO-Mi-Mf},\ref{pp-KFIO-Mfi} and \ref{pp-KFIO-Mif}). 
    As a result, when studying the boundedness of FIOs, the symbol can be ignored, and only the quantities related to the phase function need to be estimated.
    \item 
    \textbf{Separation of the high-order part of phase function $\Phi$.}
    For the estimate of phase function, a crucial step is to separate its high-order (second-order and above) derivative components, allowing the low-order (zeroth- and first-order) derivatives to reflect their roles in boundedness. 
    Leveraging the advantages of mixed modulation space estimates, we achieve this separation through a priori decomposition of the phase function and using product inequalities in mixed-modulation spaces. 
\end{itemize}
\en

In many boundedness theorems for FIOs (i.e., Theorem \ref{thm-FIO-Mpq} above), the uniform separation condition of phase function is of crucial importance.
This significance is manifested not only in the proofs of these theorems but also in the fact that phase function $\Phi(x,\xi)=x\cdot \xi$
of pseudodifferential operators, which serve as the canonical examples of FIOs,  also satisfy the uniform separation property. 
A natural question arises: if a phase function \textit{lacks} this uniform separation property, how does the boundedness of the corresponding FIO behave? 
Establishing the boundedness of such FIOs (which lie beyond the pseudodifferential operator framework) is a meaningful task.

Based on the above argument, \textbf{our second motivation} is to investigate the boundedness of FIOs without assuming the uniform separation property of the phase function. 
The following theorem characterizes all exponents $p,q\in [1,\fy],$ $s_1,s_2\in \rr$ and $\alpha\in [0,1]$
of function spaces such that: for all symbols in $M^{\fy,1}_{v_{s_1,s_2}\otimes 1}(\rdd)$ and phase functions 
satisfy the $(\alpha,0,0)$-growth condition, the corresponding operator $T_{\s,\Phi}$ is boundedness on $\mpqd$.

\begin{theorem}[FIO on $\mpq$, non-separated phase]\label{thm-FIO-Mpq-nsp}
	Let $1\leq p,q \leq \fy$, $s_1,s_2\in \rr$ and $\alpha\in [0,1]$.
	The following statements are equivalent:
	\bn
\item
For all symbols $\s\in M^{\fy,1}_{v_{s_1,s_2}\otimes 1}(\rdd)$ and all real-valued $C^2(\rdd)$ phase functions $\Phi$ satisfying 
    the $(\alpha,0,0)$-growth condition, the boundedness property holds:
    $$T_{\s,\Phi}\in \calL(\mpqd).$$
\item The following embedding relations are satisfied:
	\be
	l^{\fy}_{\frac{s_1}{1-\al}-\frac{\al d}{(1-\al)p}}\subset l^q\ \text{for}\ \al\in [0,1),
	\ \ \ \ \ 	
l^{\fy}_{s_1}\subset l^p,\ \ \ l^q_{s_2}\subset l^1.
\ee

\item
    The exponents $p,q,s_1,s_2,\al$ satisfy 
    \begin{itemize}[label=$\circ$]
        \item $s_1\geq d/p$ with strict inequality for $p<\fy$,
        \item $s_2\geq d(1-1/q)$ with strict inequality for $q>1$,
    \end{itemize}
    and the following conditions for $\al \in [0,1)$
    \be
    s_1\geq \al d/p+(1-\al)d/q\ \ \text{with strict inequality for}\ q<\fy.
    \ee
	\en
Moreover, if any of the above statements holds,  the boundedness in statement (1) satisfies
 the following norm estimate
\be
\|T_{\s,\Phi}\|_{\calL(\mpqd)}
\lesssim 
\|\s\|_{M^{\fy,1}_{v_{s_1,s_2}\otimes 1}}e^{\sum_{|\g|=2}\|\partial^{\g}\Phi\|_{W^{\fy,\fy}_{1\otimes v_{d+\ep,d+\ep}}}}.
\ee
\end{theorem}

Remarks on Theorem \ref{thm-FIO-Mpq-nsp}:
\bn
\item Theorem \ref{thm-FIO-Mpq-nsp} is the boundedness result corresponding to Theorem \ref{thm-FIO-Mpq},
with the uniform separation property of the phase function removed.
\item By comparing Theorem \ref{thm-FIO-Mpq} and Theorem \ref{thm-FIO-Mpq-nsp}, one observes that: in the absence
    of uniform separation condition on $\Phi$, 
    stronger conditions on $\s\in M^{\fy,1}_{v_{s_1,s_2}\otimes 1}$ are required to ensure the boundedness
    of $T_{\s,\Phi}$ on $\mpq$, except for the case $(p,q)=(\fy,1)$.
    This further illustrates the pivotal role of the uniform separation condition of the phase function $\Phi$
    in guaranteeing the boundedness of $T_{\s,\Phi}$ on $\mpq$.
\item The proof of boundedness in Theorem \ref{thm-FIO-Mpq-nsp} follows the same framework as that of Theorem \ref{thm-FIO-Mpq}, 
 including the decomposition of the symbol and the phase function, and the estimate in the framework of mixed modulation spaces.
 \item The optimality of the exponents for boundedness is illustrated by several examples. These examples go beyond the framework of Fourier multipliers and will be presented in Subsection 4.2.
\en

Finally, we turn our attention to the growth condition of phase functions. Whether the growth conditions considered in Theorems \ref{thm-FIO-Mpq} and \ref{thm-FIO-Mpq-nsp} (discussed above) or the tame function condition in many prior studies of FIOs (see \cite{CorderoNicolaRodino2009CoP&AA, CorderoNicola2010JFAA,ElenaAnitaPatrik2013JotLMS,CorderoGroechenigNicolaRodino2013JMPA9,FernandezGalbisPrimo2019TAMS}), 
the non-high-growth condition (bounded derivative condition of order two and more) is critical for establishing boundedness results of FIOs and their proofs. 
However, as basic examples illustrate, the assumptions of many unimodular multipliers exceed this scope (see, e.g., \cite{Miyachi2009PAMS, NicolaTabacoo2018JPDOA, GuoZhao2020JFA}).

Based on the above argument, \textbf{our third motivation} is to investigate the boundedness of FIOs beyond the
scope of non-high-growth condition. 
Precisely, our third conclusion is established under the framework of high-growth phase function and the $L^2$ boundedness of FIOs.
It characterizes the optimal exponents of $p\in [1,\fy]$, $s_1,s_2\in \rr$ and $t_1,t_2\geq 0$
of function spaces such that: for all symbols in $M^{\fy,1}_{v_{s_1,s_2}\otimes 1}(\rdd)$ and phase functions 
satisfy the $(-\fy, t_1,t_2)$-growth condition, the corresponding operator $T_{\s,\Phi}$ is boundedness on $\mpd$.

\begin{theorem}[FIO on $\mp$, high growth]\label{thm-FIO-Mp-12hg}
	Let $1\leq p\leq \fy$, $s_1,s_2\in \rr$, $t_1,t_2\geq 0$.
	The following statements are equivalent:
	\bn
	\item
	For all symbols $\s\in M^{\fy,1}_{v_{s_1,s_2}\otimes 1}$ and all real-valued $C^2(\rdd)$ functions $\Phi$ 
    satisfying 
    the $(-\fy,t_1,t_2)$-growth condition and uniform separation condition,
    if $T_{v_{s_1,s_2}\s,\Phi}\in \calL(L^2)$ with 
    \be
    \|T_{v_{s_1,s_2}\s,\Phi}\|_{\calL(L^2)}\lesssim \|\s\|_{M^{\fy,1}_{v_{s_1,s_2}\otimes 1}},
    \ee
	the following boundedness property holds:  
\be
T_{\s,\Phi}\in \calL(\mp)\cap \calL(M^{p'}).
\ee
	\item 
	The following conditions hold:
	\be
	s_1\geq dt_1|1/p-1/2|,\ \ \ s_2\geq dt_2|1/p-1/2|.
	\ee
	\en
	Moreover, if any of the above statements holds, the boundedness in statement (1) satisfies
 the following norm estimate:
	\be
	\|T_{\s,\Phi}\|_{\calL(\mpd)}+\|T_{\s,\Phi}\|_{\calL(M^{p'}(\rd))}
	\lesssim 
    \|\s\|_{\msj_{v_{s_1,s_2}\otimes 1}(\rdd)}
     e^{(A+B+C)|2/p-1|},
	\ee
    where 
    	\be
	A=\sum_{|\g|=2}\|\lan x\ran^{-s_1}\partial^{\g}_{x,x}\Phi\|_{W^{\fy,\fy}_{1\otimes v_{d+\ep,d+\ep}}},\ \ \ 
	B=\sum_{|\g|=2}\|\lan \xi\ran^{-s_2}\partial^{\g}_{\xi,\xi}\Phi\|_{W^{\fy,\fy}_{1\otimes v_{d+\ep,d+\ep}}},
	\ee
	and
	\be
	C=\sum_{|\g|=2}\|\partial^{\g}_{x,\xi}\Phi\|_{W^{\fy,\fy}_{1\otimes v_{d+\ep,d+\ep}}}.
	\ee
\end{theorem}

Remarks on Theorem \ref{thm-FIO-Mp-12hg}:
\bn
\item The conclusion of this theorem is inspired by its typical example, namely the boundedness of unimodular multipliers $e^{i|D|^{\b}}\ (\b>2)$ on modulation spaces (see \cite{Miyachi2009PAMS}). Just as in the case of non-high growth phase function, there is a correspondence between the boundedness of FIOs and the boundedness of corresponding unimodular multipliers $e^{i|D|^{\b}}\ (\b>2)$ in the modulation spaces and the Wiener amalgam spaces (see \cite{GuoZhao2020JFA}).
\item We only consider the boundedness in the modulation space $M^p$ (rather than $M^{p,q}$) because the boundedness of its typical example—unimodular multipliers with high growth $e^{i|D|^{\b}}\ (\b>2)$ in the Wiener amalgam space has not been fully resolved (see \cite{GuoZhao2020JFA}).
\item Although the conclusion of FIOs (with high-growth phase) is inspired by the corresponding result of unimodular multipliers, its proof method requires significant innovation.
\begin{itemize}[label=$\circ$]
    \item \textbf{The absence of localization property.}
    In this case of high growth, a crucial step is how to transfer the growth conditions of the second-order derivatives of the phase to the regularity loss of boundedness. To ensure that the regularity loss of boundedness is optimal, we need to localize the time domain $(x,\xi)\in \rdd$ of the phase function. 
    However, taking the boundedness of $M^1$ as an example, for the natural working space $M^{1,1,\fy,\fy}(c_1)$ derived from the Kernel theorem, we can only localize the $\xi$ variable (see Proposition \ref{pp-tlp-ws3}), while there is no possibility of localizing the $x$ variable. 
    This is completely different from the case of unimodular multipliers.
    \item \textbf{The choices of working spaces.}
    To overcome the lack of localization properties for the $x$ variable in the natural working space $M^{1,1,\fy,\fy}(c_1)$, we need to use a stronger working space $M^{1,\fy,\fy,\fy}_{1\otimes v_{d+\ep,0}}(c_1)$ to regain the localization properties of $x$ (see Proposition \ref{pp-tlp-ws2}).
    Note that $M^{1,1,\fy,\fy}(c_1)\subset M^{1,\fy,\fy,\fy}_{1\otimes v_{d+\ep,0}}(c_1)\subset W^{\fy,\fy}_{1\otimes v_{d+\ep,d+\ep}}$.
    The estimates go through the localization of the $\xi$ variable in $M^{1,1,\fy,\fy}(c_1)$ and the localization of the $x$ variable in
    $M^{1,\fy,\fy,\fy}_{1\otimes v_{d+\ep,0}}(c_1)$, and finally return to the space $W^{\fy,\fy}_{1\otimes v_{d+\ep,d+\ep}}$ where the phase function is located.
    
    \item \textbf{The interplay of scaling and localization.}
    In addition to the localization technique, another crucial approach to estimate the loss of regularity is to decompose the operator into a composite of the corresponding scaling operator and a new scaling FIO (Denoted by FIO$_{\la}$), such that the phase of the new FIO$_{\la}$ has lower growth phase function. 
    Correspondingly, in our estimation, this is reflected in performing a localization process and a scaling transformation on the two variables $x$ and $\xi$ respectively.
    At the same time, we also need to separate the high-order part of the phase function in two stages, and the selection of the timing for separation is also very important. Specifically, our processing procedure is as follows:
    \begin{itemize}
        \item The localization of variable $\xi$ in $M^{1,1,\fy,\fy}(c_1)$,
        \item The Scaling of variable $\xi$ in $\calF L^{\fy}$,
        \item For the scaled (for $\xi$) phase function, 
        we separate the part involving high-order derivatives with respect to the
        $(x,\xi)$ variables, by using $M^{1,1,\fy,\fy}(c_1) \cdot M^{1,\fy,1,\fy}(c_1)\subset M^{1,1,\fy,\fy}(c_1)$,
        \item The localization of variable $x$ in $M^{1,\fy,\fy,\fy}_{1\otimes v_{d+\ep,0}}(c_1)$,
        \item The Scaling of variable $x$ in $\calF L^1$,
        \item For the scaled (for $(x,\xi)$) phase function, we separate the part involving high-order derivatives with respect to the
        $(x,\xi)$ variables, by using the growth condition of $\Phi$.
    \end{itemize}
\end{itemize}
\en

This paper is structured as follows.
In Section 2, we introduce first some definitions and fundamental properties of the modulation spaces $M^{p,q}_m(\rd)$ and Wiener amalgam spaces 
$W^{p,q}_m(\rd)$. Then, in Subsection 2.1, we turn to the analysis of working spaces, which will be used frequently in the proof of the boundedness of FIOs. Specifically, we give the definition of a special Wiener amalgam space $W^{\fy,\fy}_{1\otimes v_{d+\ep,d+\ep}}$(severed as the working space of phase function), along with its scaling properties and localization properties. To estimate the kernel functions for the four types of boundedness ($M^1$, $M^{\fy}$, $M^{\fy,1}$ and $M^{1,\fy}$) respectively, we introduce four mixed modulation spaces, namely,  $M^{1,1,\fy,\fy}(c_i)$ with $i=1,2$ and $M^{1,\fy,1,\fy}(c_i)$ with $i=3,4$.
Additionally, we present the product inequalities related to these spaces, which are used to separate the symbol, the lower-order part, and the high-order part of the phase function. 
To handle FIOs associated with high growth phases, we introduce $M^{1,\fy,\fy,\fy}_{1\otimes v_{d+\ep,0}}(c_1)$ as another working space, and establish its localization property and that of $M^{1,1,\fy,\fy}(c_1)$ respectively.
Using the properties of the working spaces established in Subsection 2.2, in Subsection 2.3, we combine the kernel theorems of modulation spaces, present the independence of the Sj\"{o}strand symbol in the boundedness of FIOs, and shed a light for the estimation of the phase function that needs to be established for further studying the boundedness of FIOs in modulation spaces.   

Section 3 is devoted to proving Theorem \ref{thm-FIO-Mpq}.
The sufficiency part is based on the endpoint boundedness established by Theorems \ref{thm-FIO-M1}, \ref{thm-FIO-Mf}, \ref{thm-FIO-Mfi} and \ref{thm-FIO-Mif} respectively, and a standard complex interpolation of the modulation spaces. The necessity relies on the function constructed in Lemma \ref{lm-spf}.
As an application, we present a conclusion on the boundedness of unimodular multipliers on Wiener amalgam spaces $W^{p,q}(\rd)$
in Subsection 3.5.

We give the proof of Theorem \ref{thm-FIO-Mpq-nsp} in Section 4. 
When proving the sufficiency, we establish some endpoint boundedness of FIOs without the uniform separation condition of the phase,
and use interpolation and embedding to obtain the complete boundedness. 
For the proof of necessity, we construct corresponding functions based on the exponent relationships of boundedness to illustrate the optimality of the exponent range.

Finally, in Section 5, we address the boundedness of FIOs with high-growth phase function and complete the proof of Theorem \ref{thm-FIO-Mp-12hg}.
Based on complex interpolation and duality, the proof of this theorem is reduced to proving the $M^1$ boundedness, which will be given in Theorem \ref{thm-FIO-12hg-M1}. 
The optimality of the exponents relies on the construction of some specific functions.

\textbf{Notations.}
We use $A\lesssim B$ to denote the statement that $A\leq CB$, and use
$A\sim B$ to denote the statement $A\lesssim B\lesssim A$, where the implicit constants may vary with the specific context.
The notation $\mathscr{L}$ is used to denote some large positive constant which may be changed
corresponding to the specific context.

\section{Preliminaries}

\subsection{Modulation spaces}
In order to introduce modulation spaces, we first recall some notations of time-frequency representation.
We define the translation operator $T_x$ and modulation operator $M_{\xi}$ by
\be
T_xf(t)=f(t-x),\ \ \ \ M_{\xi}f(t)=e^{2\pi it\cdot\xi}f(t).
\ee
For $(f,g)\in \calS'(\rd)\times \calS(\rd)$,
the short time Fourier transform (STFT) with window function $g$ is defined by
\be
V_gf(x,\xi)=\langle f, M_{\xi}T_xg\rangle
\ee
as a map from $\calS'(\rd)\times \calS(\rd)$ into $\calS'(\rdd)$.
In fact, for $f\in \calS'(\rd)$ and $g\in \calS(\rd)$, $V_gf$ is a continuous function on $\rdd$ with polynomial growth,
see \cite[Theorem 11.2.3]{GrochenigBook2013}.
The so-called fundamental identity of time-frequency analysis is as follows:
\be
V_gf(x,\xi)=e^{-2\pi ix\cdot \xi}V_{\hat{g}}\hat{f}(\xi,-x),\ \ \ (x,\xi)\in \rdd.
\ee

As we mentioned above, modulation spaces are defined as a measure of the STFT of $f\in \calS'$ in the time-frequency plane.
In order to give a quantitative description of decay properties of the STFT,  modulation spaces
are usually defined by using appropriate weight functions.
In this paper, we limit ourselves to the class of weight functions denoted by $\scrP(\rdd)$,
consisting of weight function $m$ that are positive functions on $\rdd$ and satisfy 
$m(x+y)\lesssim v_s(x)m(y)$, $x,y\in \rdd$.
Here, we use the notation $v_s(x)=(1+|x|^2)^{\frac{s}{2}}$.
In particular, we shall consider the weight function 
$v_{s_1,s_2}(x,\xi)=\lan x\ran^{s_1}\lan \xi\ran^{s_2}$ on $\rdd$ 
or $m(z_1,z_2,\z_1,\z_2)=(v_{s_1,s_2}\otimes 1)(z_1,z_2,\z_1,\z_2)=\lan z_1\ran^{s_1}\lan z_2\ran^{s_2}$ on $\rddd$.
The weighted mixed-norm spaces used to measure the STFT are defined as follows.
\begin{definition}[Weighted mixed-norm spaces.]
Let $m\in \scrP(\rdd)$, $1\leq p,q\leq \fy$. Then the weighted mixed-norm space $L^{p,q}_m(\rdd)$
consists of all Lebesgue measurable functions on $\rdd$ such that the norm
\be
\|F\|_{L^{p,q}_m(\rdd)}=\left(\int_{\rd}\left(\int_{\rd}|F(x,\xi)|^pm(x,\xi)^pdx\right)^{q/p}d\xi\right)^{1/q}
\ee
is finite, with usual modification when $p=\fy$ or $q=\fy$.
\end{definition}
Using these weighted mixed-norms, we give the definition of modulation spaces.
\begin{definition}\label{df-M}
Let $1\leq p,q\leq \infty$.
Given a non-zero window function $\phi\in \calS(\rd)$, the (weighted) modulation space $M^{p,q}_m(\rd)$ consists
of all $f\in \calS'(\rd)$ such that the norm
\be
\begin{split}
\|f\|_{M^{p,q}_m(\rd)}&:=\|V_{\phi}f(x,\xi)\|_{L^{p,q}_m(\rdd)}
=\left(\int_{\rd}\left(\int_{\rd}|V_{\phi}f(x,\xi)m(x,\xi)|^{p} dx\right)^{{q}/{p}}d\xi\right)^{{1}/{q}}
\end{split}
\ee
is finite.
\end{definition}
We write $M^{p,q}$ for the modulation space with $m\equiv 1$. 
If $p=q$, we use the notation $M^p_m=M^{p,q}_m$ and $\mp=\mpq$. 
Recall that the above definition of $M^{p,q}_m$ is independent of the choice of window function $\phi$,
one can see \cite{GrochenigBook2013} for this fact.

We use the notation $\calM^{p,q}_m(\rd)$ for the $\calS(\rd)$ closure in $M^{p,q}_m(\rd)$.
Recall that $\calM^{p,q}_m(\rd)=M^{p,q}_m(\rd)$ for $p,q\neq \fy$. 
The duality properties for modulation spaces is as follows.
One can find the proof in \cite[Theorem 11.3.6]{GrochenigBook2013} and \cite[Lemma 2.8]{GuoChenFanZhao2022IMRN}.
\begin{lemma}[Duality of modulation spaces]\label{lm-dual}
  Let $1\leq p,q\leq \fy$, $m\in \mathscr{P}(\rdd)$, we have the duality relation:
  \ben\label{lm-dual-c1}
  (\calM^{p,q}_m(\rd))^*=M^{p',q'}_{m^{-1}}(\rd).
  \een
\end{lemma}
We also recall a complex interpolation conclusion for modulation spaces (see, e.g., \cite{FeichtingerGroechenig1989MM}).
\begin{lemma}[Complex interpolation on modulation spaces]
    Let $1\leq p_i,q_i\leq \fy$, $m_i\in \mathscr{P}(\rdd)$ for $i=1,2$.
    For $\theta\in (0,1)$, set
    \be
    \frac{1}{p}=\frac{1-\theta}{p_1}+\frac{\theta}{p_2},\ \ \ \ \ \frac{1-\theta}{p}=\frac{1}{q_1}+\frac{\theta}{q_1},\ \ \ \ \ m=m_1^{1-\theta}m_2^{\theta},
    \ee
    then
    \be
    [M^{p_1,q_1}_{m_1}(\rd),\ M^{p_2,q_2}_{m_1}(\rd)]_{\theta}=M^{p,q}_{m}(\rd).
    \ee
\end{lemma}

\begin{definition}\label{Definition, Wiener amalgam space, continuous form}
Let $1\leq p, q\leq \infty$.
Given a non-zero window function $\phi\in \calS(\rd)$, the (weighted) Wiener amalgam space $W^{p,q}_m$ consists
of all $f\in \calS'(\rd)$ such that the norm
\begin{equation}
\begin{split}
\|f\|_{W^{p,q}_{m}}&=\big\|\|V_{\phi}f(x,\xi)m(x,\xi)\|_{L^q_{\xi}}\big\|_{L^p_{x}}
\\&
=\left(\int_{\mathbb{R}^n}\left(\int_{\mathbb{R}^n}|V_{\phi}f(x,\xi)|^{q}m(x,\xi)^{q}d\xi\right)^{{p}/{q}}dx\right)^{{1}/{p}}
\end{split}
\end{equation}
is finite, with the usual modifications when $p=\infty$ or $q=\infty$.
\end{definition}

In order to obtain the sharp exponents of embedding relations mentioned in our main theorems,
we recall the following conclusion that can be viewed as a special case of \cite[Lemma 4.4]{GuoChenFanZhao2019MMJ}.

\begin{lemma}[Sharpness of embedding, discrete form] \label{lm-exp-eb}
Suppose $1\leq q_1,q_2\leq \infty$, $s_1,s_2\in \mathbb{R}$. Then
\begin{equation*}
l^{q_1}_{v_{s_1}}(\zd)\subset l^{q_2}_{v_{s_2}}(\zd)
\end{equation*}
holds if and only if
\be
s_1-s_2\geq d(1/q_2-1/q_1)\vee 0\ \ \text{with strict inequality for}\ 1/q_2>1/q_1.
\ee
\end{lemma}

\subsection{Analysis of working spaces}
In this subsection, we introduce some useful function spaces and gather some of their frequently-used properties.  To study the boundedness of FIOs on modulation spaces $M^{p,q}(\rd)$, we will need to compute the STFT of specific  functions on \( \mathbb{R}^{2d} \). 
 To enhance the clarity of the expression, we differentiate between the STFT $V_g f(x, \xi)$, where $(x, \xi) \in \mathbb{R}^{2d}$, of \( f \in \mathcal{S}'(\mathbb{R}^d) \) and the STFT \( \scrV_\Phi F(z, \zeta)\), where \((z, \zeta) \in \mathbb{R}^{4d} \) of \( F \in \mathcal{S}'(\mathbb{R}^{2d}) \). 
We write \( z = (z_1, z_2) \in \mathbb{R}^{2d} \) and \( \zeta = (\zeta_1, \zeta_2) \in \mathbb{R}^{2d} \), when necessary. We also utilize the mixed norm Lebesgue spaces on $\rddd$ whose norm is defined as follows:
\begin{equation*}
    \|F(z,\z)\|_{L^{p_1,p_2,q_1,q_2}_{z_1,z_2,\z_1,\z_2}(\rddd)}:= 
    \bigg\|\Big\|\big\|\|F(z,\z)\|_{L^{p_1}_{z_1}(\rd)}\big\|_{L^{p_2}_{z_2}(\rd)}\Big\|_{L^{q_1}_{\z_1}(\rd)}\bigg\|_{L^{q_2}_{\z_2}(\rd)}.
\end{equation*}
For simplicity, we use the notation
\begin{equation*}
    \|F(z,\z)\|_{L^{p,q}_{z,\z}(\rddd)}:= \|F(z,\z)\|_{L^{p,p,q,q}_{z_1,z_2,\z_1,\z_2}(\rddd)}.
\end{equation*}
Now, we introduce a special Wiener amalgam space as our fundamental working space. 
For a more comprehensive introduction to Wiener amalgam spaces, refer to \cite{FeichtingerIPCoFSOB1}.

\begin{definition}\label{def-Wiener}
Let $\ep>0$.
Given a window function $\Phi\in \calS(\rdd)\backslash\{0\}$, the Wiener amalgam space $W^{\fy,\fy}_{1\otimes v_{d+\ep,d+\ep}}(\rdd)$ consists
of all $f\in \calS'(\rdd)$ such that the norm
\begin{equation}
\begin{split}
\|f\|_{W^{\fy,\fy}_{1\otimes v_{d+\ep,d+\ep}}(\rdd)}=\big\|\scrV_{\Phi}f(z,\z)\lan \z_1\ran^{d+\ep} \lan \z_2\ran^{d+\ep}\big\|_{L^{\fy,\fy}_{z,\z}}
\end{split}
\end{equation}
is finite.
\end{definition}

Let $\{\eta_k\}_{k\in \zd}$ be a smooth uniform partition of unity such that $\eta_k(x)=\eta(x-k)$ for some smooth function $\eta$ with support near the origin, satisfying 
$\sum_{k}\eta_k\equiv 1$. For $(k,l)\in \zd\times \zd$, define $\eta_{k,l}=\eta_k\otimes \eta_l$. Then set
\begin{equation*}
    \eta_k^*=\sum_{n\in \La_k}\eta_n,\ \ \La_k=\{n\in \zd: \text{supp}\eta_n\cap \text{supp}\eta_k\neq \emptyset\},
    \ \ \text{and}\ \  \eta_{k,l}^*=\eta_k^*\otimes \eta_l^*.
\end{equation*}

\begin{proposition}[Time localization of $W^{\fy,\fy}_{1\otimes v_{d+\ep,d+\ep}}(\rdd)$]\label{pp-tlp-ws1}
	The function space $W^{\fy,\fy}_{1\otimes v_{d+\ep,d+\ep}}(\rdd)$ has the time localization property as follows:
	\be
	\|F(x,\xi)\|_{W^{\fy,\fy}_{1\otimes v_{d+\ep,d+\ep}}(\rdd)}
	\sim
	\sup_{k,l}\|\eta_{k,l}F\|_{W^{\fy,\fy}_{1\otimes v_{d+\ep,d+\ep}}(\rdd)}
	\sim
	\sup_{k,l}\|\eta_{k,l}F\|_{\scrF L^{\fy,\fy}_{v_{d+\ep}\otimes v_{d+\ep}}}.
	\ee
\end{proposition}
We omit the proof of this proposition, as it can be verified via a similar argument to that used in the proof of Proposition \ref{pp-tlp-ws2}.

\begin{proposition}[Dilation operator on $W^{\fy,\fy}_{1\otimes v_{d+\ep,d+\ep}}(\rdd)$]
	Let $\la_1, \la_2\in (0,1]$.
	For $\la=(\la_1,\la_2)$, define the dilation operator $D_{\la}$ by
	\be
	D_{\la}F(x,\xi)
	=D_{(\la_1,\la_2)}F(x,\xi)
	=F(\la_1 x, \la_2 \xi).
	\ee
	We have the following estimate 
	\be
	\|D_{\la}F\|_{W^{\fy,\fy}_{1\otimes v_{d+\ep,d+\ep}}(\rdd)}\leq C \|F\|_{W^{\fy,\fy}_{1\otimes v_{d+\ep,d+\ep}}(\rdd)},
	\ee
	where the constant $C$ is independent of $\la$.	
\end{proposition}

\begin{proof}
	Denote $F_{\la}=D_{\la}F$.
	Using the translation invariant and time localization property of $W^{\fy,\fy}_{1\otimes v_{d+\ep,d+\ep}}(\rdd)$ (see Proposition \ref{pp-tlp-ws1}), 
    we only need to verify that
	\be
	\|\eta_{0,0}F_{\la}\|_{W^{\fy,\fy}_{1\otimes v_{d+\ep,d+\ep}}(\rdd)}\lesssim \|F\|_{W^{\fy,\fy}_{1\otimes v_{d+\ep,d+\ep}}(\rdd)},
	\ee
	or equivalently, to verify the following inequality
	\be
	\|\eta_{0,0}F_{\la}\|_{\scrF L^{\fy,\fy}_{v_{d+\ep}\otimes v_{d+\ep}}}\lesssim \|F\|_{W^{\fy,\fy}_{1\otimes v_{d+\ep,d+\ep}}(\rdd)}.
	\ee
    Note that $D_{\la}=D_{1,\la_2}\circ D_{\la_1,1}$. We only need to consider the cases of $\la_1=1$ and $\la_2=1$ respectively.
    Since both cases can be handled via a similar method, we only consider the case $\la_1=1$.
    In this case, we write
	\be
	\begin{split}
		\|\eta_{0,0}F_{\la}\|_{\scrF L^{\fy,\fy}_{v_{d+\ep}\otimes v_{d+\ep}}}
		= &
		\|\eta_{0,0}F_{(1,\la_2)}\|_{\scrF L^{\fy,\fy}_{v_{d+\ep}\otimes v_{d+\ep}}}
		\\
		= &
		\la_2^{-d}\|\scrF_{1,2}\big(\eta_{0,0}(x,\xi/\la_2)F(x,\xi)\big)(\z_1,\z_2/\la_2)
        \lan \z_1\ran^{d+\ep}\lan \z_2\ran^{d+\ep}\|_{L^{\fy,\fy}}.
	\end{split}
	\ee
	Observe that for $|\z_2|\geq 1$ we have
	\be
	\lan \z_2/\la_2\ran\sim \la_2^{-1}\lan \z_2\ran.
	\ee
	Using this, we obtain
	\be
	\begin{split}
		I: = &
		\la_2^{-d}\|\scrF_{1,2}\big(\eta_{0,0}(x,\xi/\la_2)F(x,\xi)\big)(\z_1,\z_2/\la_2)\lan \z_1\ran^{d+\ep}\lan \z_2\ran^{d+\ep}\chi_{|\z_2|\geq 1}\|_{L^{\fy,\fy}}
		\\
		\sim &
		\la_2^{\ep}\|\scrF_{1,2}\big(\eta_{0,0}(x,\xi/\la_2)F(x,\xi)\big)(\z_1,\z_2/\la_2)\lan \z_1\ran^{d+\ep}\lan \z_2/\la_2\ran^{d+\ep}\chi_{|\z_2|\geq 1}\|_{L^{\fy,\fy}}
		\\
		\leq &
		\la_2^{\ep}\|\scrF_{1,2}\big(\eta_{0,0}(x,\xi/\la_2)F(x,\xi)\big)(\z_1,\z_2)\lan \z_1\ran^{d+\ep}\lan \z_2\ran^{d+\ep}\|_{L^{\fy,\fy}}
		\\
		\lesssim &
		\la_2^{\ep}
		\|\scrF_{1,2}\big(\eta_{0,0}(x,\xi/\la_2)\big)(\z_1,\z_2)\lan \z_1\ran^{d+\ep}\lan \z_2\ran^{d+\ep}\|_{L^{\fy,\fy}}
        \\
        & \cdot 
		\|\scrF_{1,2}\big(\eta_{0,0}^*(x,\xi)F(x,\xi)\big)(\z_1,\z_2)\lan \z_1\ran^{d+\ep}\lan \z_2\ran^{d+\ep}\|_{L^{\fy,\fy}},
	\end{split}
	\ee
	where in the last inequality we use the convolution algebra property of $L^{\fy,\fy}_{v_{d+\ep}\otimes v_{d+\ep}}$. 
	Then, the desired estimate of $I$:
	\be
	I\lesssim \|\scrF_{1,2}\big(\eta_{0,0}^*(x,\xi)F(x,\xi)\big)(\z_1,\z_2)\lan \z_1\ran^{d+\ep}\lan \z_2\ran^{d+\ep}\|_{L^{\fy,\fy}}
	\ee
	follows by the following estimate
	\be
	\begin{split}
		&
		\la_2^{\ep}
		\|\scrF_{1,2}\big(\eta_{0,0}(x,\xi/\la_2)\big)(\z_1,\z_2)\lan \z_1\ran^{d+\ep}\lan \z_2\ran^{d+\ep}\|_{L^{\fy,\fy}}
		\\
		= &
		\la_2^{\ep}\la_2^{d}
		\|\scrF_{1,2}\big(\eta_{0,0}\big)(\z_1,\la_2 \z_2)\lan \z_1\ran^{d+\ep}\lan \z_2\ran^{d+\ep}\|_{L^{\fy,\fy}}
		\\
		= &
		\la_2^{\ep}\la_2^{d}
		\|\scrF_{1,2}\big(\eta_{0,0}\big)(\z_1, \z_2)\lan \z_1\ran^{d+\ep}\lan \z_2/\la_2\ran^{d+\ep}\|_{L^{\fy,\fy}}
		\\
		\lesssim &
		\|\scrF_{1,2}\big(\eta_{0,0}\big)(\z_1, \z_2)\lan \z_1\ran^{d+\ep}\lan \z_2\ran^{d+\ep}\|_{L^{\fy,\fy}}\lesssim 1.
	\end{split}
    \ee
    On the other hand, we have
    	\be
    \begin{split}
    	II: = &
    	\la_2^{-d}\|\scrF_{1,2}\big(\eta_{0,0}(x,\xi/\la_2)F(x,\xi)\big)(\z_1,\z_2/\la_2)\lan \z_1\ran^{d+\ep} \lan \z_2\ran^{d+\ep}\chi_{|\z_2|\leq 1}\|_{L^{\fy,\fy}}
    	\\
    	\lesssim &
    	\la_2^{-d}\|\scrF_{1,2}\big(\eta_{0,0}(x,\xi/\la_2)F(x,\xi)\big)(\z_1,\z_2/\la_2)\lan \z_1\ran^{d+\ep}\chi_{|\z_2|\leq 1}\|_{L^{\fy,\fy}}
    	\\
    	\leq &
    	\la_2^{-d}\|\scrF_{1,2}\big(\eta_{0,0}(x,\xi/\la_2)F(x,\xi)\big)(\z_1,\z_2)\lan \z_1\ran^{d+\ep}\|_{L^{\fy,\fy}}
    	\\
    	\leq &
    	\la_2^{-d}
    	\|\scrF_{1,2}\big(\eta_{0,0}(x,\xi/\la_2)\big)(\z_1,\z_2)\lan \z_1\ran^{d+\ep}\|_{L^{\fy,\fy}}
        \\
        & \cdot
    	\|\scrF_{1,2}\big(\eta_{0,0}^*(x,\xi)F(x,\xi)\big)(\z_1,\z_2)\lan \z_1\ran^{d+\ep}\|_{L^{\fy,1}}.
    \end{split}
    \ee
   The desired estimate of $II$
   \be
   II\lesssim \|\scrF_{1,2}\big(\eta_{0,0}^*(x,\xi)F(x,\xi)\big)(\z_1,\z_2)\lan \z_1\ran^{d+\ep}\lan \z_2\ran^{d+\ep}\|_{L^{\fy,\fy}}
   \ee
   follows by the following two estimates
   \be
   \begin{split}
   	   \la_2^{-d}
    	\|\scrF_{1,2}\big(\eta_{0,0}(x,\xi/\la_2)\big)(\z_1,\z_2)\lan \z_1\ran^{d+\ep}\|_{L^{\fy,\fy}}
   	=
   	\|\scrF_{1,2}\big(\eta_{0,0}^*\big)(\z_1,\z_2)\lan \z_1\ran^{d+\ep}\|_{L^{\fy,\fy}},
   \end{split}
   \ee
   and
   \be
   \begin{split}
   	&\|\scrF_{1,2}\big(\eta_{0,0}^*(x,\xi)F(x,\xi)\big)(\z_1,\z_2)\lan \z_1\ran^{d+\ep}\|_{L^{\fy,1}}
   	\\
   	\lesssim &
   	\|\scrF_{1,2}\big(\eta_{0,0}^*(x,\xi)F(x,\xi)\big)(\z_1,\z_2)\lan \z_1\ran^{d+\ep}\lan \z_2\ran^{d+\ep}\|_{L^{\fy,\fy}}
        \|\lan \z_2\ran^{-(d+\ep)}\|_{L^1}
   	\\
   	\lesssim &
   	\|\scrF_{1,2}\big(\eta_{0,0}^*(x,\xi)F(x,\xi)\big)(\z_1,\z_2)\lan \z_1\ran^{d+\ep}\lan \z_2\ran^{d+\ep}\|_{L^{\fy,\fy}}.
   \end{split}
   \ee
   Using the estimates of $I$ and $II$, we obtain that
	\be
	\begin{split}
		\|\eta_{0,0}F_{\la}\|_{\scrF L^{\fy,\fy}_{v_{d+\ep}\otimes v_{d+\ep}}}
		\lesssim &
		\|\scrF_{1,2}\big(\eta_{0,0}^*(x,\xi)F(x,\xi)\big)(\z_1,\z_2)\lan \z_1\ran^{d+\ep}\lan \z_2\ran^{d+\ep}\|_{L^{\fy,\fy}}
		\\
		= &
		\|\eta_{0,0}^*F\|_{\scrF L^{\fy,\fy}_{v_{d+\ep}\otimes v_{d+\ep}}}
        =
        \|\eta_{0,0}^*F\|_{W^{\fy,\fy}_{1\otimes v_{d+\ep,d+\ep}}(\rdd)}
        \\
        \lesssim &
        \|\eta_{0,0}^*\|_{W^{\fy,\fy}_{1\otimes v_{d+\ep,d+\ep}}(\rdd)}
        \|F\|_{W^{\fy,\fy}_{1\otimes v_{d+\ep,d+\ep}}(\rdd)}
        \lesssim \|F\|_{W^{\fy,\fy}_{1\otimes v_{d+\ep,d+\ep}}(\rdd)}.
	\end{split}
	\ee
\end{proof}

\begin{proposition}[Product algebra of $W^{\fy,\fy}_{1\otimes v_{d+\ep,d+\ep}}(\rdd)$]\label{pp-agp-W}
	\be
W^{\fy,\fy}_{1\otimes v_{d+\ep,d+\ep}}(\rdd)\cdot W^{\fy,\fy}_{1\otimes v_{d+\ep,d+\ep}}(\rdd) \subset W^{\fy,\fy}_{1\otimes v_{d+\ep,d+\ep}}(\rdd).
    \ee
\end{proposition}
\begin{proof}
    Take $\Phi_1,\Phi_2\in \calS(\rdd)$ to be two window functions. 
    Denote $\Phi=\Phi_1\Phi_2$.
    For $F,G\in W^{\fy,\fy}_{1\otimes v_{d+\ep,d+\ep}}(\rdd)$, we have
    \be
    \begin{split}
        |\scrV_{\Phi}(FG)(z,\z)|
        = &
        \bigg|\int_{\rdd}F(x,y)\overline{\Phi_1(x-z_1,y-z_2)}\overline{\Phi_2(x-z_1,y-z_2)}
        e^{2\pi i(x,y)\cdot (\z_1,\z_2)}dxdy\bigg|
        \\
        \leq &
        \int_{\rdd}\bigg|\scrV_{\Phi_1}F(z,\z-\eta)\scrV_{\Phi_2}G(z,\eta)\bigg|d\eta.
    \end{split}
    \ee
    The desired conclusion follows by
    \be
    \begin{split}
        \|FG\|_{W^{\fy,\fy}_{1\otimes v_{d+\ep,d+\ep}}(\rdd)}
        = &
        \big\|\scrV_{\Phi}(FG)(z,\z)\lan \z_1\ran^{d+\ep} \lan \z_2\ran^{d+\ep}\big\|_{L^{\fy,\fy}_{z,\z}}
        \\
        \leq &
        \big\|\int_{\rdd}\big\|\scrV_{\Phi_1}F(z,\z-\eta)\big\|_{L^{\fy}_z} \big\|\scrV_{\Phi_2}G(z,\eta)\big\|_{L^{\fy}_z} 
        \lan \z_1\ran^{d+\ep} \lan \z_2\ran^{d+\ep} d\eta\big\|_{L^{\fy}_{\z}}
        \\
        \leq &
        \big\|\scrV_{\Phi_1}F(z,\z)\lan \z_1\ran^{d+\ep} \lan \z_2\ran^{d+\ep}\big\|_{L^{\fy,\fy}_{z,\z}}
        \big\|\scrV_{\Phi_2}G(z,\z)\lan \z_1\ran^{d+\ep} \lan \z_2\ran^{d+\ep}\big\|_{L^{\fy,\fy}_{z,\z}}
        \\
        = &
        \|F\|_{W^{\fy,\fy}_{1\otimes v_{d+\ep,d+\ep}}(\rdd)}\|G\|_{W^{\fy,\fy}_{1\otimes v_{d+\ep,d+\ep}}(\rdd)},
    \end{split}
    \ee
    where in the last inequality, we apply the convolution inequality $L^{\fy}_{v_{d+\ep}}(\rd)\ast L^{\fy}_{v_{d+\ep}}(\rd)\subset L^{\fy}_{v_{d+\ep}}(\rd)$ twice; for reference, see \cite[Theorem 1.3]{GuoFanWuZhao2018SM}.
\end{proof}

For $i=1,2,3,4$, Let $c_i$ denote the permutations defined as follows:
\be
c_1(z_1,z_2,\z_1,\z_2)=(z_2,\z_2,z_1,\z_1),\ \ \  c_2(z_1,z_2,\z_1,\z_2)=(\z_2,z_2,\z_1,z_1),
\ee
and
\be
c_3(z_1,z_2,\z_1,\z_2)=(z_2,\z_2,\z_1,z_1),\ \ \  c_4(z_1,z_2,\z_1,\z_2)=(z_1,\z_1,\z_2,z_2).
\ee
\begin{definition}\label{def-MM}
    Given a window function $\Phi\in \calS(\rdd)$ and let $c_i$ be the permutation mentioned above. 
    Then $M^{p_1,p_2,q_1,q_2}(c_i)$ is the mixed modulation space of tempered
distributions $f\in \calS'(\rdd)$ for which
\be
\|f\|_{M^{p_1,p_2,q_1,q_2}_c(c_i)}=\|\scrV_{\Phi}f\circ c_i(z,\z)m(z,\z)\|_{L^{p_1,p_2,q_1,q_2}_{z_1,z_2,\z_1,\z_2}}
=\|\scrV_{\Phi}f\|_{L^{p_1,p_2,q_1,q_2}_m(c_i)}
<\fy.
\ee
\end{definition}
As a natural generalization of the classical modulation spaces,
this type of function spaces was introduced by Bishop \cite{Bishop2010JMAA} for studying the
Schatten $p$-class properties of pseudodifferential operators. 
It can be shown (cf. \cite{Bishop2010JMAA}) 
that the definition of mixed modulation spaces is independent of the choice of 
the window $g$ in $\calS(\rdd)$, with different windows giving equivalent norms.
These mixed modulation spaces ($M^{1,1,\fy,\fy}(c_i)$ for $i=1,2$) were also used in Cordero-Nicola \cite{CorderoNicola2019JFAA} for the study of kernel theorems 
on modulation spaces.
Based on the kernel theorems of modulation spaces, these mixed spaces naturally arise in the study of the boundedness of Fourier integral operators (FIOs) and are employed to estimate their distribution kernels. Some basic properties of these function spaces are listed below.

\begin{lemma}[Relation between permutation $c_1$ and $c_2$] \label{lm-rlc}
	Let $f\in \calS'(\rdd)$, $\tilde{f}(x,y)=f(y,x)$,
	and let $\phi$ be a radial Schwartz function. 
	We have the following relation:
	\be
	(\scrV_{\phi}\tilde{f})\circ c_1=(\scrV_{\phi}f)\circ c_2.
	\ee
	Moreover, the map $P: f \mapsto \tilde{f}$ is an isomorphism between $M^{p_1,p_2,q_1,q_2}(c_1)$ and $M^{p_1,p_2,q_1,q_2}(c_2)$ with
	\be
	\|\tilde{f}\|_{M^{p_1,p_2,q_1,q_2}(c_1)}\sim \|f\|_{M^{p_1,p_2,q_1,q_2}(c_2)}.
	\ee
	
\end{lemma}
\begin{proof}
	A direct calculation yields that
	\be
	\begin{split}
		(\scrV_{\phi}\tilde{f})\circ c_1(z,\z)=(\scrV_{\phi}\tilde{f})(z_2,\z_2,z_1,\z_1)=(\scrV_{\phi}f)(\z_2,z_2,\z_1,z_1)
		=
		(\scrV_{\phi}f)\circ c_2(z,\z)
	\end{split}
	\ee
\end{proof}

\begin{lemma}\label{lm-eb-WM}
Let $\ep>0$. The following relation is valid:
	\be
	1\in W^{\fy,\fy}_{1\otimes v_{d+\ep,d+\ep}}(\rdd)\subset \msjdd\subset X\subset W^{\fy,1}(\rdd)\subset L^{\fy}(\rdd)
	\ee
    for $X= M^{1,\fy,1,\fy}(c_i)$ for $i=1,2,3$ and $X=M^{\fy,1,\fy,1}(c_4)$.
\end{lemma}
This lemma can be verified by a direct calculation and using Minkowski's inequality, we omit the details.
See \cite[Theorem 1.3]{GuoWuZhao2017JoFA} for the last embedding relation.

\begin{lemma}\label{lm-pd-c12}
        Let $f$ be a measurable function on $\rdd$ with at most polynomial growth.
	For $i=1,2$, the following two statements are equivalent:
	\bn
	\item
	$f\in M^{1,1,\fy,\fy}(c_i)$,
	\item
	$\|fg\|_{M^{1,1,\fy,\fy}(c_i)}\lesssim \|g\|_{M^{1,\fy,1,\fy}(c_i)}$ for all $g\in M^{1,\fy,1,\fy}(c_i)$,
	\item
	$\|fg\|_{M^{1,1,\fy,\fy}(c_i)}\lesssim \|g\|_{\msjdd}$ for all $g\in \msjdd$.
	\item
	$\|fg\|_{M^{1,1,\fy,\fy}(c_i)}\lesssim \|g\|_{W^{\fy,\fy}_{1\otimes v_{d+\ep,d+\ep}}(\rdd)}$ for all $g\in W^{\fy,\fy}_{1\otimes v_{d+\ep,d+\ep}}(\rdd)$.
	\en
	Moreover, the following product inequalities is valid:
    \be
    \|fg\|_{M^{1,1,\fy,\fy}(c_i)}\lesssim \|f\|_{M^{1,1,\fy,\fy}(c_i)}\|g\|_X, \ \ \ X=M^{1,\fy,1,\fy}(c_i), \msjdd, W^{\fy,\fy}_{1\otimes v_{d+\ep,d+\ep}}(\rdd).
    \ee 
\end{lemma}
\begin{proof}
	We only give the proof for the permutation $c_1$, then the $c_2$ case follows by using Lemma \ref{lm-rlc}.
	The relation $(2)\Longrightarrow (3)\Longrightarrow (4)\Longrightarrow (1)$ follows by the embedding relation from Lemma \ref{lm-eb-WM}:
	\be
	1\in W^{\fy,\fy}_{1\otimes v_{d+\ep,d+\ep}}(\rdd)\subset \msjdd\subset M^{1,\fy,1,\fy}(c_1).
	\ee
	To complete this proof,
	we only need to verify the relation $(1)\Longrightarrow (2)$.
	For $\Phi_1,\Phi_2\in \calS(\rdd)$, $f\in M^{1,1,\fy,\fy}(c_1)$ with at most polynomial growth, 
    and $g\in M^{1,\fy,1,\fy}(c_1)$, we have $f\overline{T_z\Phi_1}, g\overline{T_z\Phi_2}\in L^2(\rdd)$, and
        \begin{align*}
            \scrV_{\Phi_1\Phi_2}(fg)(z,\z)
            = &
            \scrF(f\overline{T_z\Phi_1}g\overline{T_z\Phi_2})(\z)
            \\
            = &
            \int_{\rdd}\scrF(f\overline{T_z\Phi_1})(\z-\eta)\scrF(g\overline{T_z\Phi_2})(\eta)d\eta
            \\
            = &
            \int_{\rdd}\scrV_{\phi_1}f(z,\z-\eta)\scrV_{\phi_2}g(z,\eta)d\eta.
        \end{align*}
    From this, we deduce that
	\be
	|\scrV_{\phi_1\phi_2}(fg)(z,\z)|
	\leq 
	\int_{\rdd}|\scrV_{\phi_1}f(z,\z-\eta)\scrV_{\phi_2}g(z,\eta)|d\eta.
	\ee
   By applying the following sequence of inequalities:
   \begin{equation*}
       L^1\ast L^1\subset L^1,\  L^1\cdot L^{\fy}\subset L^1,\  L^1\ast L^{\fy}\subset L^1,\ L^{\fy}\cdot L^{\fy}\subset L^{\fy},
   \end{equation*}
   we obtain the desired conclusion
   \begin{align*}
       \|fg\|_{M^{1,1,\fy,\fy}(c_i)}
       = &
       \bigg\|\Big\|\big\|\|\scrV_{\Phi_1\Phi_2}(fg)(z,\z)\|_{L^{1}_{\z_1}(\rd)}\big\|_{L^{1}_{z}(\rd)}\Big\|_{L^{\fy}_{\z_2}(\rd)}\bigg\|_{L^{\fy}_{z_2}(\rd)}
       \\
       \lesssim &
       \|f\|_{M^{1,1,\fy,\fy}(c_i)}\|g\|_{M^{1,\fy,1,\fy}(c_i)}.
   \end{align*}
\end{proof}
By a similar argument to that of the above lemma, we give the following two lemmas associated with the permutations $c_3$ and $c_4$.
\begin{lemma}\label{lm-pd-c3}
	Let $f$ be a measurable function on $\rdd$ with at most polynomial growth.
	The following two statements are equivalent:
	\bn
	\item
	$f\in M^{1,\fy,1,\fy}(c_3)$,
	\item
	$\|fg\|_{M^{1,\fy,1,\fy}(c_3)}\lesssim \|g\|_{M^{1,\fy,1,\fy}(c_3)}$ for all $g\in M^{1,\fy,1,\fy}(c_3)$,
	\item
	$\|fg\|_{M^{1,\fy,1,\fy}(c_3)}\lesssim \|g\|_{\msjdd}$ for all $g\in \msjdd$.
	\item
	$\|fg\|_{M^{1,\fy,1,\fy}(c_3)}\lesssim \|g\|_{W^{\fy,\fy}_{1\otimes v_{d+\ep,d+\ep}}(\rdd)}$ for all $g\in W^{\fy,\fy}_{1\otimes v_{d+\ep,d+\ep}}(\rdd)$.
	\en
    Moreover, the following product inequalities is valid:
    \be
    \|fg\|_{M^{1,1,\fy,\fy}(c_3)}\lesssim \|f\|_{M^{1,1,\fy,\fy}(c_3)}\|g\|_X, \ \ \ X=M^{1,\fy,1,\fy}(c_3), \msjdd, W^{\fy,\fy}_{1\otimes v_{d+\ep,d+\ep}}(\rdd).
    \ee 
\end{lemma}

\begin{lemma}\label{lm-pd-c4}
	Let $f$ be a measurable function on $\rdd$ with at most polynomial growth.
	The following two statements are equivalent:
	\bn
	\item
	$f\in M^{1,\fy,1,\fy}(c_4)$,
	\item
	$\|fg\|_{M^{1,\fy,1,\fy}(c_4)}\lesssim \|g\|_{M^{\fy,1,\fy,1}(c_4)}$ for all $g\in M^{\fy,1,\fy,1}(c_4)$,
	\item
	$\|fg\|_{M^{1,\fy,1,\fy}(c_4)}\lesssim \|g\|_{\msjdd}$ for all $g\in \msjdd$.
	\item
	$\|fg\|_{M^{1,\fy,1,\fy}(c_4)}\lesssim \|g\|_{W^{\fy,\fy}_{1\otimes v_{d+\ep,d+\ep}}(\rdd)}$ for all $g\in W^{\fy,\fy}_{1\otimes v_{d+\ep,d+\ep}}(\rdd)$.
	\en
    Moreover, the following product inequalities is valid:
    \be
    \|fg\|_{M^{1,1,\fy,\fy}(c_4)}\lesssim \|f\|_{M^{1,1,\fy,\fy}(c_4)}\|g\|_X, \ \ \ X=M^{\fy,1,\fy,1}(c_4), \msjdd, W^{\fy,\fy}_{1\otimes v_{d+\ep,d+\ep}}(\rdd).
	\ee
\end{lemma}

\begin{proposition}
	[Time localization for $M^{1,1,\fy,\fy}(c_1)$]\label{pp-tlp-ws3}
	The function space $M^{1,1,\fy,\fy}(c_1)$ has the time localization property as follows:
	\be
	\|F(x,\xi)\|_{M^{1,1,\fy,\fy}(c_1)}
	\sim
	\sup_{l}\|(1\otimes\eta_{l})F\|_{M^{1,1,\fy,\fy}(c_1)}
	\sim
	\sup_{l}\|\scrF_{\xi}(\eta_l(\xi)V_{\psi}F(\cdot,\xi)(z_1,\z_1))(\z_2)\|_{L^{1,1,\fy}_{\z_1,z_1,\z_2}}.
	\ee
\end{proposition}
\begin{proof}
	Take the window function $\Psi=\psi\otimes \psi\in C_c^{\fy}(\rd)\otimes C_c^{\fy}(\rd)$ 
	with $\text{supp}\psi\in B(0,\d)$. We have
	\be
	\begin{split}
		\|(1\otimes\eta_{l})F\|_{M^{1,1,\fy,\fy}(c_1)}
		\lesssim &
		\|1\otimes\eta_{l}\|_{M^{1,\fy,1,\fy}(c_1)}\|F\|_{M^{1,1,\fy,\fy}(c_1)}
		\lesssim
		\|F\|_{M^{1,1,\fy,\fy}(c_1)}.
	\end{split}
	\ee
	On the other hand, for any fixed $z_2\in \rd$ we have
	\be
	\scrV_{\Psi}F(z,\z)=\sum_{l}\scrV_{\Psi}((1\otimes\eta_{l})F)(z,\z),
	\ee
	and
	\be
	\|\scrV_{\Psi}F(z,\z)\|_{L^{1,1,\fy}_{\z_1,z_1,\z_2}}\leq \sum_{l}\|\scrV_{\Psi}((1\otimes\eta_{l})F)(z,\z)\|_{L^{1,1,\fy}_{\z_1,z_1,\z_2}}
	\lesssim \sup_{l}\|\scrV_{\Psi}((1\otimes\eta_{l})F)(z,\z)\|_{L^{1,1,\fy}_{\z_1,z_1,\z_2}},
	\ee	
	where the summation only have finite nonzero terms.
    The desired conclusion follows by taking $L^{\fy}_{z_2}$ norm.
    
	Next, we only need to verify that
	\be
	\|F\|_{M^{1,1,\fy,\fy}(c_1)}
	\sim
	\|\scrF_{\xi}(V_{\psi}F(\cdot,\xi)(z_1,\z_1))(\z_2)\|_{L^{1,1,\fy}_{\z_1,z_1,\z_2}},\ \ \ \    \text{supp}F(x,\cdot)\subset B(0,r).
	\ee
	Taking $\psi\in C_c^{\fy}(\rd)$ such that $F\cdot T_{z_2}\psi=F$ for $|z_2|<\d$, we find that for $|z_2|<\d$
	\be
	\begin{split}
		\scrV_{\Psi}F(z,\z)
		= &
		\int_{\rdd}F(x,\xi)\psi(x-z_1)\psi(\xi-z_2)e^{-2\pi i(x\cdot\z_1+\xi\cdot\z_2)}dxd\xi
		\\
		= &
		\int_{\rdd}F(x,\xi)\psi(x-z_1)e^{-2\pi i(x\cdot\z_1+\xi\cdot\z_2)}dxd\xi
		\\
		= &
		\scrF_{\xi}(V_{\psi}F(\cdot,\xi)(z_1,\z_1))(\z_2).
	\end{split}
	\ee 
	From this, we obtain that
	\be
	\begin{split}
		\|F\|_{M^{1,1,\fy,\fy}(c_1)}
		\sim &
		\sup_{z_2}\|\scrV_{\Psi}F(z,\cdot)\|_{L^{1,1,\fy}_{\z_1,z_1,\z_2}}
		\\
		\geq &
		\sup_{|z_2|<\d}\|\scrV_{\Psi}F(z,\cdot)\|_{L^{1,1,\fy}_{\z_1,z_1,\z_2}}=\|\scrF_{\xi}(V_{\psi}F(\cdot,\xi)(z_1,\z_1))(\z_2)\|_{L^{1,1,\fy}_{\z_1,z_1,\z_2}}.
	\end{split}
	\ee
	On the other hand, observe that for every $z_2$ we have
	\be
	|\scrV_{\Psi}F(z,\z)|
	\leq
	\int_{\rd}|\scrF_{\xi}(V_{\psi}F(\cdot,\xi)(z_1,\z_1))(\z_2-\eta)|\cdot |\scrF \psi(\eta)|d\eta.
	\ee
	The desired conclusion $\|F\|_{M^{1,1,\fy,\fy}(c_1)}
	\lesssim
	\|\scrF_{\xi}(V_{\psi}F(\cdot,\xi)(z_1,\z_1))(\z_2)\|_{L^{1,1,\fy}_{\z_1,z_1,\z_2}}$
	follows by taking the $L^{1,1,\fy,\fy}_{\z_1,z_1,\z_2,z_2}$ norm.
\end{proof}

\begin{proposition}
	[Time localization for $M^{1,\fy,\fy,\fy}_{1\otimes v_{d+\ep,0}}(c_1)$]\label{pp-tlp-ws2}
	The function space $M^{1,\fy,\fy,\fy}_{1\otimes v_{d+\ep,0}}(c_1)$ has the time localization property as follows:
	\be
	\|F(x,\xi)\|_{M^{1,\fy,\fy,\fy}_{1\otimes v_{d+\ep,0}}(c_1)}
	\sim
	\sup_{k,l}\|\eta_{k,l}F\|_{M^{1,\fy,\fy,\fy}_{1\otimes v_{d+\ep,0}}(c_1)}
	\sim
	\sup_{k,l}\|\eta_{k,l}F\|_{\scrF L^{1,\fy}_{1\otimes v_{d+\ep}}}.
	\ee
\end{proposition}
\begin{proof}
	Take the real-valued window function $\Psi\in C_c^{\fy}(\rdd)$ supported on $B(0,\d)$.
	A direct calculation yields that
	\be
	\begin{split}
		|\scrV_{\Psi}(\eta_{k,l}F)(z,\z)|
		\lesssim
		\int_{\rdd}|\scrV_{\Psi}F(z,\z-\eta)|\cdot |(\scrF \eta_{k,l})(\eta)|d\eta.
	\end{split}
    \ee
    From this and the fact that $L^{1,\fy}_{1\otimes v_{d+\ep}}$ is a convolution algebra, we conclude that
    \be
    \|\scrV_{\Psi}(\eta_{k,l}F)(z,\cdot )\|_{L^{1,\fy}_{1\otimes v_{d+\ep}}}
    \lesssim
    \|\scrV_{\Psi}F(z,\cdot )\|_{L^{1,\fy}_{1\otimes v_{d+\ep}}}
    \|\scrF \eta_{k,l}\|_{L^{1,\fy}_{1\otimes v_{d+\ep}}},
	\ee
	this implies the desired estimate
	\be
	\begin{split}
	\|\eta_{k,l}F\|_{M^{1,\fy,\fy,\fy}_{1\otimes v_{d+\ep,0}}(c_1)}
	\sim &
    \sup_z \|\scrV_{\Psi}(\eta_{k,l}F)(z,\cdot )\|_{L^{1,\fy}_{1\otimes v_{d+\ep}}}
    \\
    \lesssim &
    \sup_z \|\scrV_{\Psi}F(z,\cdot )\|_{L^{1,\fy}_{1\otimes v_{d+\ep}}}
    \sim
    \|F\|_{M^{1,\fy,\fy,\fy}_{1\otimes v_{d+\ep,0}}(c_1)}.
	\end{split}
	\ee
	On the other hand, for any fixed $z\in \rdd$ we have
	\be
	\scrV_{\Psi}F(z,\z)=\sum_{k,l}\scrV_{\Psi}(\eta_{k,l}F)(z,\z),
	\ee
	and
	\be
	\|\scrV_{\Psi}F(z,\cdot)\|_{L^{1,\fy}_{1\otimes v_{d+\ep}}}\leq \sum_{k,l}\|\scrV_{\Psi}(\eta_{k,l}F)(z,\cdot)\|_{L^{1,\fy}_{1\otimes v_{d+\ep}}},
	\ee	
	where the summation only have finite nonzero terms.
	From this, we further obtain the desired conclusion
	\be
	\begin{split}
	\|\scrV_{\Psi}F(z,\cdot)\|_{L^{1,\fy}_{1\otimes v_{d+\ep}}}
	\lesssim &
	\sup_{k,l}\|\scrV_{\Psi}(\eta_{k,l}F)(z,\cdot)\|_{L^{1,\fy}_{1\otimes v_{d+\ep}}}
	\\
	\leq &
	\sup_{k,l}\sup_z\|\scrV_{\Psi}(\eta_{k,l}F)(z,\cdot)\|_{L^{1,\fy}_{1\otimes v_{d+\ep}}}
	\sim
    \sup_{k,l}\|\eta_{k,l}F\|_{M^{1,\fy,\fy,\fy}_{1\otimes v_{d+\ep,0}}(c_1)}.	
   \end{split}
	\ee
	Next, we only need to verify that
	\be
	\|F\|_{M^{1,\fy,\fy,\fy}_{1\otimes v_{d+\ep,0}}(c_1)}
	\sim
	\|F\|_{\scrF L^{1,\fy}_{1\otimes v_{d+\ep}}},\ \ \ \    \text{supp}F\subset B(0,R).
	\ee
	Taking $\Psi\in C_c^{\fy}(\rdd)$ such that $F\cdot T_{z}\Psi=F$ for $|z|<\d$, we find that
	\be
	\scrV_{\Psi}F(z,\z)=(\scrF F)(\z),\ \ \   |z|<\d.
	\ee 
	From this, we obtain that
	\be
	\begin{split}
		\|F\|_{M^{1,\fy,\fy,\fy}_{1\otimes v_{d+\ep,0}}(c_1)}
		\sim &
		\sup_z\|\scrV_{\Psi}F(z,\cdot)\|_{L^{1,\fy}_{1\otimes v_{d+\ep}}}
		\\
		\geq &
		\sup_{|z|<\d}\|\scrV_{\Psi}F(z,\cdot)\|_{L^{1,\fy}_{1\otimes v_{d+\ep}}}=\|F\|_{\scrF L^{1,\fy}_{1\otimes v_{d+\ep}}}.
	\end{split}
	\ee
	On the other hand, observe that
	\be
	|\scrV_{\Psi}F(z,\z)|
	\leq
	\int_{\rdd}|\scrF F(\z-\eta)|\cdot |\scrF \Psi(\eta)|d\eta.
	\ee
	The desired conclusion $\|F\|_{M^{1,\fy,\fy,\fy}_{1\otimes v_{d+\ep,0}}(c_1)}
	\lesssim
	\|F\|_{\scrF L^{1,\fy}_{1\otimes v_{d+\ep}}}$
	follows by using the convolution algebra property of $L^{1,\fy}_{1\otimes v_{d+\ep}}$.
\end{proof}

The following lemma is used to transform weighted positions, as referenced in \cite[Proposition 2.8]{CorderoNicola2010JFAA}.
\begin{lemma}
Let $\Phi$ be a real-valued function on $\rdd$. We have
\[
\sigma \in M_{v_{s_1}, v_{s_2}}^{\infty, 1} \bigl( \mathbb{R}^{2d} \bigr) 
\iff 
v_{s_1,s_2}\s \in M^{\infty, 1} \bigl( \mathbb{R}^{2d} \bigr).
\]
Moreover, the following statements are equivalent:
\begin{enumerate}
    \item \( T_{\s,\Phi} \) is bounded from \( M^{p_1, q_1} \) into \( M^{p_2,q_2}\),
    \item \(T_{v_{s_1,s_2}\s}\) is bounded from \( M_{v_{0, s_2}}^{p_1, q_1} \) into \( M_{v_{-s_1, 0}}^{p_2,q_2} \).
\end{enumerate}
\end{lemma}

We also need the following potential lifting lemma for mixed modulation spaces.
\begin{lemma}[Potential lifting]
    Let $1\leq p_1, p_2, q_1, q_2\leq \fy$, $m\in \scrP(\rddd)$.  
    Let $c$ be a permutation.
    For $v_{s_1,s_2}(x,\xi)$, the mapping $F\mapsto v_{s_1,s_2}F$ is a topological isomorphism from $M_{m(v_{s_1,s_2}\otimes 1)\circ c}^{p_1,p_1,q_1,q_2}(c)$ to $M_m^{p_1,p_2,q_1,q_2}(c)$ with the following estimate:
    \begin{align*}
         \|v_{s_1,s_2}F\|_{M^{p_1,p_2,q_1,q_2}_m(c)}
        \sim & 
        \|(v_{s_1,s_2}\otimes 1)\scrV_{\Phi_2}(F)\|_{L^{p_1,p_2,q_1,q_2}_m(c)}
        \\
        \sim &
        \|\scrV_{\Phi_2}F\|_{L^{p_1,p_2,q_1,q_2}_{m(v_{s_1,s_2}\otimes 1)\circ c}(c)}    
        \sim
        \|F\|_{M^{p_1,p_2,q_1,q_2}_{m(v_{s_1,s_2}\otimes 1)\circ c}(c)}.
    \end{align*}
\end{lemma}
This lemma can be verified by a similar argument as in \cite[Theorem 2.2]{Toft2004AGAG}, we omit the details here.

\subsection{A revisit on the kernel theorems on modulation spaces}
It is well known that kernel theorems are highly effective in studying the boundedness of
integral operators on modulation spaces.
It can convert the investigation of boundedness into the norm estimation of the distributional kernel.
In the realm of time-frequency analysis, Feichtinger's kernel theorem for linear operators on $\calL(M^1(\rd), M^{\fy}(\rd))$ was first presented in \cite{Feichtinger1980CRHdSdldSSA} and proved in \cite{FeichtingerGroechenig1997JoFA}.
In 2019, Cordero and Nicola \cite{CorderoNicola2019JFAA} 
proposed a generalized kernel theorem related to linear operators belonging to $\calL(M^1(\rd),M^p(\rd))$ and $\calL(M^p(\rd),\mfd)$ respectively.
Based on \cite[Corollary 3.4 and Theorem 3.6]{CorderoNicola2019JFAA} and the analysis of working spaces established in Subsection 2.2, we derive a corresponding kernel theorem for FIOs on $M^1(\rd)$ and $M^{\fy}(\rd)$.
This theorem is more suitable for our context and holds independent significance.

\begin{proposition}[FIO on $M^1$ and $\mf$]\label{pp-KFIO-Mi-Mf}
	Let $\Phi$ be a real-valued function on $\rdd$. Denote $\tilde{\Phi}(x,\xi)=\Phi(\xi,x)$.
	The following statements are equivalent:
	\bn
	\item
	For any $\s\in M^{\fy,1}_{v_{s_1,s_2}\otimes 1}$, we have $T_{\s,\Phi}\in \calL(\mfid)$,
	\item
	$\|T_{\s,\Phi}\|_{\calL(\mfid)}\lesssim \|\s\|_{M^{\fy,1}_{v_{s_1,s_2}\otimes 1}}$ for any $\s\in M^{\fy,1}_{v_{s_1,s_2}\otimes 1}$,
	\item
	$\|\s e^{2\pi i\Phi}\|_{M^{1,1,\fy,\fy}(c_1)}\lesssim \|\s\|_{M^{\fy,1}_{v_{s_1,s_2}\otimes 1}}$ for any $\s\in M^{\fy,1}_{v_{s_1,s_2}\otimes 1}$,
	\item
	$v_{s_1,s_2}^{-1}e^{2\pi i\Phi}\in M^{1,1,\fy,\fy}(c_1)$,
	\item
	$T_{v_{s_1,s_2}^{-1},\Phi}\in \calL(\mfid)$,
	\item
	$T_{v_{s_2,s_1}^{-1},\tilde{\Phi}}\in \calL(\mfd)$,
	\item
	$v_{s_2,s_1}^{-1}e^{2\pi i\tilde{\Phi}}\in M^{1,1,\fy,\fy}(c_2)$,
	\item
	$\|\s e^{2\pi i\tilde{\Phi}}\|_{M^{1,1,\fy,\fy}(c_2)}\lesssim \|\s\|_{M^{\fy,1}_{v_{s_2,s_1}\otimes 1}}$ for any $\s\in M^{\fy,1}_{v_{s_2,s_1}\otimes 1}$,
	\item
	$\|T_{\s,\tilde{\Phi}}\|_{\calL(\mfd)}\lesssim \|\s\|_{M^{\fy,1}_{v_{s_2,s_1}\otimes 1}}$ for any $\s\in M^{\fy,1}_{v_{s_2,s_1}\otimes 1}$,
	\item
	For any $\s\in M^{\fy,1}_{v_{s_2,s_1}\otimes 1}$, we have $T_{\s,\tilde{\Phi}}\in \calL(\mfd)$.
	\en
Moreover, if any of the above statements holds, the following norm estimates are valid:
\begin{equation*}
    \|T_{\s,\Phi}\|_{\calL(M^1(\rd))}\sim \|\s e^{2\pi i\Phi}\|_{M^{1,1,\fy,\fy}(c_1)},\ \ \ 
    \|T_{\s,\tilde{\Phi}}\|_{\calL(\mfd)}\sim \|\s e^{2\pi i\tilde{\Phi}}\|_{M^{1,1,\fy,\fy}(c_2)}.
\end{equation*}

\end{proposition}
\begin{proof}
	We define a map from $M^{\fy,1}_{v_{s_1,s_2}\otimes 1}$ to $\calL(\mfid)$ by
	\be
	P: \s\rightarrow T_{\s,\Phi}.
	\ee
	Given any sequence of pairs that $(\s_n, T_{\s_n,\Phi})$ tends to $(\s, T)$ in the topology of graph of
	$P$, for any $f,g \in \calS(\rd)$, we find
	\be
	\begin{split}
		\lan T_{\s,\Phi}f, g\ran
		= &
		\lan \s e^{2\pi i\Phi}, g\otimes \bar{\hat{f}}\ran
		\\
		= &
		\lim_{n\rightarrow \fy}\lan \s_n e^{2\pi i\Phi}, g\otimes \bar{\hat{f}}\ran
		=
		\lim_{n\rightarrow \fy}\lan T_{\s_n,\Phi}f, g \ran=\lan Tf, g\ran.
	\end{split}
	\ee
	Here, we use the fact that $\s_n\rightarrow \s$ in $M^{\fy,1}_{v_{s_1,s_2}\otimes 1}$ implies $\s_n e^{2\pi i\Phi}\rightarrow \s e^{2\pi i\Phi}$ in $\calS'$, which can be verified by
    \begin{equation*}
        \s_n\xrightarrow{M^{\fy,1}_{v_{s_1,s_2}\otimes 1}} \s
        \Longrightarrow
        \s_n\xrightarrow{L^{\fy}_{v_{s_1,s_2}}} \s
        \Longrightarrow
        \s_n e^{2\pi i\Phi}\xrightarrow{L^{\fy}_{v_{s_1,s_2}}} \s e^{2\pi i\Phi}
                \Longrightarrow
        \s_n e^{2\pi i\Phi}\xrightarrow{\calS'} \s e^{2\pi i\Phi}.
    \end{equation*}
	From this, we conclude that $T=T_{\s,\Phi}$. Thus, the operator $P$ is a closed operator.
	Hence, $P$ is bounded by the closed graph theorem. This completes the proof of $(1)\Longrightarrow (2)$.
	The opposite direction $(2)\Longrightarrow (1)$ is obvious.
	Using the same argument, on can also verify $(9)\Longleftrightarrow (10)$
	
	Next, the equivalent relations $(2)\Longleftrightarrow (3)$, $(4)\Longleftrightarrow (5)$, 
	$(6)\Longleftrightarrow (7)$ and $(8)\Longleftrightarrow (9)$
	follow by the kernel theorems for modulation spaces (see \cite[Corollary 3.4 and Theorem 3.6]{CorderoNicola2019JFAA})
    and the fact that the distributional kernel of $T_{\s,\Phi}$ is $K=\scrF_2(\s e^{2\pi i\Phi})$, 
	the equivalent relations $(4)\Longleftrightarrow (7)$ follows by Lemma \ref{lm-rlc},
	the equivalent relations $(3)\Longleftrightarrow (4)$ and $(7)\Longleftrightarrow (8)$ follows by Lemma \ref{lm-pd-c12} and the fact that both $v_{s_1,s_2}^{-1}e^{2\pi i\Phi}$ and $v_{s_2,s_1}^{-1}e^{2\pi i\tilde{\Phi}}$ are measurable functions with at most polynomial growth.
\end{proof}

\begin{remark}
	We point out that Proposition \ref{pp-KFIO-Mi-Mf} demonstrates, to some extent, 
    the independence of the Sj\"{o}strand symbol in the boundedness of FIO. Some interesting applications are listed below:
	\bn
	\item   
	By Proposition \ref{pp-KFIO-Mi-Mf} and the fact that $1\in S_{0,0}^0$, we observe that
	for a fixed phase function $\Phi$,
	  all boundedness results on $\mfi$ and $\mf$ of FIOs
	associated with the H\"{o}rmander class $S_{0,0}^0$ automatically extend to FIOs with $\msj$ symbols. 
	\item 
	When $\Phi(x,\xi)=x\cdot \xi$, the operator $T_{1,\Phi}$ reduces to the identity operator, and  $T_{\s,\Phi}=K_{\s}$ becomes the Kohn-Nirenberg-type pseudodifferential operator with
	symbol $\s$. 
    Applying Proposition \ref{pp-KFIO-Mi-Mf}, the well-known $M^p$ boundedness of $K_{\s}$ (see \cite{Sjoestrand1994MRL, Groechenig2006RMI} ) follows immediately from the trivial fact that the identity operator is bounded on $M^p$.
	\en
\end{remark}

\begin{proposition}[FIO on $M^{\fy,1}$]\label{pp-KFIO-Mfi}
	Let $\Phi$ be a real-valued function on $\rdd$.
    	The following statements are equivalent:
        \bn
    \item $\forall \s\in M^{\fy,1}_{v_{s_1,s_2}\otimes 1}$, we have $\s e^{2\pi i\Phi}\in M^{1,\fy,1,\fy}(c_3)$,
    \item $\forall \s\in M^{\fy,1}$, we have $\s v_{s_1,s_2}^{-1} e^{2\pi i\Phi}\in M^{1,\fy,1,\fy}(c_3)$,
    \item $v_{s_1,s_2}^{-1} e^{2\pi i\Phi}\in M^{1,\fy,1,\fy}(c_3)$.
        \en
    Furthermore, if one of the above holds, the following boundedness result is valid:
    \be
    \forall \s\in M^{\fy,1}_{v_{s_1,s_2}\otimes 1}\ \ \text{we have}\ \ T_{\s,\Phi}\in \calL(M^{\fy,1}),
    \ee
    with the norm estimate as follows:
    \begin{equation*}
        \|T_{\s,\Phi}\|_{\calL(M^{\fy,1})}\lesssim \|\s e^{2\pi i\Phi}\|_{M^{1,\fy,1,\fy}(c_3)}.
    \end{equation*}
\end{proposition}
\begin{proof}
The equivalent relation $(1)\Longleftrightarrow (2)\Longleftrightarrow (3)$ follows by Lemma \ref{lm-pd-c3}.
The boundedness result follows by \cite[Theorem 4.3 (i)]{CorderoNicola2019JFAA} and the fact that
the distributional kernel of $T_{\s,\Phi}$ is $K=\scrF_2(\s e^{2\pi i\Phi})$.
\end{proof}

\begin{proposition}[FIO on $M^{1,\fy}$]\label{pp-KFIO-Mif}
	Let $\Phi$ be a real-valued function on $\rdd$.
    	The following statements are equivalent:
        \bn
    \item $\forall \s\in M^{\fy,1}_{v_{s_1,s_2}\otimes 1}$, we have $\s e^{2\pi i\Phi}\in M^{1,\fy,1,\fy}(c_4)$,
    \item $\forall \s\in M^{\fy,1}$, we have $\s v_{s_1,s_2}^{-1} e^{2\pi i\Phi}\in M^{1,\fy,1,\fy}(c_4)$,
    \item $v_{s_1,s_2}^{-1} e^{2\pi i\Phi}\in M^{1,\fy,1,\fy}(c_4)$.
        \en
    Furthermore, if one of the above holds, the following boundedness result is valid:
    \be
    \forall \s\in M^{\fy,1}_{v_{s_1,s_2}\otimes 1}\ \ \text{we have}\ \ T_{\s,\Phi}\in \calL(M^{1, \fy}),
    \ee
    with the norm estimate as follows:
    \begin{equation*}
        \|T_{\s,\Phi}\|_{\calL(M^{1,\fy})}\lesssim \|\s e^{2\pi i\Phi}\|_{M^{1,\fy,1,\fy}(c_4)}.
    \end{equation*}
\end{proposition}
\begin{proof}
The equivalent relation $(1)\Longleftrightarrow (2)\Longleftrightarrow (3)$ follows by Lemma \ref{lm-pd-c4}.
The boundedness result follows by \cite[Theorem 4.3 (ii)]{CorderoNicola2019JFAA} and the fact that
the distributional kernel of $T_{\s,\Phi}$ is $K=\scrF_2(\s e^{2\pi i\Phi})$.
\end{proof}

\section{Boundedness on $\mpq$ with separation properties}

\subsection{Boundedness on $\mp$}

\begin{theorem}[FIO on $\mfi$]\label{thm-FIO-M1}
    Let $\Phi$ be a real-valued $C^2(\rdd)$ phase function satisfying
    the $(0,0,0)$-growth condition and the uniform separation condition of $x$-type. 
    Then, we have
	\be
    e^{2\pi i\Phi}\in M^{1,1,\fy,\fy}(c_1),\ \ \ 
	\|e^{2\pi i\Phi}\|_{M^{1,1,\fy,\fy}(c_1)}\lesssim e^{C\sum_{|\g|=2}\|\partial^{\g}\Phi\|_{W^{\fy,\fy}_{1\otimes v_{d+\ep,d+\ep}}}}.
	\ee
	Moreover,
	for every $\s\in \msjdd$, the Fourier integral operator $T_{\s,\Phi}\in \calL(\mfid)$ with the estimate of operator norm
	\be
	\|T_{\s,\Phi}\|_{\calL(\mfid)}\lesssim \|\s\|_{\msjdd}e^{C\sum_{|\g|=2}\|\partial^{\g}\Phi\|_{W^{\fy,\fy}_{1\otimes v_{d+\ep,d+\ep}}}}.
	\ee
\end{theorem}
\begin{proof}
	We first deal with the estimate of $e^{2\pi i\Phi}$.
	Denote by $\eta_{k,l}^*=\eta_k^*\otimes \eta_l^*$.
	Let $\G=\La\times \La$ be a subset of $\zdd$ such that $\text{supp}\eta_{k,l}^*\cap \text{supp}\eta_{\tilde{k},\tilde{l}}^*=\emptyset$, for all $(k,l)\neq (\tilde{k},\tilde{l})$
    with $(k,l),(\tilde{k},\tilde{l})\in \G$.
	We only need to verify 
	\be
	\|\sum_{(k,l)\in \G}\eta_{k,l}e^{2\pi i\Phi}\|_{M^{1,1,\fy,\fy}(c_1)}\lesssim e^{C\sum_{|\g|=2}\|\partial^{\g}\Phi\|_{W^{\fy,\fy}_{1\otimes v_{d+\ep,d+\ep}}}}.
	\ee
	Write
	\be
	\begin{split}
		&\sum_{(k,l)\in \G}\eta_{k,l}(x,\xi)e^{2\pi i\Phi(x,\xi)}
		\\
		= &
		\left(\sum_{(k,l)\in \G}\eta_{k,l}(x,\xi)e^{2\pi i(\Phi(x,\xi)-\nabla_{x} \Phi(k,l)\cdot x-\nabla_{\xi} \Phi(k,l)\cdot \xi)}\right)
		\left(\sum_{(k,l)\in \G}\eta_{k,l}^{*}(x,\xi)e^{2\pi i(\nabla_{x} \Phi(k,l)\cdot x+\nabla_{\xi} \Phi(k,l)\cdot \xi)}\right).
	\end{split}
	\ee
	The estimate of $e^{i\Phi}$ follows by 
	the product inequality 
	\be
	M^{1,1,\fy,\fy}(c_1) \cdot W^{\fy,\fy}_{1\otimes v_{d+\ep,d+\ep}}(\rdd)\subset M^{1,1,\fy,\fy}(c_1),
	\ee
	and
	the following two estimates
	\be
	E_1=\bigg\|\sum_{(k,l)\in \G}\eta_{k,l}^{*}(x,\xi)e^{2\pi i(\nabla_{x} \Phi(k,l)\cdot x+\nabla_{\xi} \Phi(k,l)\cdot \xi)}\bigg\|_{M^{1,1,\fy,\fy}(c_1)}\lesssim 1,
	\ee
	and 
	\be
	E_2=\bigg\| \sum_{(k,l)\in \G}\eta_{k,l}(x,\xi)e^{2\pi i(\Phi(x,\xi)-\nabla_{x} \Phi(k,l)\cdot x-\nabla_{\xi} \Phi(k,l)\cdot \xi)}\bigg\|_{W^{\fy,\fy}_{1\otimes v_{d+\ep,d+\ep}}(\rdd)}
	\lesssim
	e^{C\sum_{|\g|=2}\|\partial^{\g}\Phi\|_{W^{\fy,\fy}_{1\otimes v_{d+\ep,d+\ep}}}}.
	\ee
	Now, we deal with the first estimate $E_1$. Let $\Psi=\psi\otimes \psi$ with $\psi\in C_c^{\fy}(\rd)$, such that $\text{supp}\Psi\subset B(0,\d)$
	for sufficiently small $\d>0$, satisfying that
	\be
	(\text{supp}\eta_{k,l}^*+B(0,\d)) \subset B(k, r)\times B(l, r) \ \ \ r>0
	\ee
	and the orthogonality property
		\be
	B(k, r) \cap B(\tilde{k}, r)=\emptyset,\ \ \ \text{for all}\ k\neq \tilde{k}.
	\ee
	Note that
	\be
	|\scrV_{\Psi}\eta_{0,0}^*(z_1,z_2,\z_1,\z_2)|\lesssim \chi_{B(0,r)}(z_1)\chi_{B(0,r)}(z_2)\lan \z_1\ran^{-\scrL}\lan \z_2\ran^{-\scrL},
	\ee
	and
	\be
	\begin{split}
		&|\scrV_{\Psi}\big(\eta_{k,l}^{*}(x,\xi)e^{2\pi i(\nabla_{x} \Phi(k,l)\cdot x+\nabla_{\xi} \Phi(k,l)\cdot \xi)}\big)(z,\z)|
		\\
		= &
		|\scrV_{\Psi}\big(M_{(\nabla_{x} \Phi(k,l), \nabla_{\xi} \Phi(k,l))}T_{(k,l)}\eta_{0,0}^*\big)(z,\z)|
		= 
		|\scrV_{\Psi}\eta_{0,0}^*(z_1-k,z_2-l,\z_1-\nabla_{x} \Phi(k,l),\z_2-\nabla_{\xi} \Phi(k,l))|.
	\end{split}
	\ee
	From this, we find that
	\ben\label{thm-FIO-M1-1}
	\begin{split}
&|\scrV_{\Psi}\big(\eta_{k,l}^{*}(x,\xi)e^{2\pi i(\nabla_{x} \Phi(k,l)\cdot x+\nabla_{\xi} \Phi(k,l)\cdot \xi)}\big)(z,\z)|
\\
\lesssim & 
\chi_{B(k,r)}(z_1)\chi_{B(l,r)}(z_2)\lan \z_1-\nabla_{x} \Phi(k,l)\ran^{-\scrL}\lan \z_2-\nabla_{\xi} \Phi(k,l)\ran^{-\scrL}.
\end{split}
\een
	
	Using the orthogonality property, we obtain
	\be
	\begin{split}
		E_1=&\|\sum_{(k,l)\in \G}\scrV_{\Psi}\big(\eta_{k,l}^{*}(x,\xi)e^{2\pi i(\nabla_{x} \Phi(k,l)\cdot x+\nabla_{\xi} \Phi(k,l)\cdot \xi)}\big)(z,\z)\|_{L^{1,1,\fy,\fy}_{\z_1,z_1,\z_2,z_2}}
		\\
		\lesssim &
		\|\sum_{(k,l)\in \G}\chi_{B(k,r)}(z_1)\chi_{B(l,r)}(z_2)\lan \z_1-\nabla_{x} \Phi(k,l)\ran^{-\scrL}\lan \z_2-\nabla_{\xi} \Phi(k,l)\ran^{-\scrL}\|_{L^{1,1,\fy,\fy}_{\z_1,z_1,\z_2,z_2}}
		\\
		\sim &
		\|\sum_{(k,l)\in \G}\chi_{B(k,r)}(z_1)\chi_{B(l,r)}(z_2)\lan \z_2-\nabla_{\xi} \Phi(k,l)\ran^{-\scrL}\|_{L^{1,\fy,\fy}_{z_1,\z_2,z_2}}
		\\
		\sim &
		\|\sum_{(k,l)\in \G}\chi_{B(l,r)}(z_2)\lan \z_2-\nabla_{\xi} \Phi(k,l)\ran^{-\scrL}\|_{L^{\fy,\fy}_{\z_2,z_2}}
			\\
		= & 
		\|\sum_{l\in \La}\chi_{B(l,r)}(z_2)\sum_{k\in \La}\lan \z_2-\nabla_{\xi} \Phi(k,l)\ran^{-\scrL}\|_{L^{\fy,\fy}_{\z_2,z_2}}.
	\end{split}
	\ee	
	Using the separation property, we conclude that for some $\d>0$ independent of $l$
	\be
	\bigcup_{k\in \La}B(\nabla_{\xi} \Phi(k,l),\d)\subset \rd.
	\ee
    From this, we obtain that
	\be
	\begin{split}
		\sum_{k\in \La}\lan \z_2-\nabla_{\xi} \Phi(k,l)\ran^{-\scrL}
		\sim &
		\sum_{k\in \La}\int_{B(\nabla_{\xi} \Phi(k,l),\d)}\lan \z_2-\z_1\ran^{-\scrL}d\z_1
		\\
		= &
		\int_{\bigcup_{k\in \La}B(\nabla_{\xi} \Phi(k,l),\d)}\lan \z_2-\z_1\ran^{-\scrL}d\z_1
		\\
		\leq &
		\int_{\rd}\lan \z_2-\z_1\ran^{-\scrL}d\z_1\lesssim 1.
	\end{split}
	\ee
	The desired estimate of $E_1$ follows by
		\be
	\begin{split}
		E_1\lesssim \|\sum_{l\in \La}\chi_{B(l,r)}(z_2)\sum_{k\in \La}\lan \z_2-\nabla_{\xi} \Phi(k,l)\ran^{-\scrL}\|_{L^{\fy,\fy}_{\z_2,z_2}}
		\lesssim 
	\|\sum_{l\in \La}\chi_{B(l,r)}(z_2)\|_{L^{\fy}_{z_2}}\lesssim 1.
	\end{split}
	\ee	
	
	Next, we turn to the estimate of $E_2$.
	Using the definition of $W^{\fy,\fy}_{1\otimes v_{d+\ep,d+\ep}}(\rdd)$, we find that
	\be
	\begin{split}
		E_2
		= &
		\bigg\| \sum_{(k,l)\in \G}\eta_{k,l}(x,\xi)e^{2\pi i(\Phi(x,\xi)-\nabla_{x} \Phi(k,l)\cdot x-\nabla_{\xi} \Phi(k,l)\cdot \xi)}\bigg\|_{W^{\fy,\fy}_{1\otimes v_{d+\ep,d+\ep}}(\rdd)}
		\\
		\sim &
		\sup_{k,l}\bigg\|\eta_{k,l}(x,\xi)e^{2\pi i(\Phi(x,\xi)-\nabla_{x} \Phi(k,l)\cdot x-\nabla_{\xi} \Phi(k,l)\cdot \xi)}\bigg\|_{W^{\fy,\fy}_{1\otimes v_{d+\ep,d+\ep}}(\rdd)}
		\\
		= &
		\sup_{k,l}\bigg\|\eta_{0,0}(x,\xi)e^{2\pi i(\Phi(x+k,\xi+l)-\nabla_{x} \Phi(k,l)\cdot x-\nabla_{\xi} \Phi(k,l)\cdot \xi)}\bigg\|_{W^{\fy,\fy}_{1\otimes v_{d+\ep,d+\ep}}(\rdd)}
		\\
		= &
		\sup_{k,l}\bigg\|\eta_{0,0}e^{2\pi i\tau_{k,l}}\bigg\|_{W^{\fy,\fy}_{1\otimes v_{d+\ep,d+\ep}}(\rdd)},
	\end{split}
	\ee
	where
	\be
	\begin{split}
		\tau_{k,l}:= &
		\Phi(x+k,\xi+l)-\Phi(k,l)-\nabla_{x} \Phi(k,l)\cdot x-\nabla_{\xi} \Phi(k,l)\cdot \xi
		\\
		= &
		2\sum_{|\g|=2}\frac{(x,\xi)^{\g}}{\g !}\int_{0}^1(1-t)(\partial^{\g}\Phi)(k+tx,l+t\xi)dt.
	\end{split}
	\ee
	Using the product algebra property of $W^{\fy,\fy}_{1\otimes v_{d+\ep,d+\ep}}(\rdd)$ (see Proposition \ref{pp-agp-W}), we derive that
	\be
	\begin{split}
		\bigg\|\eta_{0,0}e^{2\pi i\tau_{k,l}}\bigg\|_{W^{\fy,\fy}_{1\otimes v_{d+\ep,d+\ep}}(\rdd)}
		= 
		\bigg\|\eta_{0,0}e^{2\pi i\eta_{0,0}^*\tau_{k,l}}\bigg\|_{W^{\fy,\fy}_{1\otimes v_{d+\ep,d+\ep}}(\rdd)}
		\lesssim
		e^{C\|\eta_{0,0}^*\tau_{k,l}\|_{W^{\fy,\fy}_{1\otimes v_{d+\ep,d+\ep}}(\rdd)}}.
	\end{split}
	\ee
        Let $\phi$ be a $C_c^{\fy}(\rdd)$ function satisfying $\phi=1$ on $\text{supp}\eta_{0,0}^{*}$.
	The desired estimate follows by
	\be
	\begin{split}
		&\|\eta_{0,0}^*\tau_{k,l}\|_{W^{\fy,\fy}_{1\otimes v_{d+\ep,d+\ep}}(\rdd)}
		\\
		= &
		2\bigg\|\sum_{|\g|=2}\frac{(x,\xi)^{\g}}{\g !}\eta_{0,0}^*(x,\xi)\int_{0}^1(1-t)\phi(tx,t\xi)(\partial^{\g}\Phi)(k+tx,l+t\xi)dt\bigg\|_{W^{\fy,\fy}_{1\otimes
				v_{d+\ep,d+\ep}}(\rdd)}
		\\
		\lesssim &
		\sum_{|\g|=2}\bigg\|\int_{0}^1(1-t)\phi(tx,t\xi)(\partial^{\g}\Phi)(k+tx,l+t\xi)dt\bigg\|_{W^{\fy,\fy}_{1\otimes
				v_{d+\ep,d+\ep}}(\rdd)}
		\\
		\lesssim &
		\sum_{|\g|=2}\sup_{t\in (0,1)}\|\phi(tx,t\xi)(\partial^{\g}\Phi)(k+tx,l+t\xi)\|_{W^{\fy,\fy}_{1\otimes
				v_{d+\ep,d+\ep}}(\rdd)}
		\\
		\lesssim &
		\sum_{|\g|=2}\|\phi(x,\xi)(\partial^{\g}\Phi)(k+x,l+\xi)\|_{W^{\fy,\fy}_{1\otimes
				v_{d+\ep,d+\ep}}(\rdd)}
		\lesssim 
		\sum_{|\g|=2}\|\partial^{\g}\Phi\|_{W^{\fy,\fy}_{1\otimes
				v_{d+\ep,d+\ep}}(\rdd)}.
	\end{split}
	\ee
	With the estimate of $e^{i\Phi}$, the estimate of $\|T_{\s,\Phi}\|_{\calL(\mfid)}$ follows by Proposition \ref{pp-KFIO-Mi-Mf}. 
\end{proof}

\begin{theorem}[FIO on $\mf$]\label{thm-FIO-Mf}
	Let $\Phi$ be a real-valued $C^2(\rdd)$ phase function satisfying
    the $(0,0,0)$-growth condition and uniform separation condition of $\xi$-type, 
	we have
	\be
    e^{2\pi i\Phi}\in M^{1,1,\fy,\fy}(c_2),\ \ \ 
	\|e^{2\pi i\Phi}\|_{M^{1,1,\fy,\fy}(c_2)}\lesssim e^{C\sum_{|\g|=2}\|\partial^{\g}\Phi\|_{W^{\fy,\fy}_{1\otimes v_{d+\ep,d+\ep}}}}.
	\ee
	Moreover,
	for every $\s\in \msjdd$, the Fourier integral operator $T_{\s,\Phi}\in \calL(\mfd)$ with the estimate of operator norm
	\be
	\|T_{\s,\Phi}\|_{\calL(\mfd)}\lesssim \|\s\|_{\msjdd}e^{C\sum_{|\g|=2}\|\partial^{\g}\Phi\|_{W^{\fy,\fy}_{1\otimes v_{d+\ep,d+\ep}}}}.
	\ee
\end{theorem}
\begin{proof}
	Let $\tilde{\Phi}(x,\xi)=\Phi(\xi,x)$. Observe that $\tilde{\Phi}$ satisfies all the assumptions in Theorem \ref{thm-FIO-M1}. 
	Then the desired conclusions follow by Proposition \ref{pp-KFIO-Mi-Mf}.
\end{proof}

\begin{theorem}[FIO on $M^p$]\label{thm-FIO-Mp}
	Let $\Phi$ be a real-valued $C^2(\rdd)$ phase function satisfying 
    the $(0,0,0)$-growth condition and uniform separation condition,
	we have 
    \be
    e^{2\pi i\Phi}\in M^{1,1,\fy,\fy}(c_1)\cap M^{1,1,\fy,\fy}(c_2)
    \ee
    with
	\be
	\|e^{2\pi i\Phi}\|_{M^{1,1,\fy,\fy}(c_1)}+ \|e^{2\pi i\Phi}\|_{M^{1,1,\fy,\fy}(c_2)}
	\lesssim e^{\sum_{C|\g|=2}\|\partial^{\g}\Phi\|_{W^{\fy,\fy}_{1\otimes v_{d+\ep,d+\ep}}}}.
	\ee
	Moreover,
	for every $\s\in \msjdd$, the Fourier integral operator $T_{\s,\Phi}\in \calL(M^p(\rd))$ with the estimate of operator norm
	\be
	\|T_{\s,\Phi}\|_{\calL(M^p(\rd))}\lesssim \|\s\|_{\msjdd}e^{\sum_{C|\g|=2}\|\partial^{\g}\Phi\|_{W^{\fy,\fy}_{1\otimes v_{d+\ep,d+\ep}}}}.
	\ee
\end{theorem}
\begin{proof}
	The desired conclusions follow by Theorem \ref{thm-FIO-M1} and Theorem \ref{thm-FIO-Mf} and a standard complex interpolation argument.
\end{proof}

\subsection{Boundedness on $\msj$ and $M^{1,\fy}$}

\begin{theorem}[FIO on $\msj$]\label{thm-FIO-Mfi}
	Let $\Phi$ be a real-valued $C^2(\rdd)$ function satisfying 
    the $(\al,0,0)$-growth condition for some $\al\in [0,1]$.
	If $s_1, s_2\geq 0$ and the following relation holds
	\be
	l^{\fy}_{\frac{s_1}{1-\al}}\subset l^1\ \ \text{for}\ \al\in [0,1),
	\ee
	then $T_{\s,\Phi}\in \calL(\msj)$ for all $\s\in M^{\fy,1}_{v_{s_1,s_2}\otimes 1}$, satisfying 
	\be
	\|T_{\s,\Phi}\|_{\calL(\msjd)}
	\lesssim 
	\|\s\|_{M^{\fy,1}_{v_{s_1,s_2}\otimes 1}}e^{C\sum_{|\g|=2}\|\partial^{\g}\Phi\|_{W^{\fy,\fy}_{1\otimes v_{d+\ep,d+\ep}}}}.
	\ee
\end{theorem}
\begin{proof}
	Using Proposition \ref{pp-KFIO-Mfi}, we only need to verify that for
	$s_1\geq 0$ with $\al=1$ or $l^{\fy}_{\frac{s_1}{1-\al}}\subset l^1$ with $\al\in [0,1)$, the following estimate is valid:
	\be
	\|v_{s_1}^{-1}e^{2\pi i\Phi}\|_{M^{1,\fy,1,\fy}(c_3)}\lesssim e^{C\sum_{|\g|=2}\|\partial^{\g}\Phi\|_{W^{\fy,\fy}_{1\otimes v_{d+\ep,d+\ep}}}}.
	\ee
	Applying the decomposition as in the proof of Theorem \ref{thm-FIO-M1}, we write
	\be
	\begin{split}
		&\sum_{(k,l)\in \G}\eta_{k,l}(x,\xi)e^{2\pi i\Phi(x,\xi)}
		\\
		= &
		\left(\sum_{(k,l)\in \G}\eta_{k,l}(x,\xi)e^{2\pi i(\Phi(x,\xi)-\nabla_{x} \Phi(k,l)\cdot x-\nabla_{\xi} \Phi(k,l)\cdot \xi)}\right)
		\left(\sum_{(k,l)\in \G}\eta_{k,l}^{*}(x,\xi)e^{2\pi i(\nabla_{x} \Phi(k,l)\cdot x+\nabla_{\xi} \Phi(k,l)\cdot \xi)}\right).
	\end{split}
	\ee
	The desired estimate follows by 
	the product inequality 
	\be
	M^{1,\fy,1,\fy}(c_3) \cdot W^{\fy,\fy}_{1\otimes v_{d+\ep,d+\ep}}(\rdd)\subset M^{1,\fy,1,\fy}(c_3),
	\ee
	and
	the following two estimates
	\be
	E_1=\bigg\|v_{s_1}^{-1}(x)\sum_{(k,l)\in \G}\eta_{k,l}^{*}(x,\xi)e^{2\pi i(\nabla_{x} \Phi(k,l)\cdot x+\nabla_{\xi} \Phi(k,l)\cdot \xi)}\bigg\|_{M^{1,\fy,1,\fy}(c_3)}\lesssim 1,
	\ee
	and 
	\be
	E_2=\bigg\| \sum_{(k,l)\in \G}\eta_{k,l}(x,\xi)e^{2\pi i(\Phi(x,\xi)-\nabla_{x} \Phi(k,l)\cdot x-\nabla_{\xi} \Phi(k,l)\cdot \xi)}\bigg\|_{W^{\fy,\fy}_{1\otimes v_{d+\ep,d+\ep}}(\rdd)}
	\lesssim
	e^{C\sum_{|\g|=2}\|\partial^{\g}\Phi\|_{W^{\fy,\fy}_{1\otimes v_{d+\ep,d+\ep}}}}.
	\ee
	The estimate of $E_2$ has been established in the proof of Theorem \ref{thm-FIO-M1}. We only need to deal with the $E_1$ term.
	
	Using \eqref{thm-FIO-M1-1}, we find that
	\be
	\begin{split}
		E_1
		\sim &
		\|v_{s_1}^{-1}(z_1)\sum_{(k,l)\in \G}\scrV_{\Psi}\big(\eta_{k,l}^{*}(x,\xi)e^{2\pi i(\nabla_{x} \Phi(k,l)\cdot x+\nabla_{\xi} \Phi(k,l)\cdot \xi)}\big)(z,\z)\|_{L^{1,\fy,1,\fy}_{\z_2,z_1,\z_1,z_2}}
		\\
		\lesssim &
		\|v_{s_1}^{-1}(z_1)\sum_{(k,l)\in \G}\chi_{B(k,r)}(z_1)\chi_{B(l,r)}(z_2)\lan \z_1-\nabla_{x} \Phi(k,l)\ran^{-\scrL}\lan \z_2-\nabla_{\xi} \Phi(k,l)\ran^{-\scrL}\|_{L^{1,\fy,1,\fy}_{\z_2,z_1,\z_1,z_2}}    	
		\\
		\sim &
	    \|v_{s_1}^{-1}(z_1)\sum_{(k,l)\in \G}\chi_{B(k,r)}(z_1)\chi_{B(l,r)}(z_2)\lan \z_1-\nabla_{x} \Phi(k,l)\ran^{-\scrL}\|_{L^{\fy,1,\fy}_{z_1,\z_1,z_2}}    	
		\\
		= &
		\|\sup_{k\in \La}v_{s_1}^{-1}(k)\sum_{l\in \La}\chi_{B(l,r)}(z_2)\lan \z_1-\nabla_{x} \Phi(k,l)\ran^{-\scrL}\|_{L^{1,\fy}_{\z_1,z_2}}.
	\end{split}
	\ee
	Take $\scrL_{\al}=d+\ep$ for $\al=1$ and $\scrL_{\al}=\frac{s_1}{1-\al}$ for $\al\in [0,1)$. We have
	\be
	\begin{split}
		v_{s_1}^{-1}(k)\lan \z_1-\nabla_{x} \Phi(k,l)\ran^{-\scrL_{\al}}
		\lesssim &
		v_{s_1}^{-1}(k)\lan \nabla_{x} \Phi(0,l)-\nabla_{x} \Phi(k,l)\ran^{\scrL_{\al}}\lan \z_1-\nabla_{x} \Phi(0,l)\ran^{-\scrL_{\al}}
		\\
		\lesssim &
		v_{s_1}^{-1}(k)\lan k\ran^{(1-\al)\scrL_{\al}}\lan \z_1-\nabla_{x} \Phi(0,l)\ran^{-\scrL_{\al}}
				\\
		\lesssim &
		\lan \z_1-\nabla_{x} \Phi(0,l)\ran^{-\scrL_{\al}}\in L^1_{\z_1},
	\end{split}
	\ee
	where in the last inclusion relation we use the fact $l^{\fy}_{\frac{s_1}{1-\al}}\subset l^1$ with $\al\in [0,1)$.
	From this, the desired estimate of $E_1$ follows by
	\be
	\begin{split}
		E_1
		\lesssim &
		\|\sup_{k\in \La}v_{s_1}^{-1}(k)\sum_{l\in \La}\chi_{B(l,r)}(z_2)\lan \z_1-\nabla_{x} \Phi(k,l)\ran^{-\scrL_{\al}}\|_{L^{1,\fy}_{\z_1,z_2}}
		\\
		\lesssim &
		\|\sum_{l\in \La}\chi_{B(l,r)}(z_2)\lan \z_1-\nabla_{x} \Phi(0,l)\ran^{-\scrL_{\al}}\|_{L^{1,\fy}_{\z_1,z_2}}
		\lesssim 
		\|\sum_{l\in \La}\chi_{B(l,r)}(z_2)\|_{L^{\fy}_{z_2}}\lesssim 1.
	\end{split}
	\ee
We have now completed this proof.
\end{proof}

\begin{theorem}[FIO on $\mif$]\label{thm-FIO-Mif}
    Let $\Phi$ be a real-valued $C^2(\rdd)$ phase function satisfying 
    the $(\al,0,0)$-growth condition for some $\al\in [0,1]$ and uniform separation condition.
	If $s_1, s_2\geq 0$ and
	the following relation holds
	\be
	l^{\fy}_{\frac{s_1}{1-\al}+s_2}\subset l^1\ \ \text{for}\ \al\in [0,1),
	\ee
	then $T_{\s,\Phi}\in \calL(\mif)$ for all $\s\in M^{\fy,1}_{v_{s_1,s_2}\otimes 1}$, satisfying 
	\be
	\|T_{\s,\Phi}\|_{\calL(\mif)}
	\lesssim 
	\|\s\|_{M^{\fy,1}_{v_{s_1,s_2}\otimes 1}}e^{C\sum_{|\g|=2}\|\partial^{\g}\Phi\|_{W^{\fy,\fy}_{1\otimes v_{d+\ep,d+\ep}}}}.
	\ee
\end{theorem}
\begin{proof}
    Using Proposition \ref{pp-KFIO-Mif}, we only need to verify that for 
	$s_1,s_2\geq 0$ and $l^{\fy}_{\frac{s_1}{1-\al}+s_2}\subset l^1$ with $\al\in [0,1)$, the following estimate is valid:
	\be
	\|v_{s_1,s_2}^{-1}e^{2\pi i\Phi}\|_{M^{1,\fy,1,\fy}(c_4)}\lesssim e^{C\sum_{|\g|=2}\|\partial^{\g}\Phi\|_{W^{\fy,\fy}_{1\otimes v_{d+\ep,d+\ep}}}}.
	\ee
	Using 
	the product inequality 
	\be
	M^{1,\fy,1,\fy}(c_4) \cdot W^{\fy,\fy}_{1\otimes v_{d+\ep,d+\ep}}(\rdd)\subset M^{1,\fy,1,\fy}(c_4),
	\ee
	and the similar argument as in the proof of Theorem \ref{thm-FIO-Mfi}, we only need to verify that
	\be
	E_1=\bigg\|v_{s_1,s_2}^{-1}(x,\xi)\sum_{(k,l)\in \G}\eta_{k,l}^{*}(x,\xi)e^{2\pi i(\nabla_{x} \Phi(k,l)\cdot x+\nabla_{\xi} \Phi(k,l)\cdot \xi)}\bigg\|_{M^{1,\fy,1,\fy}(c_4)}\lesssim 1.
	\ee	
	Using \eqref{thm-FIO-M1-1}, we find that
	\be
	\begin{split}
		E_1
		= &
		\|v_{s_1,s_2}^{-1}(z_1,z_2)\sum_{(k,l)\in \G}\scrV_{\Psi}\big(\eta_{k,l}^{*}(x,\xi)e^{2\pi i(\nabla_{x} \Phi(k,l)\cdot x+\nabla_{\xi} \Phi(k,l)\cdot \xi)}\big)(z,\z)\|_{L^{1,\fy,1,\fy}_{z_1,\z_2,z_2,\z_1}}
		\\
		\lesssim &
		\|v_{s_1,s_2}^{-1}(z_1,z_2)\sum_{(k,l)\in \G}\chi_{B(k,r)}(z_1)\chi_{B(l,r)}(z_2)\lan \z_1-\nabla_{x} \Phi(k,l)\ran^{-\scrL}\lan \z_2-\nabla_{\xi} \Phi(k,l)\ran^{-\scrL}\|_{L^{1,\fy,1,\fy}_{z_1,\z_2,z_2,\z_1}}    	
		\\
		\lesssim &
		\|v_{s_2}^{-1}(z_2)\sum_{k\in \La}v_{s_1}^{-1}(k)\sum_{l\in \La}\chi_{B(l,r)}(z_2)\lan \z_1-\nabla_{x} \Phi(k,l)\ran^{-\scrL}\lan \z_2-\nabla_{\xi} \Phi(k,l)\ran^{-\scrL}\|_{L^{\fy,1,\fy}_{\z_2,z_2,\z_1}}    	
	\end{split}
	\ee
	As in the proof of Theorem \ref{thm-FIO-Mfi}, we obtain 
	\be
	\begin{split}
		v_{s_1}^{-1}(k)\lan \z_1-\nabla_{x} \Phi(k,l)\ran^{-\scrL_{\al}}
		\lesssim &
		\lan \z_1-\nabla_{x} \Phi(0,l)\ran^{-\scrL_{\al}},
	\end{split}
	\ee
    where $\scrL_{\al}=d+\ep$ for $\al=1$ and $\scrL_{\al}=\frac{s_1}{1-\al}$ for $\al\in [0,1)$. 
	From this, $E_1$ can be dominated by
	\be
	\begin{split}
		E_1
		\lesssim &
		\|v_{s_2}^{-1}(z_2)\sum_{k\in \La}v_{s_1}^{-1}(k)\sum_{l\in \La}\chi_{B(l,r)}(z_2)\lan \z_1-\nabla_{x} \Phi(k,l)\ran^{-\scrL_{\al}}\lan \z_2-\nabla_{\xi} \Phi(k,l)\ran^{-\scrL_{\al}}\|_{L^{\fy,1,\fy}_{\z_2,z_2,\z_1}}    
		\\
		\lesssim &
		\|v_{s_2}^{-1}(z_2)\sum_{l\in \La}\chi_{B(l,r)}(z_2)\lan \z_1-\nabla_{x} \Phi(0,l)\ran^{-\scrL_{\al}}\sum_{k\in \La}\lan \z_2-\nabla_{\xi} \Phi(k,l)\ran^{-\scrL}\|_{L^{\fy,1,\fy}_{\z_2,z_2,\z_1}}    
				\\
		\lesssim &
		\|v_{s_2}^{-1}(z_2)\sum_{l\in \La}\chi_{B(l,r)}(z_2)\lan \z_1-\nabla_{x} \Phi(0,l)\ran^{-\scrL_{\al}}\|_{L^{1,\fy}_{z_2,\z_1}}    
		\\
		\lesssim &
		\|\sum_{l\in \La}v_{s_2}^{-1}(l)\lan \z_1-\nabla_{x} \Phi(0,l)\ran^{-\scrL_{\al}}\|_{L^{\fy}_{\z_1}}.    
	\end{split}
	\ee
	Then the desired estimate of $E_1$ follows by the claim
	\be
		\|\sum_{l\in \La}v_{s_2}^{-1}(l)\lan \z_1-\nabla_{x} \Phi(0,l)\ran^{-\scrL_{\al}}\|_{L^{\fy}_{\z_1}}\lesssim 1.
	\ee
	
	If $\al=1$, we have
	\be
	\sum_{l\in \La}v_{s_2}^{-1}(l)\lan \z_1-\nabla_{x} \Phi(0,l)\ran^{-\scrL_{\al}}
	\lesssim 
	\sum_{l\in \La}\lan \z_1-\nabla_{x} \Phi(0,l)\ran^{-(d+\ep)}\lesssim 1,\ \ \forall \z_1\in \rd,
	\ee
	where in the last inequality we use the separation property of $\Phi$.
	
	If $\al\in [0,1)$, we have $l^{\fy}_{\frac{s_1}{1-\al}+s_2}\subset l^1$. This is equivalent to $\frac{s_1}{1-\al}+s_2>d$.
	There exists suitable $r\in [1,\fy]$ such that
	\be
	\frac{s_1}{1-\al}>d/r,\ \ s_2>d/r'.
	\ee
	From this, we conclude that for any $\z_1\in \rd$, the following estimate holds:
	\be
	\begin{split}
			\sum_{l\in \La}v_{s_2}^{-1}(l)\lan \z_1-\nabla_{x} \Phi(0,l)\ran^{-\scrL_{\al}}
		= &
		\sum_{l\in \La}\lan l\ran^{-s_2}\lan \z_1-\nabla_{x} \Phi(0,l)\ran^{\frac{-s_1}{1-\al}}
		\\
		\leq &
		\|(\lan l\ran^{-s_2})_{l}\|_{l^{r'}}\|(\lan \z_1-\nabla_{x} \Phi(0,l)\ran^{\frac{-s_1}{1-\al}})_{l}\|_{l^{r}}\lesssim 1,
	\end{split}
	\ee
	where the final inequality follows from the separation property of $\Phi$.
	This completes the proof.
\end{proof}

\subsection{The proof of Theorem \ref{thm-FIO-Mpq}: sufficiency part}
The case of $\al=1$ or $p=q$ is obvious. We only give the proof for $\al\in [0,1)$ and $p\neq q$.

\textbf{Case 1: $p>q$, $\al\in [0,1)$.}
In this case, we have $\frac{s_1}{1-\al}>d(1/q-1/p)$ and $s_2\geq 0$.
The desired conclusion follows by a standard complex interpolation argument with the fact
that for all $\s\in \msj$ and $\Phi$ satisfies all the conditions in Theorem \ref{thm-FIO-Mpq},
\be
T_{\al,\Phi}\in \calL(M^{\fy,1}, M^{\fy,1}_{v^{-1}_{\frac{s_1}{1-\theta}}\otimes 1})\ \ \text{and}\ \ T_{\al,\Phi}\in \calL(M^r, M^r)
\ee
where
\be
\frac{1-\theta}{\fy}+\frac{\theta}{r}=\frac{1}{p},\ \ \frac{1-\theta}{1}+\frac{\theta}{r}=\frac{1}{q},\ \ 1-\theta=\frac{1}{q}-\frac{1}{p}.
\ee

\textbf{Case 2: $p<q$, $\al\in [0,1)$.}
In this case, we have $\frac{s_1}{1-\al}+s_2>d(1/p-1/q)$ and $s_1, s_2\geq 0$.
The desired conclusion follows by a standard complex interpolation argument with the fact
that for all $\s\in \msj$ and $\Phi$ satisfies all the conditions in Theorem \ref{thm-FIO-Mpq},
\be
T_{\al,\Phi}\in 
\calL(M^{1,\fy}_{1\otimes v_{\frac{s_2}{1-\theta}}}, M^{1,\fy}_{v^{-1}_{\frac{s_1}{1-\theta}}\otimes 1})\cap \calL(M^r, M^r),
\ee
where
\be
\frac{1-\theta}{1}+\frac{\theta}{r}=\frac{1}{p},\ \ \frac{1-\theta}{\fy}+\frac{\theta}{r}=\frac{1}{q},\ \ 1-\theta=\frac{1}{p}-\frac{1}{q}.
\ee

\subsection{The proof of Theorem \ref{thm-FIO-Mpq}: necessity part}

\begin{lemma}\label{lm-sep}
	Let $\al\in [0,1)$, $\mu(x)=\lan x\ran^{2-\al}$ and $k_{\al}=\lan k\ran^{\frac{\al}{1-\al}}k$.
	For any fixed sufficiently  small $r>0$,
	there exists a sufficiently large constant $R>0$ such that 
    \be
    |\nabla \mu(k_{\al})-(2-\al)k|<r,\ \ \ \ 
    B(\nabla \mu(k_{\al}), r)\cap B(\nabla \mu(l_{\al}), r)=\emptyset,\ \ \ \   k\neq l,\ |k|,|l|>R.
    \ee
\end{lemma}
\begin{proof}
	A direct calculation yields that $\nabla \mu(x)=(2-\al)\lan x\ran^{-\al}x$.
	Thus,
	\be
	\nabla \mu(k_{\al})=(2-\al)\lan k_{\al}\ran^{-\al}k_{\al}=(2-\al)\lan \lan k\ran^{\frac{\al}{1-\al}}k\ran^{-\al}\lan k\ran^{\frac{\al}{1-\al}}k.
	\ee
	Note that
	\be
	\begin{split}
		\lan \lan k\ran^{\frac{\al}{1-\al}}k\ran^{\al}
		=
		\big(1+\lan k\ran^{\frac{2\al}{1-\al}}|k|^2\big)^{\frac{\al}{2}}
		\leq 
		\big(\lan k\ran^{\frac{2\al}{1-\al}}\lan k\ran^2\big)^{\frac{\al}{2}}=\lan k\ran^{\frac{\al}{1-\al}}.
	\end{split}
	\ee
	From this, we conclude that
	\be
	\begin{split}
		|\lan \lan k\ran^{\frac{\al}{1-\al}}k\ran^{-\al}\lan k\ran^{\frac{\al}{1-\al}}-1|
		= &
		\lan \lan k\ran^{\frac{\al}{1-\al}}k\ran^{-\al}\lan k\ran^{\frac{\al}{1-\al}}-1
		\\
		\leq &
		|k|^{\frac{-\al}{1-\al}}\lan k\ran^{\frac{\al}{1-\al}}-1
		\\
		\leq &
		C\ln(|k|^{\frac{-\al}{1-\al}}\lan k\ran^{\frac{\al}{1-\al}})
        =
        \frac{C \al}{1-\al}\ln((1+|k|^{-2})^{\frac{1}{2}})
        \\
        = &
        \frac{C \al}{2(1-\al)}\ln(1+|k|^{-2})\leq \frac{C \al}{1-\al}|k|^{-2}.
	\end{split}
	\ee
	Thus, 
	\be
	\begin{split}
		|\nabla \mu(k_{\al})-(2-\al)k|
		= &
		(2-\al)|\lan \lan k\ran^{\frac{\al}{1-\al}}k\ran^{-\al}\lan k\ran^{\frac{\al}{1-\al}}-1|\cdot |k|
		\\
		\leq &
		\frac{C\al(2-\al)}{1-\al}|k|^{-1}.
	\end{split}
	\ee
	For sufficiently large $R$, we have the following estimate for $|x|>R$:
	\be
	|\nabla \mu(k_{\al})-(2-\al)k|<r,
	\ee
	where $r$ is a fixed sufficiently small constant. From this, we further obtain that
	\be
	B(\nabla \mu(k_{\al}), r)\subset B((2-\al)k, 2r).
	\ee
	Therefore, for $k\neq l$ and sufficiently small $r$, we conclude that
	\be
	B(\nabla \mu(k_{\al}), r)\cap B(\nabla \mu(l_{\al}), r)\subset B((2-\al)k, 2r)\cap B((2-\al)l, 2r)=\emptyset.
	\ee
\end{proof}

\begin{lemma}[Estimates of specific functions]\label{lm-spf}
	Let $\al\in [0,1)$, $\mu(x)=\lan x\ran^{2-\al}$.
	For every fixed nonnegative truncated (only finite nonzero items) sequence $(a_k)_{k\in \zd}$, there
	exists a Schwartz function $F$ such that the following two estimates hold
	\bn
	\item
	$\|F\|_{M^{p,q}_{v_{s_1,s_2}}}\sim \|(a_k)_k\|_{l^p_{\frac{s_1}{1-\al}}}$
	\item
	$\|e^{2\pi i\mu}F\|_{M^{p,q}_{v_{s_1,s_2}}}\sim \|(a_k)_k\|_{l^q_{\frac{s_1}{1-\al}+s_2}}$
	\en
\end{lemma}
\begin{proof}
Let $h$ be a $C_c^{\fy}(\rd)$ real-valued function supported on $B(0,\d)$ with sufficiently small $\d>0$.
Denote $k_{\al}=\lan k\ran^{\frac{\al}{1-\al}}k$ and $h_k^{\al}(x)=h(x-k_{\al})$.
For any nonnegative truncated (only finite nonzero items) sequence $(a_k)_{k\in \zd}$, we set 
\be
F=\sum_{k}a_kh_k^{\al}.
\ee
Choose $\va\in C_c^{\fy}(\rd)$ as a window function supported on $B(0,\d)$.
We find that
\be
|V_{\va}h_k^{\al}(x,\xi)|=|V_{\va}h(x-k_{\al},\xi)|=|V_{\va}h(x-k_{\al},\xi)\chi_{B(k_{\al},2\d)}(x)|.
\ee
A direct calculation yields that
\be
\begin{split}
	\|F\|_{M^{p,q}_{v_{s_1,s_2}}}
	= &
	\big\|\sum_{k}a_kV_{\va}h_k^{\al}(x,\xi)\big\|_{L^{p,q}_{v_{s_1,s_2}}}
	\\
	\sim &
	\big\|\big(\sum_k a_k^p\big\|V_{\va}h_k^{\al}(x,\xi)\big\|^p_{L^p_{v_{s_1}}}\big)^{1/p}\big\|_{L^q_{v_{s_2}}}
		\\
	= &
	\big\|\big(\sum_k a_k^p\big\|V_{\va}h(x-k_{\al},\xi)\chi_{B(k_{\al},2\d)}(x)\big\|^p_{L^p_{v_{s_1}}}\big)^{1/p}\big\|_{L^q_{v_{s_2}}}
	\\
	\sim &
	\big\|\big(\sum_k a_k^pv_{s_1}(k_{\al})^p\|V_{\va}h(x-k_{\al},\xi)\chi_{B(k_{\al},2\d)}(x)\big\|^p_{L^p}\big)^{1/p}\big\|_{L^q_{v_{s_2}}}
	\\
	= &
	\big\|\big(\sum_k a_k^pv_{s_1}(k_{\al})^p\big\|V_{\va}h(x,\xi)\big\|^p_{L^p}\big)^{1/p}\big\|_{L^q_{v_{s_2}}}\sim\|(a_k)_{k\in \zd}\|_{l^p_{\frac{s_1}{1-\al}}}.
\end{split}
\ee
Next, we turn to the estimate of $\|e^{2\pi i\mu}F\|_{M^{p,q}_{v_{s_1,s_2}}}$. 
Let $h^*$ be a $C_c^{\fy}(\rd)$ real-valued function supported on $B(0,2\d)$, satisfying $h^*\equiv 1$ on $\text{supp}h$.
Denote $h_k^{\al,*}(x)=h^*(x-k_{\al})$.
Write
\be
\begin{split}
	e^{2\pi i\mu}F
	=
	\sum_{k}h_k^{\al,*}e^{2\pi i (\mu-\nabla \mu(k_{\al})x)}\cdot 
	\sum_k a_kh_k^{\al}e^{2\pi i\nabla \mu(k_{\al})x}= :m \cdot G.
\end{split}
\ee
A direct calculation yields that $m\in \calC^N\subset \msj_{1\otimes v_{|s_2|}}$, where
\be
\calC_N :=\{g\in C^{\fy}:\ \ |\partial^{\g}g|\leq C_{\g}\ \text{for all}\ \g\}.
\ee
From this and the product inequality $M^{p,q}_{v_{s_1,s_2}}\cdot \msj_{1\otimes v_{|s_2|}}\subset M^{p,q}_{v_{s_1,s_2}}$
\be
\|e^{2\pi i\mu}F\|_{M^{p,q}_{v_{s_1,s_2}}}
\lesssim 
\|m\|_{\msj_{1\otimes v_{|s_2|}}}\|G\|_{M^{p,q}_{v_{s_1,s_2}}}\lesssim \|G\|_{M^{p,q}_{v_{s_1,s_2}}}.
\ee
On the other hand, we write
\be
G=\sum_k a_kh_k^{\al}e^{2\pi i\nabla \mu(k_{\al})x}
=\overline{\sum_{k}h_k^{\al,*}e^{2\pi i (\mu-\nabla \mu(k_{\al})x)}}e^{2\pi i\mu}F=\overline{m}e^{2\pi i\mu}F
\ee
which implies that
\be
\|G\|_{M^{p,q}_{v_{s_1,s_2}}}=\|\overline{m}e^{2\pi i\mu}F\|_{M^{p,q}_{v_{s_1,s_2}}}
\lesssim
\|\overline{m}\|_{\msj_{1\otimes v_{|s_2|}}}
\|e^{2\pi i\mu}F\|_{M^{p,q}_{v_{s_1,s_2}}}
\lesssim
\|e^{2\pi i\mu}F\|_{M^{p,q}_{v_{s_1,s_2}}}.
\ee
Finally, we turn to the estimate of $\|G\|_{M^{p,q}_{v_{s_1,s_2}}}$. Note that
\be
|V_{\va}(M_{\nabla\mu(k_{\al})}h_k^{\al})(x,\xi)|=|V_{\va}h(x-k_{\al},\xi-\nabla\mu(k_{\al}))|=|V_{\va}h(x-k_{\al},\xi-\nabla\mu(k_{\al}))\chi_{B(k_{\al},2\d)}(x)|.
\ee
We have
\be
\begin{split}
	|V_{\va}G(x,\xi)|
	\leq 
	\sum_k a_k|V_{\va}(M_{\nabla\mu(k_{\al})}h_k^{\al})(x,\xi)|
	\lesssim 
	\sum_k a_k\chi_{B(k_{\al},2\d)}(x)\lan \xi-\nabla\mu(k_{\al})\ran^{-\scrL}.
\end{split}
\ee
Then the upper estimate of $\|G\|_{M^{p,q}_{v_{s_1,s_2}}}$ follows by
\be
\begin{split}
	\|G\|_{M^{p,q}_{v_{s_1,s_2}}}
	\lesssim &
	\big\|\sum_k a_k\chi_{B(k_{\al},2\d)}(x)\lan \xi-\nabla\mu(k_{\al})\ran^{-\scrL}\big\|_{L^{p,q}_{v_{s_1,s_2}}}
	\\
	\lesssim &
	\big\|\big(\sum_k a_k^p \big\|\chi_{B(k_{\al},2\d)}(x)\lan \xi-\nabla\mu(k_{\al})\ran^{-\scrL}\big\|^p_{L^p_{v_{s_1}}}\big)^{1/p}\big\|_{L^q_{v_{s_2}}}
		\\
	\lesssim &
	\big\|\big(\sum_k a_k^p v_{s_1}(k_{\al})^p \lan \xi-\nabla\mu(k_{\al})\ran^{-\scrL}\big)^{1/p}\big\|_{L^q_{v_{s_2}}}
	\\
	\lesssim &
		\big\|\sup_k a_k v_{s_1}(k_{\al}) \lan \xi-\nabla\mu(k_{\al})\ran^{-\scrL}\big\|_{L^q_{v_{s_2}}},
\end{split}
\ee
where in the last inequality we use Lemma \ref{lm-sep} and the fact
\be
\sum_k\lan \xi-\nabla\mu(k_{\al})\ran^{-\scrL}\lesssim 1,\ \ \ \text{uniformly for all}\ \xi\in \rd.
\ee
We continue the estimate by
\be
\begin{split}
	\|G\|_{M^{p,q}_{v_{s_1,s_2}}}
	\lesssim &
	\big\|\sup_k a_k v_{s_1}(k_{\al}) \lan \xi-\nabla\mu(k_{\al})\ran^{-\scrL}\big\|_{L^q_{v_{s_2}}}
	\\
	= &
	\big\|\sup_k a_k v_{s_1}(k_{\al})v_{s_2}(\xi) \lan \xi-\nabla\mu(k_{\al})\ran^{-\scrL}\big\|_{L^q}
		\\
	\lesssim &
	\big\|\sup_k a_k v_{s_1}(k_{\al})v_{s_2}(\nabla\mu(k_{\al})) v_{|s_2|}(\xi-\nabla\mu(k_{\al}))\lan \xi-\nabla\mu(k_{\al})\ran^{-\scrL}\big\|_{L^q}
			\\
	\lesssim &
	\big\|\sup_k a_k v_{s_1}(k_{\al})v_{s_2}(\nabla\mu(k_{\al})) \lan \xi-\nabla\mu(k_{\al})\ran^{-\scrL}\big\|_{L^q}
				\\
	\lesssim &
	\big\|\big(\sum_k a_k^q v_{s_1}(k_{\al})^qv_{s_2}(\nabla\mu(k_{\al}))^q \lan \xi-\nabla\mu(k_{\al})\ran^{-\scrL}\big)^{1/q}\big\|_{L^q}
	\\
	\lesssim &
	\big(\sum_k a_k^q v_{s_1}(k_{\al})^qv_{s_2}(\nabla\mu(k_{\al}))^q \big)^{1/q}.
\end{split}
\ee
Using Lemma \ref{lm-sep} and $k_{\al}=\lan k\ran^{\frac{\al}{1-\al}}k$, we have
\begin{equation*}
    v_{s_1}(k_{\al})\sim \lan k\ran^{\frac{s_1}{1-\al}}=v_{\frac{s_1}{1-\al}}(k),
    \ \ 
    v_{s_2}(\nabla\mu(k_{\al}))\sim v_{s_2}((2-\al)k)\sim v_{s_2}(k).
\end{equation*}
Then, the upper estimate of $\|G\|_{M^{p,q}_{v_{s_1,s_2}}}$ follows by
\begin{align*}
    \|G\|_{M^{p,q}_{v_{s_1,s_2}}}
    \lesssim
    \big(\sum_k a_k^q v_{s_1}(k_{\al})^qv_{s_2}(\nabla\mu(k_{\al}))^q \big)^{1/q}\sim \|(a_k)_k\|_{l^q_{\frac{s_1}{1-\al}+s_2}}.
\end{align*}

On the other hand, then the lower estimate of $\|G\|_{M^{p,q}_{v_{s_1,s_2}}}$ follows by
\be
\begin{split}
	\|G\|_{M^{p,q}_{v_{s_1,s_2}}}
	\sim &
	\big\|\big(\sum_k a_k^p\big\|V_{\va}h(x-k_{\al},\xi-\nabla\mu(k_{\al}))\chi_{B(k_{\al},2\d)}(x)\big\|^p_{L^{p}_{v_{s_1}}}\big)^{1/p}\big\|_{L^q_{v_{s_2}}}
	\\
	\sim &
	\big\|\big(\sum_k a_k^pv_{s_1}(k_{\al})^p\big\|V_{\va}h(x,\xi-\nabla\mu(k_{\al}))\big\|^p_{L^{p}}\big)^{1/p}\big\|_{L^q_{v_{s_2}}}
	\\
	\gtrsim &
	\big(\sum_k \big\|a_kv_{s_1}(k_{\al}) \big\|V_{\va}h(x,\xi-\nabla\mu(k_{\al}))\big\|_{L^{p}}\chi_{B(\nabla\mu(k_{\al}),\d)}(\xi)\big\|^q_{L^q_{v_{s_2}}}\big)^{1/q}
	\\
	\sim &
	\big(\sum_k a_k^qv_{s_1}(k_{\al})^qv_{s_2}(\nabla\mu(k_{\al}))^q
	\big\| \big\|V_{\va}h(x,\xi-\nabla\mu(k_{\al}))\big\|_{L^{p}}\chi_{B(\nabla\mu(k_{\al}),\d)}(\xi)\big\|^q_{L^q}\big)^{1/q}
	\\
	= &
		\big(\sum_k a_k^qv_{s_1}(k_{\al})^qv_{s_2}(\nabla\mu(k_{\al}))^q
	\big\| \big\|V_{\va}h(x,\xi)\big\|_{L^{p}}\chi_{B(0,\d)}(\xi)\big\|^q_{L^q}\big)^{1/q}
	\\
	\sim &
	\big(\sum_k a_k^qv_{s_1}(k_{\al})^qv_{s_2}(\nabla\mu(k_{\al}))^q\big)^{1/q}\sim \|(a_k)_k\|_{l^q_{\frac{s_1}{1-\al}+s_2}}.
\end{split}
\ee
The desired estimate of $\|G\|_{M^{p,q}_{v_{s_1,s_2}}}$ follows by the above upper and lower estimates.
\end{proof}

Now, we are in a position to prove the necessity part of Theorem \ref{thm-FIO-Mpq}. Take the phase function
\be
\Phi(x,\xi)=\lan x\ran^{2-\al}+x\cdot\xi= :\mu(x)+x\cdot\xi.
\ee
Note that $\Phi$ satisfies all the assumptions in the statement (1) of Theorem \ref{thm-FIO-Mpq}. We have
\be
T_{\s,\Phi}\in \calL(\mpq)\ \ \text{for all}\ \ \s\in \msj_{v_{s_1,s_2}}.
\ee
This is equivalent to 
\be
T_{\s,\Phi}\in \calL(\mpq_{1\otimes v_{s_2}},\mpq_{v_{s_1}^{-1}\otimes 1})\ \ \text{for all}\ \ \s\in \msj.
\ee
From this and the fact $1\in \msj$, we obtain that
\be
T_{1,\Phi}\in \calL(\mpq_{1\otimes v_{s_2}},\mpq_{v_{s_1}^{-1}\otimes 1}).
\ee
Note that for $f\in \calS$,
\be
T_{1,\Phi}f(x)=\int_{\rdd}\hat{f}(\xi)e^{2\pi i\Phi(x,\xi)}d\xi=e^{2\pi i\mu(x)}\int_{\rdd}\hat{f}(\xi)e^{2\pi ix\cdot\xi}d\xi=e^{2\pi i\mu(x)}f(x).
\ee
We deduce that for $f\in \calS$,
\be
\|e^{2\pi i\mu}f\|_{\mpq_{v_{s_1}^{-1}\otimes 1}}\lesssim \|f\|_{\mpq_{1\otimes v_{s_2}}}.
\ee
This further implies that for $f\in \calS$,
\be
\|e^{2\pi i\mu}f\|_{\mpq}\lesssim \|f\|_{\mpq_{v_{s_1}\otimes v_{s_2}}}\ \ \text{and}\ \ \ \|f\|_{\mpq}\lesssim \|e^{2\pi i\mu}f\|_{\mpq_{v_{s_1}\otimes v_{s_2}}}.
\ee
Using this and Lemma \ref{lm-spf}, for any truncated sequence $(a_k)_k$
we obtain
\be
\|(a_k)_k\|_{l^q}\lesssim \|(a_k)_k\|_{l^p_{\frac{s_1}{1-\al}}} \ \ \ \text{and}\ \ \ 
\|(a_k)_k\|_{l^p}\lesssim \|(a_k)_k\|_{l^q_{\frac{s_1}{1-\al}+s_2}},
\ee
which implies the desired embedding relations for $\al\in [0,1)$
\be
l^p_{\frac{s_1}{1-\al}}\subset l^q,\ \ \ \ l^q_{\frac{s_1}{1-\al}+s_2}\subset l^p.
\ee
Moreover, take $f=M_kh$ with $h\in \calS\bs \{0\}$, we obtain that
\be
\begin{split}
	\|e^{2\pi i\mu}h\|_{\mpq_{v_{s_1}^{-1}\otimes 1}}
	= &
	\|e^{2\pi i\mu}M_kh\|_{\mpq_{v_{s_1}^{-1}\otimes 1}}
	\\
	\lesssim &
	\|M_kh\|_{\mpq_{1\otimes v_{s_2}}}
	=
	\|V_{\va}h(x,\xi-k)v_{s_2}(\xi)\|_{L^{p,q}}
	\\
	\lesssim &
	\|V_{\va}h(x,\xi-k)v_{|s_2|}(\xi-k)\|_{L^{p,q}}v_{s_2}(k)\sim v_{s_2}(k),
\end{split}
\ee
which implies that $s_2\geq 0$.

If $\al=1$, the desired condition $s_1, s_2\geq 0$ follows by
\be
\|f\|_{\mpq}
=
\|e^{-2\pi i\mu}e^{2\pi i\mu}f\|_{\mpq}
\lesssim 
\|e^{2\pi i\mu}f\|_{\mpq}\lesssim \|f\|_{\mpq_{v_{s_1}\otimes v_{s_2}}}.
\ee
Here, we use the boundedness $e^{i\mu(D)}\in \calL(W^{q,p})$ (see \cite[Corollary 1.7]{GuoZhao2020JFA}) with the fact $W^{q,p}=\scrF M^{p,q}$.

\subsection{An improvement of the boundedness of $e^{i\mu(D)}$ on $W^{p,q}$}

As an application of Theorem \ref{thm-FIO-Mpq}, we give an improvement of the boundedness on Wiener amalgam spaces
of unimodular Fourier multipliers.

\begin{theorem}\label{thm-UFO-W}
    Let $1\leq p,q\leq \infty$.
  Let $\mu$ be a real-valued $C^2(\rd)$ function satisfying
  \be
  |\nabla \mu(\xi)|\lesssim \lan \xi\ran^{1-\al}
  \ \
  \partial^{\g}\mu \in W^{\infty,\fy}_{1\otimes v_{d+\ep}}\ (|\g|=2),
  \ee
  for some $\ep>0$, $\al\in [0,1]$.
  If $s_1,s_2\geq 0$, and 
  the following conditions hold for $\al \in [0,1)$:
  \be
  s_2>d(1-\al)(1/p-1/q)\ \text{if}\ 1/p>1/q,\ \ s_2+(1-\al)s_1>d(1-\al)(1/q-1/p)\ \text{if}\ 1/q>1/p,
  \ee
  we have the boundedness
  \be
  e^{i\mu(D)}: W^{p,q}_{v_{s_1,s_2}} \rightarrow \wpq
  \ee
  with
  \be
  \|e^{i\mu(D)}\|_{\calL(W^{p,q}_{v_{s_1,s_2}}, \wpq)}\lesssim e^{\sum_{|\g|=2}\|\partial^{\g}\mu\|_{W^{\fy,\fy}_{1\otimes v_{d+\ep}}}}.
  \ee
\end{theorem}
\begin{proof}
    Take $\Phi(x,\xi)=\frac{\mu(x)}{2\pi}+x\cdot \xi$. 
    Note that the phase function $\Phi$ satisfies the $(\alpha,0,0)$-growth condition and uniform separation condition.
    Using Theorem \ref{thm-FIO-Mpq}, we obtain the boundedness
    \be
    T_{1,\Phi}: M^{q,p}_{1\otimes v_{s_1}} \longrightarrow M^{q,p}_{v_{s_2}^{-1}\otimes 1}
    \ee
    with
    \be
    \|T_{1,\Phi}\|_{\calL(M^{q,p}_{1\otimes v_{s_1}},M^{q,p}_{v_{s_2}^{-1}\otimes 1})}\lesssim e^{\sum_{|\g|=2}\|\partial^{\g}\mu\|_{W^{\fy,\fy}_{1\otimes v_{d+\ep}}}}.
    \ee
    Then, the desired conclusion follows by the equivalent relation
    \be
    T_{1,\Phi}: M^{q,p}_{1\otimes v_{s_1}} \longrightarrow M^{q,p}_{v_{s_2}^{-1}\otimes 1}
    \Longleftrightarrow
    T_{1,\Phi}: M^{q,p}_{v_{s_2,s_1}} \longrightarrow M^{q,p}
    \Longleftrightarrow
    e^{i\mu(D)}: W^{p,q}_{v_{s_1,s_2}} \rightarrow \wpq,
    \ee
    where we use the facts $\scrF M^{q,p}_{v_{s_2,s_1}}=W^{p,q}_{v_{s_1,s_2}}$ and
    \be
    T_{1,\Phi}f(x)=e^{i\mu(x)}\int_{\rd}\hat{f}(\xi)e^{2\pi ix\cdot \xi}d\xi=e^{i\mu(x)}f(x).
    \ee
\end{proof}

\begin{remark}
    Theorem \ref{thm-UFO-W} is an improvement of \cite[Theorem 1.3]{GuoZhao2020JFA}, 
    because this new theorem allows the function $f$ to be in the Wiener amalgam space with mixed weight $v_{s_1}\otimes v_{s_2}$. 
    Meanwhile, the growth condition $|\nabla \mu(\xi)|\lesssim \lan \xi\ran^{1-\al}$
    is weaker than the corresponding condition 
    $\langle \xi \rangle^{\al-1}\nabla \mu(\xi) \in (W^{\infty,\fy}_{1+v_{d+\ep}})^d$ used in \cite[Theorem 1.3]{GuoZhao2020JFA}.
\end{remark}

\section{Boundedness on $\mpq$ without separation properties}
The careful reader may have noticed that Theorem \ref{thm-FIO-Mfi} does not require the assumption of separated property
of phase function.
An interesting question is how to establish the sharp boundedness result without separated property.
We give a complete answer in Theorem \ref{thm-FIO-Mpq-nsp}.

\subsection{The proof of Theorem \ref{thm-FIO-Mpq-nsp}: sufficiency part}
 We first deal with the case of $p= q$.
 Note that 
 \be
 l^{\fy}_{\frac{s_1}{1-\al}-\frac{\al d}{(1-\al)p}}\subset l^q\Longleftrightarrow l^{\fy}_{s_1}\subset l^p
 \ee
 for $\al\in [0,1)$, $p=q$.
 We only need to establish the corresponding boundedness results under the assumptions that 
 $\al=0$,
 $l^{\fy}_{v_{s_1}}\subset l^p$ and $l^q_{v_{s_2}}\subset l^1$.
 First, we consider the boundedness on $M^1(\rd)$ and $M^{\fy}(\rd)$, then the desired
conclusion follows by a complex interpolation argument.

In order to establish the $M^1(\rd)$ boundeness, we assume that $l^{\fy}_{v_{s_1}}\subset l^1$ and $s_2=0$.
By Proposition \ref{pp-KFIO-Mi-Mf}, we only need to verify that
\be
\|v_{s_1,0}^{-1} e^{2\pi i\Phi}\|_{M^{1,1,\fy,\fy}(c_1)}\lesssim e^{C\sum_{|\g|=2}\|\partial^{\g}\Phi\|_{W^{\fy,\fy}_{1\otimes v_{d+\ep,d+\ep}}}}.
\ee
As in the proof of Theorem \ref{thm-FIO-M1},
we note that the following estimate of $E_2$ remains valid independently of the separated property of $\Phi$
\be
E_2=\bigg\| \sum_{(k,l)\in \G}\eta_{k,l}(x,\xi)e^{2\pi i(\Phi(x,\xi)-\nabla_{x} \Phi(k,l)\cdot x-\nabla_{\xi} \Phi(k,l)\cdot \xi)}\bigg\|_{W^{\fy,\fy}_{1\otimes v_{d+\ep,d+\ep}}(\rdd)}
\lesssim
e^{C\sum_{|\g|=2}\|\partial^{\g}\Phi\|_{W^{\fy,\fy}_{1\otimes v_{d+\ep,d+\ep}}}},
\ee
then the desired estimate follows by
\be
\begin{split}
	E_1
	= &
	\bigg\|v_{s_1}^{-1}(x)\sum_{(k,l)\in \G}\eta_{k,l}^{*}(x,\xi)e^{2\pi i(\nabla_{x} \Phi(k,l)\cdot x+\nabla_{\xi} \Phi(k,l)\cdot \xi)}\bigg\|_{M^{1,1,\fy,\fy}(c_1)}
	\\
	= &	
	\|v_{s_1}^{-1}(z_1)\sum_{(k,l)\in \G}\scrV_{\Psi}\big(\eta_{k,l}^{*}(x,\xi)e^{2\pi i(\nabla_{x} \Phi(k,l)\cdot x+\nabla_{\xi} \Phi(k,l)\cdot \xi)}\big)(z,\z)\|_{L^{1,1,\fy,\fy}_{\z_1,z_1,\z_2,z_2}}
	\\
	\lesssim &
	\|v_{s_1}^{-1}(z_1)\sum_{(k,l)\in \G}\chi_{B(k,r)}(z_1)\chi_{B(l,r)}(z_2)\lan \z_1-\nabla_{x} \Phi(k,l)\ran^{-\scrL}\lan \z_2-\nabla_{\xi} \Phi(k,l)\ran^{-\scrL}\|_{L^{1,1,\fy,\fy}_{\z_1,z_1,\z_2,z_2}}
	\\
	\lesssim &
	\|\sum_{(k,l)\in \G}\chi_{B(k,r)}(z_1)\chi_{B(l,r)}(z_2)\lan \z_1-\nabla_{x} \Phi(k,l)\ran^{-\scrL}\lan \z_2-\nabla_{\xi} \Phi(k,l)\ran^{-\scrL}\|_{L^{1,\fy,\fy,\fy}_{\z_1,z_1,\z_2,z_2}}
	\\
	\lesssim &
	\|\sum_{(k,l)\in \G}\chi_{B(k,r)}(z_1)\chi_{B(l,r)}(z_2)\lan \z_2-\nabla_{\xi} \Phi(k,l)\ran^{-\scrL}\|_{L^{\fy,\fy,\fy}_{z_1,\z_2,z_2}}\lesssim 1,
\end{split}
\ee
where in the second inequality we utilize the assumption $l^{\fy}_{v_{s_1}}\subset l^1$.

In order to establish the $M^{\fy}(\rd)$ boundeness, we assume $l^{\fy}_{v_{s_2}}\subset l^1$ and $s_1=0$. 
Let $\tilde{\Phi}(x,\xi)=\Phi(\xi,x)$. 
By Proposition \ref{pp-KFIO-Mi-Mf}, it suffices to verify that
\be
\|v_{s_2,0}^{-1} e^{2\pi i\tilde{\Phi}}\|_{M^{1,1,\fy,\fy}(c_1)}\lesssim e^{C\sum_{|\g|=2}\|\partial^{\g}\tilde{\Phi}\|_{W^{\fy,\fy}_{1\otimes v_{d+\ep,d+\ep}}}},
\ee
which has already been established above, since $\tilde{\Phi}$ also satisfy the 
$(0,0,0)$-growth condition.

Now, we are in a position to establish the $\mpqd$ boundedness for $p= q$.
Recall that $l^{\fy}_{v_{s_1}}\subset l^p$ and $l^q_{v_{s_2}}\subset l^1$.
For $p=q\in (1,\fy)$, there exists a constant $\theta\in (0,1)$ such that
\be
\frac{1-\theta}{\fy}+\frac{\theta}{1}=\frac{1}{p},\ \ \theta=\frac{1}{p}.
\ee
Note that
\be
l^{\fy}_{v_{s_1}}\subset l^p\Longrightarrow l^{\fy}_{v_{ps_1}}\subset l^1,\ \ \ 
l^q_{s_2}\subset l^1 \Longrightarrow
l^{\fy}_{v_{q's_2}}\subset l^1.
\ee

The desired conclusion follows by a standard complex interpolation argument with the fact
that for all $\s\in \msj$ and $\Phi$ satisfies all the conditions in Theorem \ref{thm-FIO-Mpq-nsp}, we have
\be
T_{\s,\Phi}\in \calL(M^{1}, M^{1}_{v^{-1}_{s_1p}\otimes 1})\cap \calL(\mf_{1\otimes v_{s_2q'}}, \mf).
\ee

Next, we turn to the case $p>q$. Note that the $\msjd$ boundedness has been established in Theorem \ref{thm-FIO-Mfi}.
We only need to consider the case $(p,q)\neq (\fy, 1)$.

\textbf{Case 1: $\al=1$.}
In this case, we have $l^q_{v_{s_2}}\subset l^1$ and $l^{\fy}_{v_{s_1}}\subset l^p$.

If $q\neq 1$, $p\neq \fy$, there exists $r\in (1,\fy)$ and $\theta\in (0,1)$ such that
\be
\frac{1-\theta}{\fy}+\frac{\theta}{r}=\frac{1}{p},\ \ \frac{1-\theta}{1}+\frac{\theta}{r}=\frac{1}{q},\ \ 1-\theta=\frac{1}{q}-\frac{1}{p},\ \ \theta=\frac{r}{p}=\frac{r'}{q'}.
\ee
Note that
\be
l^q_{v_{s_2}}\subset l^1\Longrightarrow l^r_{v_{\frac{q's_2}{r'}}}\subset l^1,\ \ \ 
l^{\fy}_{v_{s_1}}\subset l^p \Longrightarrow
l^{\fy}_{v_{\frac{ps_1}{r}}}\subset l^r.
\ee
The desired conclusion follows by a standard complex interpolation argument with the fact
that for all $\s\in \msj$ and $\Phi$ satisfies all the conditions in Theorem \ref{thm-FIO-Mpq-nsp}, we have
\be
T_{\s,\Phi}\in \calL(\msj, \msj)\cap \calL(M^r_{1\otimes v_{\frac{q's_2}{r'}}}, M^r_{v_{\frac{ps_1}{r}}^{-1}\otimes 1}).
\ee

If $q=1$, $p\neq \fy$, 
there exists $\theta\in (0,1)$ such that
\be
\frac{1-\theta}{\fy}+\frac{\theta}{1}=\frac{1}{p},\ \ \frac{1-\theta}{1}+\frac{\theta}{1}=\frac{1}{1},\ \ \theta=\frac{1}{p}.
\ee
Note that
\be
l^{\fy}_{v_{s_1}}\subset l^p \Longrightarrow
l^{\fy}_{v_{ps_1}}\subset l^1.
\ee
The desired conclusion follows by a standard complex interpolation argument with the fact 
\be
T_{\s,\Phi}\in \calL(\msj, \msj)\cap \calL(M^1, M^1_{v_{ps_1}^{-1}\otimes 1}).
\ee

If $p= \fy$, $q\neq 1$, 
there exists $\theta\in (0,1)$ such that
\be
\frac{1-\theta}{\fy}+\frac{\theta}{\fy}=\frac{1}{\fy},\ \ \frac{1-\theta}{1}+\frac{\theta}{\fy}=\frac{1}{q},\ \ \theta=\frac{1}{q'}.
\ee
Note that
\be
l^q_{v_{s_2}}\subset l^1 \Longrightarrow
l^{\fy}_{v_{q's_2}}\subset l^1.
\ee
The desired conclusion follows by a standard complex interpolation argument with the fact 
\be
T_{\s,\Phi}\in \calL(\msj, \msj)\cap \calL(M^{\fy}_{1\otimes v_{q's_2}}, M^{\fy}).
\ee

\textbf{Case 2: $\al\in [0,1)$.}
In this case, we have 
\be
\begin{split}
	l^{\fy}_{\frac{s_1}{1-\al}-\frac{\al d}{(1-\al)p}}\subset l^q
	\Longrightarrow
	s_1>d\bigg(\frac{1-\al}{q}+\frac{\al}{p}\bigg)
	\Longrightarrow
	s_1>\frac{d}{p}
	\Longrightarrow
	l^{\fy}_{v_{s_1}}\subset l^p.
\end{split}
\ee
Therefore, we only need to establish the boundedness results under the assumptions that
$l^{\fy}_{\frac{s_1}{1-\al}-\frac{\al d}{(1-\al)p}}\subset l^q$ 
and $l^q_{v_{s_2}}\subset l^1$.

If $q\neq 1$, $p\neq \fy$, we have
\be
s_1=d\bigg(\frac{1-\al}{q}+\frac{\al}{p}\bigg)+\ep_1,\ \ \ s_2=\frac{d}{q'}+\ep_2
\ee
for some $\ep_1, \ep_2>0$.
There exists $r\in (1,\fy)$ and $\theta\in (0,1)$ such that
\be
\frac{1-\theta}{\fy}+\frac{\theta}{r}=\frac{1}{p},\ \ \frac{1-\theta}{1}+\frac{\theta}{r}=\frac{1}{q},\ \ 1-\theta=\frac{1}{q}-\frac{1}{p},\ \ \theta=\frac{r}{p}=\frac{r'}{q'}.
\ee
Note that for sufficiently small $\ep>0$ and
\be
\tilde{s_1}=(1-\al)d+\ep,\ \ \tilde{s_2}=0,\ \ \hat{s_1}=\frac{d}{r}+\ep,\ \ \hat{s_2}=\frac{d}{r'}+\ep,
\ee
we have
\be
\begin{split}
	(1-\theta)\tilde{s_1}+\theta \hat{s_1}
	= &
	\bigg(\frac{1}{q}-\frac{1}{p}\bigg)(1-\al)d+\bigg(\frac{1}{q}-\frac{1}{p}\bigg)\ep+\frac{d}{p}+\frac{r}{p}\ep
	\\
	= &
	d\bigg(\frac{1-\al}{q}+\frac{\al}{p}\bigg)+\bigg(\frac{1}{q}-\frac{1}{p}+\frac{r}{p}\bigg)\ep\leq s_1
\end{split}
\ee
and
\be
\begin{split}
	(1-\theta)\tilde{s_2}+\theta \hat{s_2}
	= &
	\frac{r'}{q'}\bigg(\frac{d}{r'}+\ep\bigg)=\frac{d}{q'}+\frac{d}{q'}\ep\leq s_2.
\end{split}
\ee
The desired conclusion follows by a standard complex interpolation argument with the fact
that for all $\s\in \msj$ and $\Phi$ satisfies all the conditions in Theorem \ref{thm-FIO-Mpq-nsp}, we have
\be
T_{\s,\Phi}\in \calL(\msj_{1\otimes v_{\tilde{s_2}}}, \msj_{v_{\tilde{s_1}}^{-1}\otimes 1})\cap 
\calL(M^r_{1\otimes v_{\hat{s_2}}}, M^r_{v_{\hat{s_1}}^{-1}\otimes 1}).
\ee

If $q=1$, $p\neq \fy$, we have
\be
s_1=d\bigg(\frac{1-\al}{q}+\frac{\al}{p}\bigg)+\ep_1=d\bigg(1-\al+\frac{\al}{p}\bigg)+\ep_1,\ \ \ s_2\geq 0
\ee
for some $\ep_1>0$.
There exists $\theta\in (0,1)$ such that
\be
\frac{1-\theta}{\fy}+\frac{\theta}{1}=\frac{1}{p},\ \ \frac{1-\theta}{1}+\frac{\theta}{1}=\frac{1}{1},\ \ \theta=\frac{1}{p}.
\ee
Note that for sufficiently small $\ep>0$ and
\be
\tilde{s_1}=(1-\al)d+\ep,\ \ \tilde{s_2}=0,\ \ \hat{s_1}=d+\ep,\ \ \hat{s_2}=0,
\ee
we have
\be
(1-\theta)\tilde{s_1}+\theta \hat{s_1}\leq s_1,\ \ \ (1-\theta)\tilde{s_2}+\theta \hat{s_1}\leq s_2.
\ee
The desired conclusion follows by a standard complex interpolation argument with the fact 
\be
T_{\s,\Phi}\in \calL(\msj_{1\otimes v_{\tilde{s_2}}}, \msj_{v_{\tilde{s_1}}^{-1}\otimes 1})\cap 
\calL(M^1_{1\otimes v_{\hat{s_2}}}, M^1_{v_{\hat{s_1}}^{-1}\otimes 1}).
\ee

If $p= \fy$, $q\neq 1$, we have
\be
s_1=\frac{d(1-\al)}{q}+\ep_1,\ \ \ s_2=\frac{d}{q'}+\ep_2
\ee
for some $\ep_1, \ep_2>0$.

There exists $\theta\in (0,1)$ such that
\be
\frac{1-\theta}{\fy}+\frac{\theta}{\fy}=\frac{1}{\fy},\ \ \frac{1-\theta}{1}+\frac{\theta}{\fy}=\frac{1}{q},\ \ \theta=\frac{1}{q'}.
\ee
Note that for sufficiently small $\ep>0$ and
\be
\tilde{s_1}=(1-\al)d+\ep,\ \ \tilde{s_2}=0,\ \ \hat{s_1}=0,\ \ \hat{s_2}=d+\ep,
\ee
we have
\be
(1-\theta)\tilde{s_1}+\theta \hat{s_1}\leq s_1,\ \ \ (1-\theta)\tilde{s_2}+\theta \hat{s_1}\leq s_2.
\ee
The desired conclusion follows by a standard complex interpolation argument with the fact 
\be
T_{\s,\Phi}\in \calL(\msj_{1\otimes v_{\tilde{s_2}}}, \msj_{v_{\tilde{s_1}}^{-1}\otimes 1})\cap 
\calL(M^{\fy}_{1\otimes v_{\hat{s_2}}}, M^{\fy}_{v_{\hat{s_1}}^{-1}\otimes 1}).
\ee

   Finally, we turn to $\mpqd$ boundedness for $p< q$. 
   In this case, we have
\be
\begin{split}
	l^{\fy}_{v_{s_1}}\subset l^p
	\Longrightarrow
	s_1>\frac{d}{p}
	\Longrightarrow
	s_1>d\bigg(\frac{1-\al}{q}+\frac{\al}{p}\bigg)
	\Longrightarrow
	l^{\fy}_{\frac{s_1}{1-\al}-\frac{\al d}{(1-\al)p}}\subset l^q
\end{split}
\ee
for $\al\in [0,1)$.
Therefore, we only need to establish the boundedness results under the assumptions that   
$l^{\fy}_{v_{s_1}}\subset l^p$ and $l^q_{v_{s_2}}\subset l^1$.
This implies that
\be
l^{\fy}\subset l^q_{v_{\frac{ps_1}{q}}^{-1}}\subset l^p_{v_{s_1}^{-1}},\ \ \ \ M^q_{v_{\frac{ps_1}{q}}^{-1}\otimes 1}\subset \mpq_{v_{s_1}^{-1}\otimes 1},
\ee
in which the embedding relation of modulation spaces follows by \cite[Theorem 1.5]{GuoChenFanZhao2019MMJ}.
The desired $\mpqd$ boundedness follows by
\be
\begin{split}
	\|T_{\s,\Phi}f\|_{\mpq_{v_{s_1}^{-1}\otimes 1}}
	\lesssim 
	\|T_{\s,\Phi}f\|_{M^q_{v_{\frac{ps_1}{q}}^{-1}\otimes 1}}
	\lesssim 
	\|f\|_{M^q_{1\otimes v_{s_2}}}
	\lesssim
	\|f\|_{\mpq_{1\otimes v_{s_2}}},
\end{split}
\ee
where the second inequality follows by the case $p=q$ proved before, with the facts $L^{\fy}_{v_{\frac{ps_1}{q}}}\subset l^q$ and $l^q_{v_{s_2}}\subset l^1$.

\begin{remark}
	Note that for $p\leq q$ the boundedness on $\mpq$ of FIOs without separation property, is independent
	of the growth condition associated with $\al\in [0,1]$.
\end{remark}

\subsection{The proof of Theorem \ref{thm-FIO-Mpq-nsp}: necessity part}

Take $\s \equiv	1\in \msjdd$, $\Phi(x,\xi)=\va(\xi)\in 1\otimes C_c^{\fy}(\rd)$, and $f\in \calS(\rd)$, such as
\be
T_{1,\Phi}f(x)=\int_{\rd}\hat{f}(\xi)e^{2\pi i\va(\xi)}d\xi=1.
\ee
If (1) holds, we have
$T_{1,\Phi}\in \calL(\mpq_{1\otimes v_{s_2}}(\rd), \mpq_{v_{s_1}^{-1}\otimes 1}(\rd))$ , then $1\in \mpq_{v_{s_1}^{-1}\otimes 1}(\rd)$. 
By a direct calculation, for $g\in \calS(\rd)$ we have
\be
|V_g(1)(x,\xi)|=|\scrF(\bar{g})(\xi)|\in L^{p,q}_{v_{s_1}^{-1}\otimes 1}.
\ee
This implies that $1\in L^p_{v_{s_1}^{-1}}$, which is equivalent to $l^{\fy}_{s_1}\subset l^p$.

Take $\s \equiv	1\in \msjdd$, $\Phi(x,\xi)=\va(x)=\lan x\ran^{2-\al}$, and $f\in \calS(\rd)$. A direct calculation yields that
\be
T_{\s,\Phi}f(x)=e^{2\pi i\lan x\ran^{2-\al}}\int_{\rd}\hat{f}(\xi)d\xi=e^{2\pi i\lan x\ran^{2-\al}} f(0).
\ee
If $T_{\s,\Phi}\in \calL(\mpq_{1\otimes v_{s_2}}(\rd), \mpq_{v_{s_1}^{-1}\otimes 1}(\rd))$, we have 
\be
|f(0)|\cdot \|e^{2\pi i\lan x\ran^{2-\al}}\|_{\mpq_{v_{s_1}^{-1}\otimes 1}(\rd)}= \|T_{\s,\Phi}f\|_{\mpq_{v_{s_1}^{-1}\otimes 1}(\rd))}\lesssim \|f\|_{\mpq_{1\otimes v_{s_2}}(\rd)}.
\ee
This implies that $\|e^{2\pi i\lan x\ran^{2-\al}}\|_{\mpq_{v_{s_1}^{-1}\otimes 1}(\rd)}\lesssim 1$, and 
\begin{equation*}
    |f(0)|\lesssim 
    \|f\|_{\mpq_{1\otimes v_{s_2}}(\rd)},\ \ \ \ \forall f\in \calS(\rd).
\end{equation*}
For $a_k\geq 0$,
we take $f=\sum_{k}a_kM_k\scrF^{-1}\va$ 
with some $\va\in C_c^{\fy}$ satisfying $\text{supp}\va\subset B(0,\d)$ and $\scrF^{-1}\va(0)=1$.

Then, the inequality $|f(0)|\lesssim \|f\|_{\mpq_{1\otimes v_{s_2}}}$ implies that
\be
\|(a_k)_k\|_{l^1}=|f(0)|\lesssim \|f\|_{\mpq_{1\otimes v_{s_2}}} \sim \|(a_k)_k\|_{l^q_{v_{s_2}}}.
\ee
The embedding relation $l^q_{v_{s_2}}\subset l^1$ follows.

Moreover, we also have $\|e^{2\pi i\lan x\ran^{2-\al}}\|_{\mpq_{v_{s_1}^{-1}\otimes 1}(\rd)}\lesssim 1$. Denote $\mu(x)=\lan x\ran^{2-\al}$.
Set
\be
g_k^{\al}(x)=g(\frac{x-k_{\al}}{\lan k\ran^{\frac{\al}{1-\al}}}),
\ee
where $g$ is a smooth function satisfying that
\be
g(x)=1\ \text{for}\ |x|\leq \d,\ \ \ \ \ g(x)=0\ \text{for}\ |x|> 2\d.
\ee
Write
\be
\begin{split}
	G:=
	\sum_k g_k^{\al}(x)e^{2\pi i\nabla\mu(k_{\al})x}
	=
	\sum_k g_k^{\al}(x)e^{2\pi i\nabla\mu(k_{\al})x-\mu(x)}e^{2\pi i\mu(x)}=: m e^{2\pi i\mu(x)}.
\end{split}
\ee
Using the fact $m\in \calC^N\subset \msj$ and the product inequality $\mpq_{v_{s_1}^{-1}\otimes 1} \cdot \msj\subset \mpq_{v_{s_1}^{-1}\otimes 1}$, we conclude that
\be
\|G\|_{\mpq_{v_{s_1}^{-1}\otimes 1}(\rd)}\lesssim \|m\|_{\msj(\rd)}\|e^{2\pi i\mu(x)}\|_{\mpq_{v_{s_1}^{-1}\otimes 1}(\rd)}\lesssim 1.
\ee
Choose real-valued $\va\in C_c^{\fy}(\rd)$ as a window function supported on $B(0,\d)$.
We find that
\be
g_{k}^{\al}T_x\va=T_x\va,\ \ \ x\in B(k_{\al},\lan k\ran^{\frac{\al}{1-\al}}\d/2),
\ee
which implies that for $x\in B(k_{\al},\lan k\ran^{\frac{\al}{1-\al}}\d/2)$,
\be
\begin{split}
	|V_{\va}(M_{\nabla\mu(k_{\al})}g_k^{\al})(x,\xi)|
	= &
	|V_{\va}g_k^{\al}(x,\xi-\nabla\mu(k_{\al}))|
	\\
	= &
	|\scrF(g_{k}^{\al}(\cdot)T_x\va(\cdot))(\xi-\nabla\mu(k_{\al}))|
	\\
	= &
	|\scrF(T_x\va(\cdot))(\xi-\nabla\mu(k_{\al}))|
	=|\scrF\va(\xi-\nabla\mu(k_{\al}))|.
\end{split}
\ee
From this and the orthogonality derived by
\be
|V_{\va}(M_{\nabla\mu(k_{\al})}g_k^{\al})(x,\xi)|=|V_{\va}(M_{\nabla\mu(k_{\al})}g_k^{\al})(x,\xi)|\chi_{B(k_{\al},4\lan k\ran^{\frac{\al}{1-\al}}\d)},
\ee
for $p\neq \fy$ we conclude that
\be
\begin{split}
	\|G\|_{\mpq_{v_{s_1}^{-1}\otimes 1}(\rd)}
	\sim &
	\|\big(\sum_k\|V_{\va}(M_{\nabla\mu(k_{\al})}g_k^{\al})(x,\xi)\|^p_{L^p_{v_{s_1}^{-1}}}\big)^{1/p}\|_{L^q}
	\\
	\gtrsim &
	\|\big(\sum_k|\scrF\va(\xi-\nabla\mu(k_{\al}))|^p\|\chi_{B(k_{\al},\lan k\ran^{\frac{\al}{1-\al}}\d/2)}(x)\|^p_{L^p_{v_{s_1}^{-1}}}\big)^{1/p}\|_{L^q}
	\\
	\sim &
	\|\big(\sum_kv_{s_1}(k_{\al})^{-p}\lan k\ran^{\frac{\al d}{1-\al}}|\scrF\va(\xi-\nabla\mu(k_{\al}))|^p\big)^{1/p}\|_{L^q}
	\\
	\gtrsim &
	\big(\sum_kv_{s_1}(k_{\al})^{-q}\lan k\ran^{\frac{\al dq}{(1-\al)p}}\|\scrF\va(\xi-\nabla\mu(k_{\al}))\chi_{B(\nabla\mu(k_{\al}),\d)}(\xi)\|^q_{L^q}\big)^{1/q}
	\\
	= &
	\big(\sum_kv_{s_1}(k_{\al})^{-q}\lan k\ran^{\frac{\al dq}{(1-\al)p}}\|\scrF\va(\xi)\chi_{B(0,\d)}(\xi)\|^q_{L^q}\big)^{1/q}
	\sim
	\big\|\big(\lan k\ran^{\frac{\al d}{(1-\al)p}-\frac{s_1}{1-\al}}\big)_k\big\|_{l^q}.
\end{split}
\ee
From this, we obtain the desired embedding relation
\be
l^{\fy}_{\frac{s_1}{1-\al}-\frac{\al d}{(1-\al)p}}\subset l^q\ \text{for}\ \al\in [0,1).
\ee
The case of $p=\fy$ can also be handled by a similar way. 
Now, we have completed this proof. 

\begin{remark}\label{rmk-role-sep}
    When comparing Theorem \ref{thm-FIO-Mpq} and Theorem \ref{thm-FIO-Mpq-nsp}, one can find that in the absence
    of uniform separation condition on $\Phi$, 
    stronger conditions on $\s\in M^{\fy,1}_{v_{s_1,s_2}\otimes 1}$ are required 
    for the corresponding boundedness on $\mpq$ of $T_{\s,\Phi}$, unless the case of $(p,q)=(\fy,1)$.
    This also demonstrates the crucial role of the uniform separation condition on the phase function $\Phi$ in the boundedness on $\mpq$ of $T_{\s,\Phi}$.
\end{remark}

\section{the boundedness on $\mp$ with high growth condition}
\begin{theorem}[FIO on $\mfi$, high growth]\label{thm-FIO-12hg-M1}
	Let $\Phi$ be a real-valued $C^2(\rdd)$ function 
    satisfying 
    the $(-\fy,t_1,t_2)$-growth condition and uniform separation condition of $x$-type.
	Denote 
	\be
	A=\sum_{|\g|=2}\|\lan x\ran^{-t_1}\partial^{\g}_{x,x}\Phi\|_{W^{\fy,\fy}_{1\otimes v_{d+\ep,d+\ep}}},\ \ \ 
	B=\sum_{|\g|=2}\|\lan \xi\ran^{-t_2}\partial^{\g}_{\xi,\xi}\Phi\|_{W^{\fy,\fy}_{1\otimes v_{d+\ep,d+\ep}}},
	\ee
	and
	\be
	C=\sum_{|\g|=2}\|\partial^{\g}_{x,\xi}\Phi\|_{W^{\fy,\fy}_{1\otimes v_{d+\ep,d+\ep}}}
	\ee
	We have $\lan x\ran^{-dt_1/2}\lan \xi\ran^{-dt_2/2}e^{2\pi i\Phi(x,\xi)}\in M^{1,1,\fy,\fy}(c_1)$ with
	\ben\label{thm-FIO-12hg-M1-cd1}
	\|\lan x\ran^{-dt_1/2}\lan \xi\ran^{-dt_2/2}e^{2\pi i\Phi(x,\xi)}\|_{M^{1,1,\fy,\fy}(c_1)}
	\lesssim 
	e^{A+B+C}.
	\een
	Moreover,
	for every $\s\in \msj_{v_{\frac{dt_1}{2},\frac{dt_2}{2}}\otimes 1}(\rdd)$, the Fourier integral operator $T_{\s,\Phi}\in \calL(\mfid)$ with the estimate of operator norm
	\be
	\begin{split}
		\|T_{\s,\Phi}\|_{\calL(\mfid)}
		\lesssim  \|\s\|_{\msj_{v_{\frac{dt_1}{2},\frac{dt_2}{2}}\otimes 1}(\rdd)}e^{A+B+C}.
	\end{split}
	\ee
\end{theorem}
\begin{proof}
     We only need to verify \eqref{thm-FIO-12hg-M1-cd1},
     and then the boundedness follows from Proposition \ref{pp-KFIO-Mi-Mf}.
     By the local property of $M^{1,1,\fy,\fy}(c_1)$, we have
     \be
     \begin{split}
    \|\lan x\ran^{-dt_1/2}\lan \xi\ran^{-dt_2/2}e^{2\pi i\Phi(x,\xi)}\|_{M^{1,1,\fy,\fy}(c_1)}
    \sim &
    \sup_l \|\lan x\ran^{-dt_1/2}\lan \xi\ran^{-dt_2/2}\eta_l(\xi)e^{2\pi i\Phi(x,\xi)}\|_{M^{1,1,\fy,\fy}(c_1)}
    \\
    \sim &
    \sup_l \lan l\ran^{-dt_2/2}\|\lan x\ran^{-dt_1/2}\eta_l(\xi)e^{2\pi i\Phi(x,\xi)}\|_{M^{1,1,\fy,\fy}(c_1)}.
    \end{split}
     \ee
    From this, in order to verify \eqref{thm-FIO-12hg-M1-cd1}, we only need to establish the following estimate
    \be
    \|\lan x\ran^{-dt_1/2}\eta_l(\xi)e^{2\pi i\Phi(x,\xi)}\|_{M^{1,1,\fy,\fy}(c_1)}\lesssim \lan l\ran^{dt_2/2},\ \ \ \forall l\in \zd.
    \ee
    For fixed $l\in \zd$, we denote $\la_2=\lan l\ran^{-t_2/2}$. Using the following estimate
    \be
\begin{split}
    \|\lan x\ran^{-dt_1/2}\eta_l(\xi)e^{2\pi i\Phi(x,\xi)}\|_{M^{1,1,\fy,\fy}(c_1)}
    \sim &
    \bigg\|\scrF_{\xi}\bigg(\eta_l(\xi)V_{\va}(\lan \cdot\ran^{-dt_1/2}e^{2\pi i\Phi(\cdot,\xi)})(z_1,\z_1)\bigg)(\z_2)\bigg\|_{L^{1,1,\fy}_{\z_1,z_1,\z_2}}
    \\
    = &
    \la_2^{d}\bigg\|\scrF_{\xi}\bigg(\eta_l(\la_2\xi)V_{\va}(\lan \cdot\ran^{-dt_1/2}e^{2\pi i\Phi(\cdot,\la_2\xi)})(z_1,\z_1)\bigg)(\z_2)\bigg\|_{L^{1,1,\fy}_{\z_1,z_1,\z_2}},
\end{split}
\ee
we only need to verify
\be
\bigg\|\scrF_{\xi}\bigg(\eta_l(\la_2\xi)V_{\va}(\lan \cdot\ran^{-dt_1/2}e^{2\pi i\Phi(\cdot,\la_2\xi)})(z_1,\z_1)\bigg)(\z_2)\bigg\|_{L^{1,1,\fy}_{\z_1,z_1,\z_2}}\lesssim \lan l\ran^{dt_2}.
\ee
Write
$\eta_l(\la_2\xi)=\eta_l(\la_2\xi)\sum_{m\in A_l}\eta_m(\xi)$ where
\be
A_l=\{m\in \zd, \eta_m(\xi)\eta_l(\la_2\xi)\neq 0\},\ \ \ |A_l|\sim \lan l\ran^{d t_2/2}.
\ee
Using this and the triangle inequality, we have
\be
\begin{split}
     &
    \bigg\|\scrF_{\xi}\bigg(\eta_l(\la_2\xi)V_{\va}(\lan \cdot\ran^{-dt_1/2}e^{2\pi i\Phi(\cdot,\la_2\xi)})(z_1,\z_1)\bigg)(\z_2)\bigg\|_{L^{1,1,\fy}_{\z_1,z_1,\z_2}}
    \\
    \lesssim &
    \sum_{m\in A_l}
    \bigg\|\scrF_{\xi}\bigg(\eta_l(\la_2\xi)\eta_m(\xi)V_{\va}(\lan \cdot\ran^{-dt_1/2}e^{2\pi i\Phi(\cdot,\la_2\xi)})(z_1,\z_1)\bigg)(\z_2)\bigg\|_{L^{1,1,\fy}_{\z_1,z_1,\z_2}}
    \\
    \lesssim &
    \sum_{m\in A_l}\bigg\|\scrF_{\xi}\bigg(\eta_m(\xi)V_{\va}(\lan \cdot\ran^{-dt_1/2}e^{2\pi i\Phi(\cdot,\la_2\xi)})(z_1,\z_1)\bigg)(\z_2)\bigg\|_{L^{1,1,\fy}_{\z_1,z_1,\z_2}}
    \|\scrF_{\xi}\eta_l(\la_2\xi)\|_{L^1}
    \\
    \lesssim &
    \lan l\ran^{d t_2/2}\sup_{m\in A_l}\bigg\|\scrF_{\xi}\bigg(\eta_m(\xi)V_{\va}(\lan \cdot\ran^{-dt_1/2}e^{2\pi i\Phi(\cdot,\la_2\xi)})(z_1,\z_1)\bigg)(\z_2)\bigg\|_{L^{1,1,\fy}_{\z_1,z_1,\z_2}}.
\end{split}
\ee
We only need to verify that
\be
\bigg\|\scrF_{\xi}\bigg(\eta_m(\xi)V_{\va}(\lan \cdot\ran^{-dt_1/2}e^{2\pi i\Phi(\cdot,\la_2\xi)})(z_1,\z_1)\bigg)(\z_2)\bigg\|_{L^{1,1,\fy}_{\z_1,z_1,\z_2}}
\lesssim \lan l\ran^{dt_2/2},\ \ \ \forall m\in A_l,
\ee
which is equivalent to
\be
\|\lan x\ran^{-dt_1/2}\eta_m(\xi)e^{2\pi i\Phi_{\la_2}(x,\xi)}\|_{M^{1,1,\fy,\fy}(c_1)}
\lesssim \lan l\ran^{\frac{dt_2}{2}},\ \ \  \Phi_{\la_2}(x,\xi)=\Phi(x,\la_2\xi).
\ee
Let $\La$ be a subset of $\zd$ such that $\text{supp}\eta_{k}^*\cap \text{supp}\eta_{\tilde{k}}^*=\emptyset$, for all $k\neq \tilde{k}$
    with $k,\tilde{k}\in \La$.
Write
\be
\begin{split}
	&\lan x\ran^{-dt_1/2}\sum_{k\in \La}\eta_{k,m}(x,\xi)e^{2\pi i\Phi_{\la_2}(x,\xi)}
    = \left(\sum_{k\in \La}\eta_{k,m}^{*}(x,\xi)e^{2\pi i(\nabla_{x} \Phi_{\la_2}(k,m)\cdot x+\nabla_{\xi} \Phi_{\la_2}(k,m)\cdot \xi)}\right)
	\\
	\cdot &
	\left(\lan x\ran^{-dt_1/2}\sum_{k\in \La}\eta_{k,m}(x,\xi)e^{2\pi i(\Phi_{\la_2}(x,\xi)-\nabla_{x} \Phi_{\la_2}(k,m)\cdot x-\nabla_{\xi} \Phi_{\la_2}(k,m)\cdot \xi)}\right).
\end{split}
\ee
Using the product inequality in Lemma \ref{lm-pd-c12}:
	\be
	M^{1,1,\fy,\fy}(c_1) \cdot M^{1,\fy,1,\fy}(c_1)\subset M^{1,1,\fy,\fy}(c_1),
	\ee
    we only need to verify that
	\be
	E_1=\bigg\|\sum_{k\in \La}\eta_{k,m}^{*}(x,\xi)e^{2\pi i(\nabla_{x} \Phi_{\la_2}(k,m)\cdot x+\nabla_{\xi} \Phi_{\la_2}(k,m)\cdot \xi)}\bigg\|_{M^{1,1,\fy,\fy}(c_1)}\lesssim \lan l\ran^{\frac{dt_2}{2}}
	\ee
    and 
	\be
	E_2=\bigg\|\lan x\ran^{-dt_1/2}\sum_{k\in \La}\eta_{k,m}(x,\xi)e^{2\pi i(\Phi_{\la_2}(x,\xi)-\nabla_{x} \Phi_{\la_2}(k,m)\cdot x-\nabla_{\xi} \Phi_{\la_2}(k,m)\cdot \xi)}\bigg\|_{M^{1,\fy,1,\fy}(c_1)}
	\lesssim
	1.
	\ee

Let us deal with the estimate of $E_1$ first. Let $\Psi=\psi\otimes \psi$ with $\psi\in C_c^{\fy}(\rd)$, such that $\text{supp}\Psi\subset B(0,\d)$
	for sufficiently small $\d>0$, satisfying that
	\be
	(\text{supp}\eta_{k,m}^*+B(0,\d)) \subset B(k, r)\times B(m, r) \ \ \ r>0
	\ee
	and the orthogonality property
		\be
	B(k, r) \cap B(\tilde{k}, r)=\emptyset,\ \ \ \text{for all}\ k\neq \tilde{k}.
	\ee
	Note that
	\be
	|\scrV_{\Psi}\eta_{0,0}^*(z_1,z_2,\z_1,\z_2)|\lesssim \chi_{B(0,r)}(z_1)\chi_{B(0,r)}(z_2)\lan \z_1\ran^{-\scrL}\lan \z_2\ran^{-\scrL},
	\ee
	and
	\be
	\begin{split}
		&|\scrV_{\Psi}\big(\eta_{k,m}^{*}(x,\xi)e^{2\pi i(\nabla_{x} \Phi_{\la_2}(k,m)\cdot x+\nabla_{\xi} \Phi_{\la_2}(k,m)\cdot \xi)}\big)(z,\z)|
		\\
		= &
		|\scrV_{\Psi}\big(M_{(\nabla_{x} \Phi_{\la_2}(k,m), \nabla_{\xi} \Phi_{\la_2}(k,m))}T_{(k,m)}\eta_{0,0}^*\big)(z,\z)|
        \\
		= &
		|\scrV_{\Psi}\eta_{0,0}^*(z_1-k,z_2-m,\z_1-\nabla_{x} \Phi_{\la_2}(k,m),\z_2-\nabla_{\xi} \Phi_{\la_2}(k,m))|.
	\end{split}
	\ee
	From this, we find that
	\be
	\begin{split}
&|\scrV_{\Psi}\big(\eta_{k,m}^{*}(x,\xi)e^{2\pi i(\nabla_{x} \Phi_{\la_2}(k,m)\cdot x+\nabla_{\xi} \Phi_{\la_2}(k,m)\cdot \xi)}\big)(z,\z)|
\\
\lesssim & 
\chi_{B(k,r)}(z_1)\chi_{B(m,r)}(z_2)\lan \z_1-\nabla_{x} \Phi_{\la_2}(k,m)\ran^{-\scrL}\lan \z_2-\nabla_{\xi} \Phi_{\la_2}(k,m)\ran^{-\scrL}.
\end{split}
\ee
	Using the orthogonality property, we obtain 
	\be
	\begin{split}
		E_1=&\|\sum_{k}\scrV_{\Psi}\big(\eta_{k,m}^{*}(x,\xi)e^{2\pi i(\nabla_{x} \Phi_{\la_2}(k,m)\cdot x+\nabla_{\xi} \Phi_{\la_2}(k,m)\cdot \xi)}\big)(z,\z)\|_{L^{1,1,\fy,\fy}_{\z_1,z_1,\z_2,z_2}}
		\\
		\lesssim &
		\|\sum_{k}\chi_{B(k,r)}(z_1)\chi_{B(m,r)}(z_2)\lan \z_1-\nabla_{x} \Phi_{\la_2}(k,m)\ran^{-\scrL}\lan \z_2-\nabla_{\xi} \Phi_{\la_2}(k,m)\ran^{-\scrL}\|_{L^{1,1,\fy,\fy}_{\z_1,z_1,\z_2,z_2}}
		\\
		\sim &
		\|\sum_{k}\chi_{B(k,r)}(z_1)\chi_{B(m,r)}(z_2)\lan \z_2-\nabla_{\xi} \Phi_{\la_2}(k,m)\ran^{-\scrL}\|_{L^{1,\fy,\fy}_{z_1,\z_2,z_2}}
		\\
		\sim &
		\|\sum_{k}\chi_{B(m,r)}(z_2)\lan \z_2-\nabla_{\xi} \Phi_{\la_2}(k,m)\ran^{-\scrL}\|_{L^{\fy,\fy}_{\z_2,z_2}}
			\\
		= & 
		\|\chi_{B(m,r)}(z_2)\sum_{k}\lan \z_2-\nabla_{\xi} \Phi_{\la_2}(k,l)\ran^{-\scrL}\|_{L^{\fy,\fy}_{\z_2,z_2}}.
	\end{split}
	\ee	
	Using the separation property, we have the following estimate for $k\neq \tilde{k}$:
    \be
    \begin{split}
    |\nabla_{\xi} \Phi_{\la_2}(k,m)-\nabla_{\xi} \Phi_{\la_2}(\tilde{k},m)|
    = &
    \la_2 |\nabla_{\xi} \Phi(k,\la_2 m)-\nabla_{\xi} \Phi(\tilde{k},\la_2 m)|\gtrsim \la_2.
    \end{split}
    \ee
    There exists some $\d>0$ independent of $m$ and $l$ such that the sequence of sets
    \be
    \big\{B(\nabla_{\xi} \Phi_{\la_2}(k,m),\la_2 \d)\big\}_{k\in \zd}
    \ee
    is pairwise disjoint.
    From this, we obtain that
	\be
	\begin{split}
		\sum_{k}\lan \z_2-\nabla_{\xi} \Phi(k,m)\ran^{-\scrL}
		\sim &
		\la_2^{-d}\sum_{k}\int_{B(\nabla_{\xi} \Phi(k,m),\la_2\d)}\lan \z_2-\z_1\ran^{-\scrL}d\z_1
		\\
		= &
		\la_2^{-d}\int_{\bigcup_{k}B(\nabla_{\xi} \Phi(k,m),\la_2\d)}\lan \z_2-\z_1\ran^{-\scrL}d\z_1
		\\
		\leq &
		\la_2^{-d}\int_{\rd}\lan \z_2-\z_1\ran^{-\scrL}d\z_1\lesssim \la_2^{-d}=\lan l\ran^{\frac{dt_2}{2}}.
	\end{split}
	\ee
	The desired estimate of $E_1$ follows by
		\be
	\begin{split}
		E_1\lesssim \|\chi_{B(m,r)}(z_2)\sum_{k}\lan \z_2-\nabla_{\xi} \Phi_{\la_2}(k,l)\ran^{-\scrL}\|_{L^{\fy,\fy}_{\z_2,z_2}}
		\lesssim 
	\lan l\ran^{\frac{dt_2}{2}}\|\chi_{B(m,r)}(z_2)\|_{L^{\fy}_{z_2}}\lesssim \lan l\ran^{\frac{dt_2}{2}}.
	\end{split}
	\ee	
	Next, we turn to the estimate of $E_2$. Using the embedding relation
	\be
	M^{1,\fy,\fy,\fy}_{1\otimes v_{d+\ep,0}}(c_1) \subset M^{1,\fy,1,\fy}(c_1),
	\ee
	we only need to verify
		\be
	\bigg\|\lan x\ran^{-dt_1/2}\sum_{k}\eta_{k,m}(x,\xi)e^{2\pi i(\Phi_{\la_2}(x,\xi)-\nabla_{x} \Phi_{\la_2}(k,m)\cdot x-\nabla_{\xi} \Phi_{\la_2}(k,m)\cdot \xi)}\bigg\|_{M^{1,\fy,\fy,\fy}_{1\otimes v_{d+\ep,0}}(c_1)}\lesssim 1.
	\ee
	Using Proposition \ref{pp-tlp-ws2}, we find 
	\be
	\begin{split}
		&\bigg\|\lan x\ran^{-dt_1/2}\sum_{k}\eta_{k,m}(x,\xi)e^{2\pi i(\Phi_{\la_2}(x,\xi)-\nabla_{x} \Phi_{\la_2}(k,m)\cdot x-\nabla_{\xi} \Phi_{\la_2}(k,m)\cdot \xi)}\bigg\|_{M^{1,\fy,\fy,\fy}_{1\otimes v_{d+\ep,0}}(c_1)}
		\\
		\sim &
		\sup_k \bigg\|\lan x\ran^{-dt_1/2}\eta_{k,m}(x,\xi)e^{2\pi i(\Phi_{\la_2}(x,\xi)-\nabla_{x} \Phi_{\la_2}(k,m)\cdot x-\nabla_{\xi} \Phi_{\la_2}(k,m)\cdot \xi)}\bigg\|_{\scrF L^{1,\fy}_{1\otimes v_{d+\ep}}}
		\\
		\sim &
		\sup_k \lan k\ran^{\frac{-dt_1}{2}}\bigg\|\eta_{k,m}(x,\xi)e^{2\pi i(\Phi_{\la_2}(x,\xi)-\nabla_{x} \Phi_{\la_2}(k,m)\cdot x-\nabla_{\xi} \Phi_{\la_2}(k,m)\cdot \xi)}\bigg\|_{\scrF L^{1,\fy}_{1\otimes v_{d+\ep}}}
	\end{split}
	\ee
	From this, we only need to verify
			\ben\label{thm-FIO-12hg-M1-2}
	\bigg\|\eta_{k,m}(x,\xi)e^{2\pi i(\Phi_{\la_2}(x,\xi)-\nabla_{x} \Phi_{\la_2}(k,m)\cdot x-\nabla_{\xi} \Phi_{\la_2}(k,m)\cdot \xi)}\bigg\|_{\scrF L^{1,\fy}_{1\otimes v_{d+\ep}}}
	\lesssim 
	\lan k\ran^{\frac{dt_1}{2}}.
	\een
	Let $\la_1=\lan k\ran^{\frac{-t_1}{2}}$.
        Define
        \begin{equation*}
            B_k=\{j\in \zd: \eta_j(x)\eta_k(\la_1x)\neq 0\}.
        \end{equation*}
        Write the left term of \eqref{thm-FIO-12hg-M1-2} by
	\be
	\begin{split}
		&\bigg\|\eta_{k,m}(x,\xi)e^{2\pi i(\Phi_{\la_2}(x,\xi)-\nabla_{x} \Phi_{\la_2}(k,m)\cdot x-\nabla_{\xi} \Phi_{\la_2}(k,m)\cdot \xi)}\bigg\|_{\scrF L^{1,\fy}_{1\otimes v_{d+\ep}}}
		\\
		\sim &
		\bigg\|\eta_{k,m}(\la_1x,\xi)e^{2\pi i(\Phi_{\la_2}(\la_1x,\xi)-\nabla_{x} \Phi_{\la_2}(k,m)\cdot \la_1x-\nabla_{\xi} \Phi_{\la_2}(k,m)\cdot \xi)}\bigg\|_{\scrF L^{1,\fy}_{1\otimes v_{d+\ep}}}
		\\
		\sim &
		\bigg\|\sum_{j\in B_k}\eta_j(x)\eta_{k}(\la_1x)\eta_{m}(\xi)e^{2\pi i(\Phi(\la_1x,\la_2\xi)-\nabla_{\xi} \Phi(k,\la_2m)\cdot \la_2\xi)}\bigg\|_{\scrF L^{1,\fy}_{1\otimes v_{d+\ep}}}
		\\
		\lesssim &
		\la_1^{-d}
		\sup_{j\in B_k}\bigg\|\eta_j(x)\eta_{m}(\xi)e^{2\pi i(\Phi(\la_1x,\la_2\xi)-\nabla_{\xi} \Phi(k,\la_2m)\cdot \la_2\xi)}\bigg\|_{\scrF L^{1,\fy}_{1\otimes v_{d+\ep}}}
		\bigg\|\eta_{k}(\la_1x)\eta_m^*(\xi)\bigg\|_{\scrF L^{1,\fy}_{1\otimes v_{d+\ep}}}
		\\
		\lesssim &
		\lan k\ran^{\frac{dt_1}{2}}
		\sup_{j\in B_k}\bigg\|\eta_j(x)\eta_{m}(\xi)e^{2\pi i(\Phi(\la_1x,\la_2\xi)-\nabla_{\xi} \Phi(k,\la_2m)\cdot \la_2\xi)}\bigg\|_{\scrF L^{1,\fy}_{1\otimes v_{d+\ep}}}.
	\end{split}
	\ee
	From this, in order to verify \eqref{thm-FIO-12hg-M1-2}, we only need to verify that
	\be
	\bigg\|\eta_j(x)\eta_{m}(\xi)e^{2\pi i(\Phi(\la_1x,\la_2\xi)-\nabla_{\xi} \Phi(k,\la_2m)\cdot \la_2\xi)}\bigg\|_{\scrF L^{1,\fy}_{1\otimes v_{d+\ep}}}
	\lesssim 1.
	\ee
	Denote $\Phi_{\la}(x,\xi)=\Phi(\la_1 x,\la_2 \xi)$. The above estimate is equivalent to
	\be
	\bigg\|\eta_{0,0}e^{2\pi i(\Phi_{\la}(x+j,\xi+m)-\nabla_{\xi} \Phi_{\la}(k/\la_1,m)\cdot \xi)}\bigg\|_{\scrF L^{1,\fy}_{1\otimes v_{d+\ep}}}\lesssim 1.
	\ee
	Write
	\be
	\begin{split}
		&
		\Phi_{\la}(x+j,\xi+m)-\nabla_{\xi} \Phi_{\la}(k/\la_1,m)\cdot \xi
		\\
		= &
		(\Phi_{\la}(x+j,\xi+m)-\Phi_{\la}(j,m)-\nabla_x \Phi_{\la}(j,m)\cdot x-\nabla_{\xi} \Phi_{\la}(j,m)\cdot \xi)
		\\
		& +
		(\Phi_{\la}(j,m)+\nabla_x \Phi_{\la}(j,m)\cdot x)+(\nabla_{\xi} \Phi_{\la}(j,m)\cdot \xi-\nabla_{\xi} \Phi_{\la}(k/\la_1,m)\cdot \xi)
		\\
		= &:
		\tau_{k,l,j,m}^{\la}(x,\xi)+(\Phi_{\la}(j,m)+\nabla_x \Phi_{\la}(j,m)\cdot x)+(\nabla_{\xi} \Phi_{\la}(j,m)\cdot \xi-\nabla_{\xi} \Phi_{\la}(k/\la_1,m)\cdot \xi).
	\end{split}
	\ee	
	Note that
	\be
	\begin{split}
		|\nabla_{\xi} \Phi_{\la}(j,m)-\nabla_{\xi} \Phi_{\la}(k/\la_1,m)|
		\lesssim &
		|j-k/\la_1|\sup_{|\g|=2}\|\partial_{x,\xi}^{\g}\Phi_{\la}\|_{L^{\fy}}
		\\
		\lesssim &
		\la_1\la_2 |j-k/\la_1|\sup_{|\g|=2}\|\partial_{x,\xi}^{\g}\Phi\|_{L^{\fy}}
		\\
		\lesssim & 
		\la_2|\la_1 j-k|\sup_{|\g|=2} \|\partial_{x,\xi}^{\g}\Phi\|_{M^{1,\fy,\fy,\fy}_{1\otimes v_{d+\ep,0}}(c_1)}\lesssim 1,
	\end{split}
	\ee
	We conclude that
	\be
	\begin{split}
		&\bigg\|\eta_{0,0}e^{2\pi i(\Phi_{\la}(x+j,\xi+m)-\nabla_{\xi} \Phi_{\la}(k/\la_1,m)\cdot \xi)}\bigg\|_{\scrF L^{1,\fy}_{1\otimes v_{d+\ep}}}
		\\
		= &
		\bigg\|\eta_{0,0}e^{2\pi i(\tau_{k,l,j,m}^{\la}(x,\xi)+\nabla_{\xi} \Phi_{\la}(j,m)\cdot \xi-\nabla_{\xi} \Phi_{\la}(k/\la_1,m)\cdot \xi)}\bigg\|_{\scrF L^{1,\fy}_{1\otimes v_{d+\ep}}}
		\sim 
		\bigg\|\eta_{0,0}e^{2\pi i\tau_{k,l,j,m}^{\la}(x,\xi)}\bigg\|_{\scrF L^{1,\fy}_{1\otimes v_{d+\ep}}}.
	\end{split}
	\ee
	From this and Proposition \ref{pp-tlp-ws2}, we only need to verify that
	\be
	\bigg\|\eta_{0,0}e^{2\pi i\tau_{k,l,j,m}^{\la}(x,\xi)}\bigg\|_{M^{1,\fy,\fy,\fy}_{1\otimes v_{d+\ep,0}}(c_1)} \lesssim 1.
	\ee
    Furthermore, by using the embedding relation $W^{\fy,\fy}_{1\otimes v_{d+\ep,d+\ep}}(\rdd) \subset M^{1,\fy,\fy,\fy}_{1\otimes v_{d+\ep,0}}(c_1)$, we only need to verify that
    \be
    \bigg\|\eta_{0,0}e^{2\pi i\tau_{k,l,j,m}^{\la}(x,\xi)}\bigg\|_{W^{\fy,\fy}_{1\otimes v_{d+\ep,d+\ep}}} \lesssim 1.
    \ee

	Note that
	\be
	\begin{split}
		\bigg\|\eta_{0,0}e^{2\pi i\tau_{k,l,j,m}^{\la}(x,\xi)}\bigg\|_{W^{\fy,\fy}_{1\otimes v_{d+\ep,d+\ep}}}
		= 
		\bigg\|\eta_{0,0}e^{2\pi i\eta_{0,0}^{*}\tau_{k,l,j,m}^{\la}(x,\xi)}\bigg\|_{W^{\fy,\fy}_{1\otimes v_{d+\ep,d+\ep}}}
		\lesssim
		e^{C\|\eta_{0,0}^{*}\tau_{k,l,j,m}^{\la}\|_{W^{\fy,\fy}_{1\otimes v_{d+\ep,d+\ep}}}}.
	\end{split}
	\ee
	Write
	\be
	\begin{split}
		\tau_{k,l,j,m}^{\la}
		= 
		2\sum_{|\g|=2}\frac{(x,\xi)^{\g}}{\g !}\int_{0}^1(1-t)(\partial^{\g}\Phi_{\la})(j+tx,m+t\xi)dt
		= 
		\tau_{k,l,j,m}^{\la,x,x}+\tau_{k,l,j,m}^{\la,x,\xi}+\tau_{k,l,j,m}^{\la,\xi,\xi},
	\end{split}
	\ee
	Here, we use $\tau_{k,l,j,m}^{\la,x,x}$, $\tau_{k,l,j,m}^{\la,\xi,\xi}$ and $\tau_{k,l,j,m}^{\la,x,\xi}$ to denote the partial summation of $\tau_{k,l,j,m}$ associated with $\partial_{x,x}^{\g}$ terms,
	$\partial_{\xi,\xi}^{\g}$ terms and $\partial_{x,\xi}^{\g}$ terms, respectively.
	We only need to verify that
	\be
	\|\eta_{0,0}^{*}\tau_{k,l,j,m}^{\la,x,x}\|_{W^{\fy,\fy}_{1\otimes v_{d+\ep,d+\ep}}}
	\lesssim
	A=\sum_{|\g|=2}\|\lan x\ran^{-t_1} \partial^{\g}_{x,x}\Phi\|_{W^{\fy,\fy}_{1\otimes v_{d+\ep,d+\ep}}},
	\ee
	\be
	\|\eta_{0,0}^{*}\tau_{k,l,j,m}^{\la,\xi,\xi}\|_{W^{\fy,\fy}_{1\otimes v_{d+\ep,d+\ep}}}
	\lesssim 
	B=\sum_{|\g|=2}\|\partial^{\g}_{\xi,\xi}\Phi\|_{W^{\fy,\fy}_{1\otimes v_{d+\ep,d+\ep}}}
	\ee
	and
	\be
	\|\eta_{0,0}^{*}\tau_{k,l,j,m}^{\la,x,\xi}\|_{W^{\fy,\fy}_{1\otimes v_{d+\ep,d+\ep}}}
	\lesssim 
	C=\sum_{|\g|=2}\|\lan \xi\ran^{-t_2}\partial^{\g}_{x,\xi}\Phi\|_{W^{\fy,\fy}_{1\otimes v_{d+\ep,d+\ep}}}.
	\ee
	We only give the proof for the first and second estimates, since the estimate of the third term is similar and easier.
	Take $\phi$ to be a $C_c^{\fy}(\rdd)$ function such that $\phi=1$ on the $\text{supp}\eta_{0,0}^{*}$.
	Write
	\be
	\begin{split}
		&\|\eta_{0,0}^{*}\tau_{k,l,j,m}^{\la,x,x}\|_{W^{\fy,\fy}_{1\otimes v_{d+\ep,d+\ep}}}
		\\
		= &
		2\bigg\|\sum_{|\g|=2}\frac{x^{\g}}{\g !}\eta_{0,0}^{*}(x,\xi)\int_{0}^1(1-t)\phi(tx,t\xi)(\partial^{\g}_{x,x}\Phi_{\la})(j+tx,m+t\xi)dt\bigg\|_{W^{\fy,\fy}_{1\otimes v_{d+\ep,d+\ep}}}
		\\
		\lesssim &
		\sum_{|\g|=2}\bigg\|\int_{0}^1(1-t)\phi(tx,t\xi)(\partial^{\g}_{x,x}\Phi_{\la})(j+tx,m+t\xi)dt\bigg\|_{W^{\fy,\fy}_{1\otimes v_{d+\ep,d+\ep}}}
		\\
		\lesssim &
		\sum_{|\g|=2}\sup_{t\in (0,1)}\|\phi(tx,t\xi)(\partial^{\g}_{x,x}\Phi_{\la})(j+tx,m+t\xi)\|_{W^{\fy,\fy}_{1\otimes v_{d+\ep,d+\ep}}}
		\\
		\lesssim &
		\sum_{|\g|=2}\|\phi(x,\xi)(\partial^{\g}_{x,x}\Phi_{\la})(j+x,m+\xi)\|_{W^{\fy,\fy}_{1\otimes v_{d+\ep,d+\ep}}}.
	\end{split}
	\ee
	Similarly, we have
	\be
	\|\eta_{0,0}^{*}\tau_{k,l,j,m}^{\la,\xi,\xi}\|_{W^{\fy,\fy}_{1\otimes v_{d+\ep,d+\ep}}}
	\lesssim
	\sum_{|\g|=2}\|\phi(x,\xi)(\partial^{\g}_{\xi,\xi}\Phi_{\la})(j+x,m+\xi)\|_{W^{\fy,\fy}_{1\otimes v_{d+\ep,d+\ep}}}.
	\ee
	The only remaining thing is to verify
	\be
	\|\phi(x,\xi)(\partial^{\g}_{x,x}\Phi_{\la})(j+x,m+\xi)\|_{W^{\fy,\fy}_{1\otimes v_{d+\ep,d+\ep}}}
	\lesssim 
	\|\lan x\ran^{-t_1} \partial^{\g}_{x,x}\Phi\|_{W^{\fy,\fy}_{1\otimes v_{d+\ep,d+\ep}}}.
	\ee
	and
		\be
	\|\phi(x,\xi)(\partial^{\g}_{\xi,\xi}\Phi_{\la})(j+x,m+\xi)\|_{W^{\fy,\fy}_{1\otimes v_{d+\ep,d+\ep}}}
	\lesssim 
	\|\lan \xi\ran^{-t_2} \partial^{\g}_{\xi,\xi}\Phi\|_{W^{\fy,\fy}_{1\otimes v_{d+\ep,d+\ep}}}.
	\ee
	
	Take a suitable function $\rho\in C_c^{\fy}(\rdd)$ such that for all $j\in B_k, m\in A_l$.
        Denote $\rho_{k,l}(x,\xi)=\rho(x-l,\xi-k)$.
        We have
	\be
	\phi(x-j,\xi-m)\rho_{k,l}(\la_1 x,\la_2 \xi)=\phi(x-j,\xi-m).
	\ee
	Using this, we conclude that
	\be
	\begin{split}
		&
		\|\rho_{k,l}(\la_1 x,\la_2 \xi)(\partial_{x,x}^{\g}\Phi_{\la})(x,\xi)\|_{W^{\fy,\fy}_{1\otimes v_{d+\ep,d+\ep}}}
		\\
		= &
		\la_1^2
		\|\rho_{k,l}(\la_1 x,\la_2 \xi)(\partial_{x,x}^{\g}\Phi)(\la_1 x,\la_2 \xi)\|_{W^{\fy,\fy}_{1\otimes v_{d+\ep,d+\ep}}}
		\\
		\lesssim &
		\la_1^2\|\rho_{k,l}\partial_{x,x}^{\g}\Phi\|_{W^{\fy,\fy}_{1\otimes v_{d+\ep,d+\ep}}}
		\\
		\lesssim &
		\la_1^2 \lan k\ran^{t_1}
		\|\lan x\ran^{-t_1}\partial^{\g}_{x,x}\Phi\|_{W^{\fy,\fy}_{1\otimes v_{d+\ep,d+\ep}}}
		=\|\lan x\ran^{-t_1}\partial^{\g}_{x,x}\Phi\|_{W^{\fy,\fy}_{1\otimes v_{d+\ep,d+\ep}}},
	\end{split}
	\ee
	and
		\be
	\begin{split}
		&
		\|\rho_{k,l}(\la_1 x,\la_2 \xi)(\partial_{\xi,\xi}^{\g}\Phi_{\la})(x,\xi)\|_{W^{\fy,\fy}_{1\otimes v_{d+\ep,d+\ep}}}
		\\
		= &
		\la_2^2
		\|\rho_{k,l}(\la_1 x,\la_2 \xi)(\partial_{\xi,\xi}^{\g}\Phi)(\la_1 x,\la_2 \xi)\|_{W^{\fy,\fy}_{1\otimes v_{d+\ep,d+\ep}}}
		\\
		\lesssim &
		\la_2^2\|\rho_{k,l}\partial_{\xi,\xi}^{\g}\Phi\|_{W^{\fy,\fy}_{1\otimes v_{d+\ep,d+\ep}}}
		\\
		\lesssim &
		\la_2^2 \lan l\ran^{t_2}
		\|\lan \xi\ran^{-t_2}\partial^{\g}_{\xi,\xi}\Phi\|_{W^{\fy,\fy}_{1\otimes v_{d+\ep,d+\ep}}}
		\lesssim \|\lan \xi\ran^{-t_2}\partial^{\g}_{\xi,\xi}\Phi\|_{W^{\fy,\fy}_{1\otimes v_{d+\ep,d+\ep}}}.
	\end{split}
	\ee
	From this, we obtain the desired estimate by
	\be
	\begin{split}
		&
		\|\phi(x,\xi)(\partial^{\g}_{x,x}\Phi_{\la})(j+x,m+\xi)\|_{W^{\fy,\fy}_{1\otimes v_{d+\ep,d+\ep}}}
		\\
		= &
		\|\phi(x-j,\xi-m)(\partial^{\g}_{x,x}\Phi_{\la})(x,\xi)\|_{W^{\fy,\fy}_{1\otimes v_{d+\ep,d+\ep}}}
		\\
		= &
		\|\phi(x-j,\xi-m)\rho_{k,l}(\la_1 x,\la_2 \xi)(\partial^{\g}_{x,x}\Phi_{\la})(x,\xi)\|_{W^{\fy,\fy}_{1\otimes v_{d+\ep,d+\ep}}}
		\\
		\lesssim &
		\|\rho_{k,l}(\la_1 x,\la_2 \xi)(\partial^{\g}_{x,x}\Phi_{\la})(x,\xi)\|_{W^{\fy,\fy}_{1\otimes v_{d+\ep,d+\ep}}}
		\|\phi(x-j,\xi-m)\|_{W^{\fy,\fy}_{1\otimes v_{d+\ep,d+\ep}}}
		\\
		\lesssim &
		\|\lan x\ran^{-t_1} \partial^{\g}_{x,x}\Phi\|_{W^{\fy,\fy}_{1\otimes v_{d+\ep,d+\ep}}}.
	\end{split}
	\ee
   The estimate of $\|\phi(x,\xi)(\partial^{\g}_{\xi,\xi}\Phi_{\la})(j+x,m+\xi)\|_{W^{\fy,\fy}_{1\otimes v_{d+\ep,d+\ep}}}$ can be treated similarly.
\end{proof}

\begin{theorem}[FIO on $\mf$, high growth]\label{thm-FIO-12hg-Mf}
	Let $\Phi$ be a real-valued $C^2(\rdd)$ function 
	satisfying 
    the $(-\fy,t_1,t_2)$-growth condition and uniform separation condition of $\xi$-type.
	Denote 
	\be
	A=\sum_{|\g|=2}\|\lan x\ran^{-t_1}\partial^{\g}_{x,x}\Phi\|_{W^{\fy,\fy}_{1\otimes v_{d+\ep,d+\ep}}},\ \ \ 
	B=\sum_{|\g|=2}\|\lan \xi\ran^{-t_2}\partial^{\g}_{\xi,\xi}\Phi\|_{W^{\fy,\fy}_{1\otimes v_{d+\ep,d+\ep}}},
	\ee
	and
	\be
	C=\sum_{|\g|=2}\|\partial^{\g}_{x,\xi}\Phi\|_{W^{\fy,\fy}_{1\otimes v_{d+\ep,d+\ep}}}
	\ee
	We have $\lan x\ran^{-dt_1/2}\lan \xi\ran^{-dt_2/2}e^{2\pi i\Phi(x,\xi)}\in M^{1,1,\fy,\fy}(c_2)$ with
	\ben\label{thm-FIO-12hg-Mf-cd1}
	\|\lan x\ran^{-dt_1/2}\lan \xi\ran^{-dt_2/2}e^{2\pi i\Phi(x,\xi)}\|_{M^{1,1,\fy,\fy}(c_2)}
	\lesssim 
	e^{A+B+C}.
	\een
	Moreover,
	for every $\s\in \msj_{v_{\frac{dt_1}{2},\frac{dt_2}{2}}\otimes 1}(\rdd)$, the Fourier integral operator $T_{\s,\Phi}\in \calL(\mfd)$ with the estimate of operator norm
	\be
	\begin{split}
		\|T_{\s,\Phi}\|_{\calL(\mfd)}
		\lesssim  \|\s\|_{\msj_{v_{\frac{dt_1}{2},\frac{dt_2}{2}}\otimes 1}(\rdd)}e^{A+B+C}.
	\end{split}
	\ee
\end{theorem}
\begin{proof}
	Let $\tilde{\Phi}(x,\xi)=\Phi(\xi,x)$. Observe that $\tilde{\Phi}$ satisfies the 
    uniform separation property of $x$-type:
	\be
	|\nabla_{\xi} \tilde{\Phi}(x_1,\xi)-\nabla_{\xi} \tilde{\Phi}(x_2,\xi)|\gtrsim 1\ \ \text{uniformly for all}\ \xi\in \rd,\ |x_1-x_2|\gtrsim 1,
	\ee
	and norm conditions of the two order derivatives:
	\be
	\partial_{x,\xi}^{\g}\tilde{\Phi}\in W^{\fy,\fy}_{1\otimes v_{d+\ep,d+\ep}}\ \ \text{for all}\ |\g|=2,
	\ee
	and
	\be
	\lan x\ran^{-t_2}\partial_{x,x}^{\g}\tilde{\Phi}\in W^{\fy,\fy}_{1\otimes v_{d+\ep,d+\ep}},\ \ 
	\lan \xi\ran^{-t_1}\partial_{\xi,\xi}^{\g}\tilde{\Phi}\in W^{\fy,\fy}_{1\otimes v_{d+\ep,d+\ep}},
	\ \ \text{for all}\ |\g|=2.
	\ee
    From this and Theorem \ref{thm-FIO-12hg-M1}, we conclude that $\lan x\ran^{-dt_2/2}\lan \xi\ran^{-dt_1/2}e^{2\pi i\tilde{\Phi}(x,\xi)}\in M^{1,1,\fy,\fy}(c_1)$ with
	\be
	\|\lan x\ran^{-dt_2/2}\lan \xi\ran^{-dt_1/2}e^{2\pi i\tilde{\Phi}(x,\xi)}\|_{M^{1,1,\fy,\fy}(c_1)}
	\lesssim 
	e^{\tilde{A}+\tilde{B}+\tilde{C}},
    \ee
    where
    	\be
	\tilde{A}=\sum_{|\g|=2}\|\lan x\ran^{-t_2}\partial_{x,x}^{\g}\tilde{\Phi}\|_{W^{\fy,\fy}_{1\otimes v_{d+\ep,d+\ep}}},\ \ \ 
	\tilde{B}=\sum_{|\g|=2}\|\lan \xi\ran^{-t_1}\partial_{\xi,\xi}^{\g}\tilde{\Phi}\|_{W^{\fy,\fy}_{1\otimes v_{d+\ep,d+\ep}}},
	\ee
	and
	\be
	\tilde{C}=\sum_{|\g|=2}\|\partial^{\g}_{x,\xi}\tilde{\Phi}\|_{W^{\fy,\fy}_{1\otimes v_{d+\ep,d+\ep}}}
	\ee
    Using this and Lemma \ref{lm-rlc}, we obtain the desired estimate \eqref{thm-FIO-12hg-Mf-cd1}.
    The corresponding boundedness on $\mf$ follows by Proposition \ref{pp-KFIO-Mi-Mf}.
\end{proof}

\begin{proof}[The proof of Theorem \ref{thm-FIO-Mp-12hg}: sufficiency part]
The cases of $p=1$ and $p=\fy$ have been handled in Theorem \ref{thm-FIO-12hg-M1} and Theorem \ref{thm-FIO-12hg-Mf} respectively.
We only need to consider the cases $1<p<2$ and $2<p<\fy$. 
Since both of these two cases can be treated by the same way, we only
give the proof for $1<p<2$. 
In this case, we need to verify the boundedness:
\be
\|T_{\s,\Phi}\|_{\calL(M^p)}+\|T_{\s,\Phi}\|_{\calL(M^{p'})}
\lesssim 
\|\s\|_{M^{\fy,1}_{v_{s_1,s_2}\otimes 1}}e^{(A+B+C)|2/p-1|},
\ee
which is equivalent to 
\be
\|T_{v_{s_1,s_2}\s,\Phi}\|_{\calL(M^p_{1\otimes v_{s_2}}, M^p_{v_{s_1}^{-1}\otimes 1})}
+
\|T_{v_{s_1,s_2}\s,\Phi}\|_{\calL(M^{p'}_{1\otimes v_{s_2}}, M^{p'}_{v_{s_1}^{-1}\otimes 1})}
\lesssim 
\|\s\|_{M^{\fy,1}_{v_{s_1,s_2}\otimes 1}}e^{(A+B+C)|2/p-1|}.
\ee
Using Theorems \ref{thm-FIO-12hg-M1} and \ref{thm-FIO-12hg-Mf} and the fact
\be
v_{s_1-\frac{dt_1}{2},s_2-\frac{dt_2}{2}}\s\in M^{\fy,1}_{v_{\frac{dt_1}{2},\frac{dt_2}{2}}},
\ee
we have the boundedness
\be
\begin{split}
    &\|T_{v_{s_1-\frac{dt_1}{2},s_2-\frac{dt_2}{2}}\s, \Phi}\|_{\calL(M^1)}+
    \|T_{v_{s_1-\frac{dt_1}{2},s_2-\frac{dt_2}{2}}\s, \Phi}\|_{\calL(M^{\fy})}
    \\
\lesssim & 
\|v_{s_1-\frac{dt_1}{2},s_2-\frac{dt_2}{2}}\s\|_{M^{\fy,1}_{v_{\frac{dt_1}{2},\frac{dt_2}{2}}}}e^{A+B+C}
= 
\|\s\|_{M^{\fy,1}_{v_{s_1,s_2}}}e^{A+B+C}.
\end{split}
\ee
This is equivalent to
\be
\|T_{v_{s_1,s_2}\s,\Phi}\|_{\calL(M^1_{1\otimes v_{\frac{dt_2}{2}}},M^1_{v_{\frac{dt_2}{2}^{-1}\otimes 1}})}+
\|T_{v_{s_1,s_2}\s,\Phi}\|_{\calL(M^{\fy}_{1\otimes v_{\frac{dt_2}{2}}},M^{\fy}_{v_{\frac{dt_2}{2}^{-1}\otimes 1}})}
\lesssim 
\|\s\|_{M^{\fy,1}_{v_{s_1,s_2}}}e^{A+B+C}.
\ee
Applying the complex interpolation argument between this boundedness and
 the assumption 
$\|T_{v_{s_1,s_2}\s,\Phi}\|_{\calL(L^2)}\lesssim \|\s\|_{M^{\fy,1}_{v_{s_1,s_2}}}$, we conclude that
\be
\|T_{v_{s_1,s_2}\s,\Phi}\|_{\calL(M^p_{1\otimes v_{\tilde{s_2}}}, M^p_{v_{\tilde{s_1}}^{-1}\otimes 1})}+
\|T_{v_{s_1,s_2}\s,\Phi}\|_{\calL(M^{p'}_{1\otimes v_{\tilde{s_2}}}, M^{p'}_{v_{\tilde{s_1}}^{-1}\otimes 1})}
\lesssim 
\|\s\|_{M^{\fy,1}_{v_{s_1,s_2}\otimes 1}}e^{(A+B+C)|2/p-1|},
\ee
where
\be
\tilde{s_1}=dt_1|1/p-1/2|,\ \ \  \tilde{s_2}=dt_2|1/p-1/2|.
\ee
The desired conclusion follows by
\be
\begin{split}
    &\|T_{v_{s_1,s_2}\s,\Phi}\|_{\calL(M^p_{1\otimes v_{s_2}}, M^p_{v_{s_1}^{-1}\otimes 1})}
    +
    \|T_{v_{s_1,s_2}\s,\Phi}\|_{\calL(M^{p'}_{1\otimes v_{s_2}}, M^{p'}_{v_{s_1}^{-1}\otimes 1})}
    \\
    \lesssim & 
    \|T_{v_{s_1,s_2}\s,\Phi}\|_{\calL(M^p_{1\otimes v_{\tilde{s_2}}}, M^p_{v_{\tilde{s_1}}^{-1}\otimes 1})}
    +
    \|T_{v_{s_1,s_2}\s,\Phi}\|_{\calL(M^{p'}_{1\otimes v_{\tilde{s_2}}}, M^{p'}_{v_{\tilde{s_1}}^{-1}\otimes 1})}
    \lesssim 
\|\s\|_{M^{\fy,1}_{v_{s_1,s_2}\otimes 1}}e^{(A+B+C)|2/p-1|}.
\end{split}
\ee
\end{proof}

\begin{lemma}[\cite{Domar1971AM,Littman1963BAMS}]\label{lm-osc}
Let $\Omega$ be a bounded open set of $\ \rd$. Assume that $g$ is a smooth function, $\mathrm{supp}g\subset \Omega.$
Let $\phi$ be a real-valued $C^{\infty}$ function on $\rd$ satisfying
\begin{equation}
\begin{split}
(1)~ &
|\partial^{\gamma}\phi|\leq A~on~\Omega~for~|\gamma|\leq N, \hspace{60mm}
\\
(2)~ &
|\det Hess\phi|\geq 1/A~on~\Omega,
\end{split}
\end{equation}
where $N$ is a positive integer determined by the dimension $d$. Then there exists a constant $C=C(d,g,A)$ such that
\begin{equation}
\|\mathcal {F}^{-1}[g(\xi)e^{i\lambda \phi(\xi)}]\|_{L^{\infty}}\leq C(1+|\lambda|)^{-d/2}
\end{equation}
for all $\lambda\in \mathbb{R}$. 
\end{lemma}

\begin{proof}[The proof of Theorem \ref{thm-FIO-Mp-12hg}: necessity part]
Let
\be
\mu_1(x)=\lan x\ran^{2+t_1}, \ \ \  \mu_2(\xi)=\lan \xi\ran^{2+t_2},
\ee
and
\be
\s(x,\xi)=\lan x\ran^{-s_1}\lan \xi\ran^{-s_2}\in M^{\fy,1}_{v_{t_1,t_2}\otimes 1},
\ \ \ 
\Phi(x,\xi)=\frac{\mu_1(x)+\mu_2(\xi)}{2\pi}+x\cdot \xi.
\ee
Then, the symbol function $\s\in M^{\fy,1}_{v_{s_1,s_2}\otimes 1}$,
the phase function $\Phi$ satisfy $(-\fy,t_1,t_2)$-growth condition and uniform separation condition.
We have
\be
T_{\s,\Phi}f(x)=\int_{\rd}\lan x\ran^{-s_1}\lan \xi\ran^{-s_2}e^{i\mu_1(x)}e^{i\mu_2(\xi)}\hat{f}(\xi)e^{2\pi ix\cdot \xi}d\xi
=\lan x\ran^{-s_1}e^{i\mu_1(x)}\lan D\ran^{-s_2}e^{i\mu_2(D)}f(x).
\ee
Moreover, we have the boundedness $T_{v_{s_1,s_2}\s,\Phi}\in \calL(L^2)$ since
\be
T_{v_{s_1,s_2}\s,\Phi}f(x)=T_{1,\Phi}f(x)=e^{i\mu_1(x)}e^{i\mu_2(D)}f(x).
\ee
What we want to prove is the following conclusion
\be
T_{\s,\Phi}\in \calL(M^p)\cap \calL(M^{p'})\Longrightarrow\ \ s_i\geq d t_i|1/p-1/2|.
\ee
Take $\va$ to be a smooth function supported on $B(0,\d)$, and let $\psi=\check{\va}$. 
Using the boundedness of $T_{\s,\Phi}$, we have
\be
\|\lan x\ran^{-s_1}e^{i\mu_1(x)}T_k\va\|_{\mp} \lesssim \|\lan D\ran^{s_2}e^{-i\mu_2(D)}T_k\va\|_{\mp}\lesssim 1
\ee
and
\be
1\lesssim \|\lan x\ran^{-s_1}e^{i\mu_1(x)}M_k\psi\|_{M^{p'}} \lesssim \|\lan D\ran^{s_2}e^{-i\mu_2(D)}M_k\psi\|_{M^{p'}}.
\ee
From this, we deduce that
\be
\lan k\ran^{-s_1}\|e^{i\mu_1(x)}T_k\va\|_{\scrF L^p}
\sim
\|\lan x\ran^{-s_1}e^{i\mu_1(x)}T_k\va\|_{\mp}
\lesssim 1
\ee
and
\be
1
\lesssim 
\|\lan D\ran^{s_2}e^{-i\mu_2(D)}M_k\psi\|_{M^{p'}}
\sim
\lan k\ran^{s_2}\|e^{-i\mu_2(x)}T_k\va\|_{\scrF L^{p'}}.
\ee
To achieve our goal, we need to verify that
\be
\|e^{i\mu_1(x)}T_k\va\|_{\scrF L^p} \gtrsim \lan k\ran^{dt_1|1/2-1/p|},
\ \ \ 
\|e^{-i\mu_2(x)}T_k\va\|_{\scrF L^{p'}}\lesssim \lan k\ran^{-dt_2|1/2-1/p|}.
\ee
We only deal with the second estimate, then the first follows by 
a similar argument and
the following inequality
\begin{equation*}
    1\lesssim \|e^{i\mu_1(x)}T_k\va\|_{\scrF L^2}^2
    \lesssim
    \|e^{i\mu_1(x)}T_k\va\|_{\scrF L^p}\|e^{-i\mu_1(x)}T_k\va\|_{\scrF L^{p'}}.
\end{equation*}
Let $\r$ be a smooth function supported on the annulus $\{x: 1/4\leq |x|\leq 4\}$ satisfies 
$\r(x)=1$ on $\{x: 1/2\leq |x|\leq 2\}$.
For sufficiently large $k$ we have $T_k\va \r(\frac{\cdot}{|k|})=T_k\va$.
Then, we have the estimate
\be
\begin{split}
    \|e^{-i\mu_2(x)}T_k\va\|_{\scrF L^{p'}}
    \lesssim &
    \|e^{-i\mu_2(x)}T_k\va\|_{\scrF L^{2}}^{\frac{2}{p'}}\|e^{-i\mu_2(x)}T_k\va\|_{\scrF L^{\fy}}^{\frac{2}{p}-1}
    \\
    \lesssim &
    \|e^{-i\mu_2(x)}T_k\va\|_{\scrF L^{\fy}}^{\frac{2}{p}-1}
    = 
    \|e^{-i\mu_2(x)}T_k\va \r(\frac{\cdot}{|k|})\|_{\scrF L^{\fy}}^{\frac{2}{p}-1}
    \\
    \lesssim &
    \|T_k\va \|_{\scrF L^{1}}^{\frac{2}{p}-1}
    \|e^{-i\mu_2(x)}\r(\frac{\cdot}{|k|})\|_{\scrF L^{\fy}}^{\frac{2}{p}-1}
    \lesssim 
    \|e^{-i\mu_2(x)}\r(\frac{\cdot}{|k|})\|_{\scrF L^{\fy}}^{\frac{2}{p}-1}.
\end{split}
\ee
From this, the only remaining thing is to verify 
\be
\|e^{-i\mu_2(x)}\r(\frac{\cdot}{|k|})\|_{\scrF L^{\fy}}\lesssim \lan k\ran^{\frac{-dt_2}{2}}.
\ee
Write
\be
\begin{split}
\|e^{-i\mu_2(x)}\r(\frac{\cdot}{|k|})\|_{\scrF L^{\fy}}
= &
|k|^d \|e^{-i\lan |k|x\ran^{(2+t_2)}}\rho\|_{\scrF L^{\fy}}
\\
\sim &
\lan k\ran^d \|e^{-i |k|^{(2+t_2)}\phi_k(x)}\rho\|_{\scrF L^{\fy}},
\end{split}
\ee
where $\phi_k(x)=\frac{\lan |k|x\ran^{(2+t_2)}}{|k|^{(2+t_2)}}$.
A direct calculation yields that
\be
|\partial^{\g}\phi_k(x)|\lesssim 1,\ \ \ \ \ |\det \text{Hess} \phi_k(x)|\gtrsim 1
\ee
uniformly for all $k$.
Using this and Lemma \ref{lm-osc}, we obtain the desired estimate
\be
\begin{split}
    \|e^{-i\mu_2(x)}\r(\frac{\cdot}{|k|})\|_{\scrF L^{\fy}}
    \sim &
    \lan k\ran^d \|e^{-i |k|^{(2+t_2)}\phi_k(x)}\rho\|_{\scrF L^{\fy}}
    \\
    \lesssim &
    \lan k\ran^d |k|^{\frac{-d(2+t_2)}{2}}
    \sim
    \lan k\ran^{\frac{-dt_2}{2}}.
\end{split}
\ee
\end{proof}


\bibliographystyle{abbrv}

\end{document}